\documentclass[12pt,leqno]{amsart} 

\usepackage{amsmath}
\usepackage{amssymb}
\usepackage{amsthm} 
\usepackage[pagebackref]{hyperref} 
\hypersetup{
	colorlinks=true, % make hyperlinks different color
	linkcolor=blue,  % color of \ref
	citecolor=blue, % color of \cite
}
\usepackage{esint} 
\usepackage{enumerate} 
\usepackage[shortlabels]{enumitem}
\usepackage{mathtools}
\usepackage{comment} 
\usepackage[english]{babel}
\usepackage[dvipsnames]{xcolor}

\usepackage[a4paper, margin={3cm}]{geometry}
\usepackage{ulem}

\numberwithin{equation}{section}

\newtheorem{lemma}{Lemma}[section]

\newtheorem{prop}[lemma]{Proposition}
\newtheorem{thm}[lemma]{Theorem}
\newtheorem{cor}[lemma]{Corollary}

\newtheorem{thmx}{Theorem}
 
\theoremstyle{definition}
\newtheorem{rmk}[lemma]{\bf Remark}

\newtheorem{defn}[lemma]{\bf Definition}

\newcommand{\N}{\mathbb{N}}
\newcommand{\Z}{\mathbb{Z}}

\newcommand{\R}{\mathbb{R}}
\newcommand{\C}{\mathbb{C}}

\newcommand{\ra}{\rightarrow}

\newcommand{\norm}[1]{\left \lVert #1 \right \rVert}
\newcommand{\abs}[1]{\left \lvert #1 \right \rvert}
\newcommand{\loc}{\mathrm{loc}}
\DeclareMathOperator{\divv}{\mathrm{div}}

\DeclareMathOperator{\Lip}{\mathrm{Lip}}
\renewcommand{\epsilon}{\varepsilon}
\newcommand{\T}{\mathbb{T}}

\newcommand{\eps}{\varepsilon}
\newcommand{\vp}{\varphi}

\newcommand{\al}{\alpha}
\newcommand{\be}{\beta}
\newcommand{\ga}{\gamma}
\newcommand{\de}{\delta}

\newcommand{\te}{\theta}
\newcommand{\la}{\lambda}
\newcommand{\La}{\Lambda}

\newcommand{\si}{\sigma}

\newcommand{\ol}{\overline}

\newcommand{\nid}{\noindent}

\newcommand{\iny}{\infty}
\newcommand{\del}{ \partial}
\newcommand{\su}{\subset}
\newcommand{\LP}{\Delta}
\newcommand{\gr}{\nabla}

\newcommand{\cyl}{\mathbb{T}^d \times \R^m}

\newcommand{\innp}[1]{\left< #1 \right>}
\newcommand{\set}[1]{\left\{#1\right\}}
\newcommand{\brac}[1]{\left[#1\right]}
\newcommand{\pr}[1]{\left( #1 \right) }

\newcommand{\disp}{\displaystyle}
\DeclareMathOperator{\di}{div}

\title[Decay at infinity for solutions to Schr\"{o}dinger equations]{Optimal rates of decay at infinity for solutions to Schr\"{o}dinger equations}
\author{Blair Davey}
\address[B. Davey]{Department of Mathematical Sciences, Montana State University, Bozeman, MT 59717}
\email{{blairdavey@montana.edu}}

\author{Cole Jeznach}
\address[C. Jeznach]{Departament de Matem\`{a}tiques, Universitat Aut\`{o}noma de Barcelona, Bellaterra, 08193, Spain}
\email{colejeffrey.jeznach@uab.cat} 

\thanks{B. Davey was partially supported by the NSF CAREER DMS-2236491.
C.J. was supported by the European Research Council (ERC) under the European Union's Horizon 2020 research and innovation programme (grant agreement 101018680), and by NSF-DMS-2503326}
\date{\today}
\addtocontents{toc}{\protect\setcounter{tocdepth}{1}}

\begin{document}

\begin{abstract}
We prove rates of decay at infinity for solutions to variable-coefficient Schr\"{o}dinger equations of the form $-\divv(A \nabla u) + W \cdot \nabla u + V u = \la u$ in cylinders, $\T^d \times \R^m$.
We assume that $W$ and $V$ are bounded and that $\la \in \C$. 
Our rates depend on the decay of $\abs{\nabla A}$ at infinity.
In particular, we prove a range of quantitative unique continuation-type results at infinity when $\abs{\nabla A(\te, x)} \le C (1 + \abs{x})^{-\tau}$ for $\tau \in [0,1]$. 
By adapting the methods in \cite{KLP25}, we construct explicit solutions to demonstrate the sharpness of our estimates for each such $\tau$.
\end{abstract}

\maketitle
\tableofcontents

\section{Introduction}

A modern-day version of Landis' conjecture states that if $V \in L^\infty(\R^n)$ is real-valued, and $u$ is a solution of
\begin{align}
\label{eqn:landis}
    -\Delta u + Vu = 0  \, \,  \text{ in }   \,  \R^n, 
\end{align}
then $\abs{u} \lesssim \exp \left ( -c \abs{x}^{1+} \right )$ implies that $u \equiv 0$. 
When $n = 2$, Landis' conjecture was proven by Logunov, Malinnikova, Nadirashvili, and Nazarov in \cite{LMNN25}.
However, Landis' conjecture currently remains open for $n \ge 3$. 
That $V$ is real-valued in the conjecture is crucial; in \cite{M92}, Meshkov constructed an example of a bounded potential $V : \R^2 \to \C$ and a non-trivial solution $u : \R^2 \to \C$ of the equation \eqref{eqn:landis} when $n = 2$ with the property that $\abs{u(x)} \lesssim \exp \left(- c \abs{x}^{4/3} \right)$. 
In the same paper, Meshkov used Carleman estimate techniques to show that for \textit{complex}-valued bounded potentials $V$, the power $4/3$ corresponds to the optimal rate of decay at infinity.
Specifically, if $u$ solves \eqref{eqn:landis} for $V$ bounded, and $ \abs{u} \lesssim  \exp \left( -C \abs{x}^{4/3} \right)$ for each $C > 1$, then $u \equiv 0$. 
In \cite{BK05}, Bourgain and Kenig established a quantitative version of Meshkov's estimate:
If $u$ is a bounded, normalized solution of \eqref{eqn:landis}, then for $R \gg 1$, whenever $ \abs{x_0}  = R$, it holds that 
$$\norm{u}_{L^2\pr{B_1(x_0)}} \gtrsim \exp\pr{- C R^{4/ 3} \log R}.$$
The work of Filonov and Krymskii \cite{FK24} showed that if one replaces $\R^n$ by the cylinder $\mathbb{T}^d \times \R$ for $d \ge 3$, then the analogous version of Landis' conjecture fails. 
That is, there exists a real-valued function $u : \T^d \times \R \to \R$ that satisfies
\begin{align*}
	\abs{\Delta u(\te, x)} \le C \abs{u(\te, x)}, \, \, \abs{u(\te, x)} \lesssim \exp \left(  -c \abs{x}^{4/3}   \right), \, (\te, x) \in \T^d \times \R.
\end{align*}
Here, $\T^d$ denotes the $d$-dimensional torus, $\T^d := \R^d / \pr{2 \pi \Z}^d$.

For generalized Schr\"odinger equations of the form
\begin{align}
\label{eqn:Landis+}
	-\divv(A \nabla u) + W \cdot \nabla u + V u = \la u \, \,  \text{  in } \, \R^n,
\end{align}
the optimal rate of decay at infinity for solutions is well-understood whenever $\abs{\gr A}$ decays fast enough.
If $\abs{\gr A(x)} \lesssim \abs{x}^{-1-\eps}$ for some $\eps > 0$, $W \in L^\iny(\R^n; \C^n)$, $V \in L^\iny(\R^n; \C)$, $\la \in \C$, and $u$ is a bounded, normalized solution to \eqref{eqn:Landis+}, then for $R \gg 1$, whenever $ \abs{x_0}  = R$, it holds that 
$$\norm{u}_{L^2\pr{B_1(x_0)}} \gtrsim \exp\pr{- C R^{2} \log R}.$$
If $W \equiv 0$, then the power of $2$ may be replaced with $4/3$.
And if $W, V \equiv 0$, then the power of $2$ may be replaced with $1$.
Versions of this result, examples that prove sharpness, and various generalizations are given in \cite{Dav14}, \cite{LW14}, and \cite{Dav26}, for example.
The decay condition on $\abs{\gr A}$ appeared previously in \cite{Tu10} where Nguyen explored unique continuation problems for parabolic equations.
As we discuss below, the above decay condition on $\abs{\gr A}$ is sufficient, but also nearly necessary.

In the more recent work \cite{KLP25}, Krymskii, Logunov, and Pagano explored Landis-type results for variable-coefficient second-order equations in $3$-dimensional cylinders, $\T^2 \times \R$. 
They constructed eigenfunctions of a uniformly elliptic equation with smooth coefficients that exhibit \textit{doubly} exponential decay at infinity: 
For each $\la > 0$, there exists a matrix $A(\te, \vp, x) \in C^\infty( \T^2 \times \R)$ and a non-trivial eigenfunction $u \not \equiv 0$ that solves
\begin{align*}
	-\divv(A \nabla u) & = \lambda u \, \, \text{ in } \, \T^2 \times \R,
\end{align*}
and satisfies $\abs{u(\te, \vp, x)} \le C \exp \set{- c \exp \pr{c \abs{x}}}$ for some constants $c, C >0$. 
Moreover, the coefficients $A$ are such that $\abs{\nabla A(\te, \vp, x)} \le C$ for all $(\te, \vp, x) \in \T^2 \times \R$. 
In comparison with the case of the Laplacian (or when $\abs{\nabla A(x)} \lesssim \abs{x}^{-1 - \epsilon}$), these examples are, in a suitable sense, counterexamples to quantitative unique continuation (at infinity) for elliptic operators with uniformly Lipschitz coefficients. 
This result is in contrast with local unique continuation properties, where the Lipschitz nature of $A$ is sufficient, and also very close to necessary.

Given a partial differential operator $\mathcal{L}$ on a connected domain $\Omega \subset \R^n$, recall that $\mathcal{L}$ satisfies the unique continuation property (UCP) if whenever $u$ solves $\mathcal{L} u = 0$ in $\Omega$,  then $u|_V \equiv 0$ for some non-trivial open subset $V \subset \Omega$ implies that $u \equiv 0$. 
Furthermore, $\mathcal{L}$ satisfies the strong unique continuation property (SUCP) if whenever a solution $u$ of $\mathcal{L} u = 0$ in $\Omega$ vanishes to infinite order at $x_0 \in \Omega$ in a suitable sense, then also $u \equiv 0$. 
Linear elliptic equations with Lipschitz continuous coefficients have the SUCP; see \cite{GL86, GL87}, for example.
However, when the coefficients are only assumed to be H\"older continuous, there are various counterexamples to unique continuation due to Plis \cite{Pli63}, Miller \cite{Mil74}, Mandache \cite{Man98}, and Filonov \cite{Fil01}.
The work of the second-named author in \cite{Jez25} combined with the constructions of \cite{Man98} shows that a log-Lipschitz condition on the coefficient matrix $A$ is essentially the sharp condition that guarantees the SUCP. 
Thus, in the case of local unique continuation properties, optimal conditions on the smoothness of $A$ are known for second-order linear elliptic equations. 
In this article, we pursue the same goal for unique continuation (or Landis-type) results at infinity, and so it is useful to give such results a name.

We say that an elliptic operator $\mathcal{L}$ satisfies the \textit{unique continuation property at infinity} in $\T^d \times \R^m$ if there is a (polynomial) power $p >0$ so that any solution $u$ of the equation
\begin{align*}
    \mathcal{L} u = 0 \, \,  \text{ in }   \,   \T^d \times \R^m
\end{align*}
that satisfies 
\begin{align*}
    \abs{u(\te, x)} \le \exp\{-C \abs{x}^q\} \,\, \text{ for all } \, (\te, x) \in \T^d \times \R^m,
\end{align*}
for some $q > p$, is necessarily the trivial solution, $u \equiv 0$. 
Notice that Landis-type results can be formulated with this language.
The classical Schr\"{o}dinger operator $-\Delta + V$ with bounded $V$ satisfies the unique continuation property at infinity in $\R^n$ with polynomial power $p = 4/3$, \cite{M92}.
However, the eigenvalue operator $-\divv(A \nabla \, \cdot \,) - \lambda $ constructed by \cite{KLP25} does not satisfy the unique continuation property at infinity in $\T^2 \times \R$ for any polynomial power despite the fact that $\abs{\nabla A }\lesssim 1$. 
On the other hand, the results in \cite{LW14} and \cite{Dav26} establish unique continuation results at infinity for generalized Schr\"odinger operators with the assumption that $\abs{\gr A(x)} \lesssim \abs{x}^{-1-\eps}$ for some $\eps > 0$. 
In this article, we examine the role of the smoothness of the coefficients $A$.
Specifically, we fill in the gap between the conditions $\abs{\nabla A(x)} \lesssim 1$ and $\abs{\nabla A(x)} \lesssim \abs{x}^{-1-\epsilon}$ and show the threshold for unique continuation at infinity. 

Our main contribution is to show that the ``right'' threshold condition is that $\abs{\nabla A(x)} \lesssim \abs{x}^{-1}$ at infinity. 
This sort of gradient-decay has appeared previously in backwards-in-time uniqueness problems for parabolic equations. 
More specifically, suppose that $u(x,t)$ is a solution with bounded growth to
\begin{align*}
    \partial_t u - \divv(A \nabla u) + W \cdot \nabla u + V u = 0  \,\, \text{ in } \R^n \times (0, 1],
\end{align*}
where $W$ and $V$ are bounded, and $A(x,t)$ is uniformly elliptic and bounded. 
In \cite{LO74}, Landis and Oleinik conjectured that, under suitable conditions on $A$, the super-Gaussian decay
\begin{align*}
    \abs{u(x, 1)} \le C \exp \{ - \abs{x}^{2 + \epsilon}\} \,\, \text{ for all  } x \in \R^n  
\end{align*}
for some $\epsilon > 0$, implies $u$ is the trivial solution: $u(x,t) \equiv 0$. This conjecture was proved in the affirmative for $A = I$ by Escauriaza, Kenig, Ponce, and Vega \cite{EKPV06}, and later in Nguyen \cite{Tu10} for $A(x,t)$ that satisfies 
\begin{align*}
    \abs{\nabla A(x,t)} \lesssim \frac{1}{1 + \abs{x}^{1 + \delta}}, \; \; \;\;
    \abs{\partial_t A(x,t)} \lesssim 1,  
    \; \; \;\;
    \abs{A(x,t) - A(x,s)} \lesssim \frac{\abs{t-s}^{1/2}}{1 + \abs{x}}.
\end{align*}
More recently, Wu and Zhang \cite{WZ16P, WZ16} proved the conjecture for coefficient matrices that satisfy $\abs{\nabla A(x,t)} + \abs{\partial_t A(x,t)} \lesssim 1$ along with the spatial decay
\begin{align*}
    \abs{\nabla A(x,t)} \le E\abs{x}^{-1}
\end{align*}
for some $ E > 0$ sufficiently small. 
We point out that the results of \cite{WZ16P, WZ16} imply our Theorem \ref{ThmA} part (c) when $q = 2$, while all other cases of Theorem \ref{ThmA} do not follow from the parabolic results.
Moreover, our techniques differ from those used in \cite{EKPV06}, \cite{Tu10}, \cite{WZ16P}, and \cite{WZ16}.

\vspace{0.5em}

\subsection{Theorem statements}

For $d \in \Z_{\ge 0}$, $m \in \N$, we consider eigenfunctions $u(\te, x)$ for $(\te, x) \in \T^d \times \R^m$ associated to generalized Schr\"odinger operators of the form $-\divv(A \nabla ) + W \cdot \nabla  + V $.
We assume that $A$ is a uniformly elliptic, symmetric, bounded matrix that is locally Lipschitz and has the block structure
\begin{align}
	\label{eqn:A_struct_intro}
	A(\te, x) = \begin{pmatrix}
		A^{(1)}(\te, x) & 0 \\
		0 & A^{(2)}(\te, x)
	\end{pmatrix},
\end{align}
where $A^{(1)}$ is a $d \times d$ matrix, and $A^{(2)}$ is an $m \times m$ matrix.
\footnote{While this structural assumption may seem arbitrary, the examples of \cite{KLP25} fall into this class. 
Notice that when $d = 0$, we are reduced to $\R^m$ and we have no restriction on the structure of the matrix in $\R^m$.} 
We also assume that $W$ and $V$ are bounded. 
Under various assumptions on the decay of $\abs{\nabla A(\te, x)}$ as $\abs{x} \ra \infty$, we obtain lower bounds on the possible rate of decay of eigenfunctions $u(\te, x)$ in the $x$-directions. 
As an application of such results, we obtain optimal decay rates of eigenfunctions when $\abs{\nabla A(\te, x)}$ decays polynomially in $\abs{x}$: 
\begin{align}
\label{eqn:smooth_intro}
	\abs{\nabla A(\te, x)} & \le \frac{L}{ (1 + \abs{x})^\tau    }  \, \,\,\, \text{ for a.e. } (\te, x) \in \T^d \times \R^m,
\end{align}
for $\tau \in [0,1]$. 
With this set up, our main results are described by the following statement. 
(For more precise statements regarding the meaning of normalized and bounded growth at infinity, see Theorems \ref{thm:tau_betw} and \ref{thm:tau=1}.)

\begin{thmx}[Rate of decay at infinity when coefficients have polynomial decay]
\label{ThmA}
    Let $A$ be a uniformly elliptic, symmetric, bounded matrix that has the block structure in \eqref{eqn:A_struct_intro} and satisfies \eqref{eqn:smooth_intro}.
    Let $W \in L^\iny(\T^d \times \R^m; \C^{d+m})$, $V \in L^\iny(\T^d \times \R^m; \C)$, $\la \in \C$, and define 
    $$q := \begin{cases}
           2 & W \not\equiv 0 \\  \frac 4 3 & W \equiv 0, V \not\equiv 0  \\
           1 & otherwise
    \end{cases}.$$ 
    Let $u$ be a normalized solution of 
    \begin{align}
    \label{eqn:eigensoln}
	   -\divv(A \nabla u) + W \cdot \nabla u + V u = \la u \, \text{  in } \,\, \T^d \times \R^m
    \end{align}
    with bounded growth at infinity.
    \begin{itemize}
        \item[(a)] If $\tau \in [0, 1)$, then there exists $R_0 \gg 1$ and $C > 0$ so that whenever $\abs{x_0} \ge R_0$, it holds that
        \begin{equation*}
    	\norm{u}_{L^2(\T^d \times B_1(x_0))} \ge \exp\set{- \exp\pr{C \abs{x_0}^{1-\tau}}}.
	   \end{equation*}	
       \item[(b)] If $\tau = 1$ and $L \ge L_0(q)$, then there exists $p(L) > q$, $R_0 \gg 1$ and $C > 0$ so that whenever $\abs{x_0} \ge R_0$, it holds that
        \begin{equation*}
    	\norm{u}_{L^2(\T^d \times B_1(x_0))} \ge \exp\set{- C \abs{x_0}^{p}}.
	   \end{equation*}	
       \item[(c)] If $\tau = 1$ and $L \le L_0(q)$, then there exists $R_0 \gg 1$ and $C > 0$ so that whenever $\abs{x_0} \ge R_0$, it holds that
        \begin{equation*}
    	\norm{u}_{L^2(\T^d \times B_1(x_0))} \ge \exp\set{- C \abs{x_0}^{q} \pr{\log \abs{x_0}}^C}.
	   \end{equation*}	
    \end{itemize}
\end{thmx}  

\begin{rmk}
    In parts (b) and (c) above, $L_0(q) = c_0(q - 1)$.
    In particular, if $q = 1$, then (b) is always applicable and the ``small $L$" regime doesn't apply.
\end{rmk}

The above Theorem effectively says that when $\abs{\nabla A(\te, x)} \le L \abs{x}^{-1}$ for $\abs{x}$ large, then we have unique continuation at infinity in $\T^d \times \R^m$ with some polynomial power $p \ge q$. 
In fact, when the constant $L$ from \eqref{eqn:smooth_intro} is small enough, we recover the same power $p$ as in the case for the Laplacian, depending on whether a bounded complex-valued drift ($p = 2$) or bounded potential ($p = 4/3$) appears in the equation. 
On the other hand, the power $p=p(L)$ could be large when $L$ is large. 
In addition, if one only has $\abs{\nabla A(\te, x)} \lesssim \abs{x}^{-\tau}$ for some fixed $0 \le  \tau < 1$ and for all $\abs{x}$ large, then we cannot guarantee unique continuation at infinity for some polynomial power, but instead have a super-exponential lower bound. 
Our second main result shows by example that all of these behaviors are possible.

\begin{thmx}[Sharp examples]
\label{ThmB}
For each $\epsilon \in (0, 1]$, there exists a bounded, symmetric, uniformly elliptic matrix $A$ defined on $\T^2 \times \R$ that satisfies \eqref{eqn:A_struct_intro} and \eqref{eqn:smooth_intro} with $\tau = 1 - \eps$, a bounded, real-valued function $V$ with compact support, and a real-valued solution $u$ of the equation
\begin{align}
\label{eqn:counterexample_intro}
    \divv(A \nabla u) = V u  \, \text{ in } \,\, \T^2 \times \R
\end{align}
with the property that for every $(\te, \vp, x)  \in \T^2 \times \R$,
\begin{align*}
    \abs{u(\te, \vp, x)} \le c \exp \set{ - \exp (-C \abs{x}^\epsilon)}.
\end{align*}

Similarly, for each $\epsilon \in (0, \tfrac 1 2]$, there there exists a bounded, symmetric, uniformly elliptic matrix $A$ defined on $\T^2 \times \R$ that satisfies \eqref{eqn:A_struct_intro} and \eqref{eqn:smooth_intro} with $\tau = 1$ and $L = C \eps^{-1}$, a bounded, real-valued function $V$ with compact support, and a real-valued solution $u$ of the equation \eqref{eqn:counterexample_intro}
with the property that for every $(\te, \vp, x)  \in \T^2 \times \R$,
\begin{align*}
    \abs{u(\te, \vp, x)} \le c \exp \set{ - C_\eps \abs{x}^{1/\epsilon}}.
\end{align*}
\end{thmx}

The combination of our two main results shows that the sharp condition on the smoothness of $A$ that guarantees unique continuation at infinity (with a polynomial power) is \textit{linear} decay of the gradient $\abs{\nabla A (\te, x)}$ at infinity:
\begin{align}
\label{eqn:linear_intro}
    \abs{\nabla A(\te, x)} \le \frac{L}{1 + \abs{x}} \,\,\,\, \text{ for a.e. } \, (\te, x) \in \T^d \times \R^m.
\end{align}
We may refer to this kind of decay as ``Lipschitz decay at infinity".
Of course, by taking $d = 0$, we also obtain unique continuation at infinity in non-cylindrical domains (and without structural assumptions on the matrix $A$), and thus our work improves upon the previous known sufficient condition, i.e., $\abs{\nabla A(x)} \lesssim \abs{x}^{-1-\epsilon}$. 
Notice that the condition \eqref{eqn:linear_intro} does not force $A$ to have a limit at infinity.

Finally, when we have no lower order terms, i.e., $W, V \equiv 0$ in \eqref{eqn:eigensoln}, then we can improve the polynomial power of $q$ in Theorem \ref{ThmA} case (b), provided that $\abs{\nabla A(\te, x)}$ decays slightly faster in $x$:

\begin{thmx}[Rate of decay at infinity when coefficients have additional decay]
\label{ThmC}
    Let $A$ be a uniformly elliptic, symmetric, bounded matrix that has the block structure in \eqref{eqn:A_struct_intro} and satisfies
    \begin{align*}
        \abs{\nabla A(\te, x)} \le \frac{L}{1 + \abs{x} \log(\abs{x} + 1)} \,\,\,\, \text{ for a.e. } \, (\te, x) \in \T^d \times \R^m.
    \end{align*}
    Let $\la \in \C$ and let $u$ be a normalized solution to 
    \begin{align*}
        -\divv(A \nabla u) = \lambda u  \, \text{ in } \, \, \T^d \times \R^m,
    \end{align*}
    with bounded growth at infinity. Then there exists $R_0 \gg 1$ and $C >0$ so that whenever $\abs{x_0} \ge R_0$, it holds that 
    \begin{align*}
        \norm{u}_{L^2(  \T^d \times B_1(x_0) )} \ge \exp \set{   - C\pr{\sqrt{\la} +1} \abs{x_0} \exp{ \brac{\pr{\log\abs{x_0}}^{3/4}  }} } .
    \end{align*}
\end{thmx}

We don't know if the above theorem is still true when we assume \eqref{eqn:smooth_intro} with $\tau  = 1$ and $L$ sufficiently small. 
When we attempted to prove the result in this more general setting, the exponent $q$ stayed ``stuck" above $1$.
For a discussion of this subtlety and how it relates to our techniques, see Remark \ref{rmk:log_prob}.

\subsection{Discussion of techniques, and outline of the paper}

To discuss the techniques that we use to prove our main results, it is useful to draw a parallel to the case of \textit{local} unique continuation properties for elliptic PDEs and the tools used there.

\textit{Frequency functions} are one of the main tools that are used to study unique continuation properties of elliptic PDEs. 
If $u$ is a solution of $-\divv(A \nabla u) = 0$, then Almgren's frequency function takes the form 
\begin{align*}
    N(r) = \frac{r \int_{E_r} A \nabla u \cdot \nabla u \; dx   }{  \int_{\partial E_r} u^2 \mu \; d\sigma },
\end{align*}
where $E_r$ is an ellipse centered at the origin of diameter $\simeq r$, and $\mu \simeq 1$ is some factor depending on $A$. 
When $A$ is Lipschitz, then $N(r)$ is almost monotone \cite{GL86}: 
\begin{align*}
    \dfrac{d}{dr} N(r) \ge -CN(r), 
\end{align*}
where $C$ depends on the dimension, $\norm{\nabla A}_\infty$, and the ellipticity of $A$. 
The almost-monotonicity of $N(r)$ has far-reaching consequences to quantitative unique continuation results for solutions (e.g., propagation of smallness, nodal set estimates, boundary unique continuation). 
One of the fundamental applications of the almost-monotonicity of $N(r)$ is a three-ball inequality: 
If $u$ solves $-\divv(A \nabla u )= 0$ in $B_1$, then for $0 < r_1 < r< r_2 < 1$, 
\begin{align}\label{eqn:3ball}
    \fint_{B_r} u^2 \le C \left( \fint_{B_{r_1}}u^2  \right)^{\alpha} \left( \fint_{B_{r_2}} u^2 \right)^{1-\alpha}
\end{align}
where $C > 0$ and $\alpha \in (0,1)$ depends only on $r_1, r, r_2$, and $A$. 

Inequalities of the type \eqref{eqn:3ball} are especially useful when proving unique continuation at infinity results.
For example, if $u$ is a bounded solution, then the inequality above gives lower bounds on $u$ on the smaller ball $B_{r_1}$ in terms of the lower bound on $u$ in the larger (mid-size) ball $B_r$. 
This mechanism is precisely what we use to pass quantitative information about our solution near the origin to information about our solution on cylinders that are far from the origin. 
One of the main difficulties is that the constants $C > 1$ and $\alpha \in(0,1)$ come with errors that depend on $\norm{\nabla A}_\infty$.
Therefore, to obtain sharp unique continuation at infinity results, one must obtain three-cylinder inequalities that minimize these errors in a very precise way.
The lower order terms introduce additional challenges.

Our starting point is to prove three-cylinder inequalities for solutions of \eqref{eqn:eigensoln}, see Proposition \ref{prop:3cylinder} and Corollary \ref{cor:3cylinder}.
As indicated above, these inequalities follow from monotonicity properties of appropriate frequency functions, see Corollary \ref{cor:monotone_cyl}. 
The frequency function that we use (see \eqref{eqn:N_def}) generalizes those that appear in \cite{KN98}, \cite{Kuk00} to the cylindrical setting. 
Frequency functions of this type have garnered more attention in recent years for their utility in proving sharp quantitative unique continuation results, both locally and at infinity (see for example, \cite{Zhu16}, \cite{BG16}, \cite{Dav26} and \cite{TWZ26}). 
While such three-cylinder inequalities can also be proved by other means (e.g., Carleman estimates), the benefit of the frequency function approach is the precision with which we control the errors in the monotonicity, and how they depend on $\norm{\nabla A}_\infty$. 

With an appropriate three-cylinder inequality in hand, we next prove a very general iteration scheme that allows us to establish a variety of Landis-type results. 
In particular, this scheme shows precisely how the rate of decay of $\abs{\nabla A(\te, x)}$ affects the rate of decay at infinity of solutions, see Proposition \ref{prop:decay_cyl_p_eig}. 
We apply this very general iteration scheme to obtain different results when $\abs{\nabla A(\te, x)}$ has polynomial decay in $x$, i.e., $\abs{\nabla A(\te, x)} \lesssim (1 + \abs{x})^{-\tau}$ for some $\tau \in [0,1]$.
This gives Theorem \ref{ThmA}; see Theorems \ref{thm:tau_betw} and \ref{thm:tau=1} for the precise statements.
We also apply the iteration scheme when $\abs{\nabla A(\te, x)}$ has additional decay in $x$, i.e., $\abs{\nabla A(\te, x)} \lesssim \brac{1 + \abs{x} \log\pr{1 + \abs{x}} }^{-1}$.
This gives Theorem \ref{ThmC}, precisely stated in Theorem \ref{thm:tau=1+}.

In our second main theorem, we adapt the techniques of \cite{KLP25} to construct global solutions of Schr\"{o}dinger-type equations in the cylinder $\T^2 \times \R$ that exhibit fast decay at infinity. 
The main tool in this construction is the ``building block", Lemma \ref{lemma:transformation_block}.
Essentially, this lemma shows that one can build an elliptic operator $-\divv (A \nabla \, \cdot)$  in $\T^2 \times [0, T]$ that ``transforms'' the harmonic function $u(\te, \vp, x) = \cos(k\te) e^{-kx}$ into the harmonic function (with faster decay) $v(\te, \vp, x) = \cos(2k\vp)e^{-2kx}$. 
By ``transforms'', we mean that there is a solution $w(\te, \vp, x)$ of $-\divv(A \nabla w) = 0$ in $\T^2 \times [0,T]$ that agrees with $u$ near $x = 0$ and agrees with $v$ near $x = T$. 
The cost of this transformation is that $A$ must oscillate at a certain rate: $\abs{ \nabla A } \le C T^{-1}$. 
This construction originated in \cite{Mil74}, but has since been used in \cite{Man98} and \cite{KLP25}.
For completeness, we provide a complete proof here. 
Using this ``building block", we construct elliptic operators $A$ and solutions $u$ in the whole $\T^2 \times \R$ that satisfy the conclusion of Theorem \ref{ThmB}.
For a precise statement, see Theorem \ref{thm:eg_1}.

The rest of the paper is organized as follows. 
In Section \ref{sec:montonicity}, we introduce our frequency function $N(r)$ and deduce almost-monotonicity results as described in Corollary \ref{cor:monotone_cyl}. 
With these monotonicity results, we prove three-cylinder inequalities for solutions and eigenfunctions, see Proposition \ref{prop:3cylinder} and Corollary \ref{cor:3cylinder}. 
In Section \ref{sec:applications}, we apply the three-cylinder inequalities to obtain decay rates at infinity of solutions in the case that $\abs{\nabla A(\te, x)} \le f(\abs{x})$ as $\abs{x} \ra \infty$, see Proposition \ref{prop:decay_cyl_p_eig}.
We then use this proposition to deduce lower bounds on solutions in the case that $\abs{\nabla A(\te, x)}$ decays polynomially in $x$ (Theorems \ref{thm:tau_betw} and \ref{thm:tau=1}) and slightly faster than polynomially (Theorem \ref{thm:tau=1+}). 
Finally, in Section \ref{sec:examples}, we construct the sharp examples that are described in Theorem \ref{ThmB} and made precise in Theorem \ref{thm:eg_1}.

\subsection{Notation}

For $d \in \Z_{\ge 0}$ and $m \in \N$, we work in $(d+m)$-dimensional cylinders of the form $\T^d \times \R^m$, where $\T^d := \R^d / \pr{2\pi \Z}^d$ denotes a $d$-dimensional torus.
For a general point, we write $(\te, x) \in \T^d \times \R^m$, where $\te = (\te_1, \ldots, \te_d) \in \T^d$ and $x = (x_1, \ldots, x_m) \in \R^m$.
Given a function $u = u(\te, x)$, we write $\gr u(\te, x)$ to denote the full gradient, the $(d+m)$-vector field.
We may decompose the gradient as $\gr u = \pr{\gr_\te u, \gr_x u}$.
Given a point $x_0 \in \R^m$, $\mathbb{K}_r(x_0) \coloneqq \T^d \times B_r(x_0)$ denotes the cylinder of radius $r > 0$ centered at $x_0$.
If $x_0 = 0$ or is understood from the context, we may simply write $\mathbb{K}_r \coloneqq \T^d \times B_r$.

\section{Monotonicity formulas and three-cylinder inequalities}
\label{sec:montonicity}

In this section, we focus on solutions $u$ to uniformly elliptic equations of the type
\begin{align}\label{eqn:soln_cyl}
	-\divv(A \nabla u) + W \cdot \nabla u + V u & = 0
\end{align}
in $\cyl$, where $V \in L^\infty(\T^d \times \R^m ; \C)$ and $W \in L^\infty(\T^d \times \R^m ; \C^{d+m})$ satisfy
\begin{align}\label{eqn:pot_bounds}
	\norm{W}_{L^\infty} \le K, \, \, \norm{V}_{L^\infty} \le M.
\end{align}
By solution, we mean \textit{weak} solution, i.e., that $u \in H^1_{\loc}(\T^d \times \R^m)$ and the equation \eqref{eqn:soln_cyl}
 holds in the weak sense:
\begin{align*}
    \int_{\T^d \times \R^m} \brac{A \nabla u \cdot \nabla v + (W \cdot \nabla u) v  + Vuv }\; d\te dx = 0
    \quad \text{ for all } v \in C_c^\infty(\T^d \times \R^m). 
\end{align*}
We always assume that $A$ is at least Lipschitz continuous, and thus $A$ is differentiable almost everywhere. 
By the standard elliptic regularity theory (see for example \cite[Theorem 1, Section 6.3.1]{Ev10} that only uses that $A$ is Lipschitz), we then have that $u \in H^2_{\loc}(\T^d \times \R^m)$.
Therefore, the equation \eqref{eqn:soln_cyl} also holds in the strong sense, i.e., the equation \eqref{eqn:soln_cyl} holds pointwise almost everywhere.

\subsection{Matrix assumptions}

For $d \in \Z_{\ge 0}$ and $m \in \N$, we assume that the matrix $A \in \Lip(\T^d \times \R^m; \mathbb{M}^{d+m})$ is bounded and uniformly elliptic, symmetric, and has the block structure
\begin{align}
\label{eqn:A_struct}
	A(\te, x) = \begin{pmatrix}
		A^{(1)}(\te, x) & 0 \\
		0 & A^{(2)}(\te, x)
	\end{pmatrix}.
\end{align}
We write $A = (a_{ij})_{i, j = 1}^{d+m}$, $A^{(1)} = (a^{(1)}_{ij})_{i, j = 1}^{d}$, and $A^{(2)} = (a^{(2)}_{ij})_{i, j = 1}^{m}$ to denote the entries of $A$, $A^{(1)}$, and $A^{(2)}$, respectively.
We assume that $A^{(1)}$ is a symmetric, bounded, uniformly elliptic $d \times d$ matrix, and that $A^{(2)}$ is a symmetric, bounded, uniformly elliptic $m \times m$ matrix. 
The symmetry, uniform boundedness, and ellipticity of the whole matrix $A$ are such that for all $(\te, x) \in \cyl$,
\begin{align}
    & a_{ij}(\te, x) = a_{ji}(\te, x) 
    \label{eqn:A_symm} \\
    & \abs{a_{ij}(\te, x) \xi_i \zeta_j} \le \La \abs{\xi} \abs{\zeta}  \text{ for all } \xi, \zeta \in \R^{d + m} 
    \label{eqn:A_bound} \\
    & a_{ij}(\te, x) \xi_i \xi_j \ge \La^{-1} \abs{\xi}^2 \text{ for all } \xi \in \R^{d + m}.
    \label{eqn:A_ellip}
\end{align}
For local results, we assume that such conditions hold on some $\mathbb{K}_R \su \T^d \times \R^m$.

To make some of our statements more succinct, we introduce the following matrix classes.

\begin{defn}[Matrix classes]
\label{defn:matClass}
    For $d \in \Z_{\ge 0}$ and $m \in \N$, we say that the matrix $A \in \Lip(\T^d \times \R^m; \mathbb{M}^{d+m})$ belongs to the matrix class $\mathcal{M}(d, m, \La)$ if for all $(\te, x) \in \T^d \times \R^m$, $A$ is symmetric \eqref{eqn:A_symm}, bounded \eqref{eqn:A_bound}, elliptic \eqref{eqn:A_ellip}, and has the block structure in \eqref{eqn:A_struct} where $A^{(1)}$ is $d \times d$ and $A^{(2)}$ is $m \times m$.
    Further, we say that $A \in \mathcal{M}(d, m, \La)$ belongs to the matrix class $\mathcal{M}_0(d, m, \La)$ if $A^{(2)} = I_m$.
    If $A \in \Lip(\mathbb{K}_R; \mathbb{M}^{d+m})$ satisfies \eqref{eqn:A_struct} -- \eqref{eqn:A_ellip} for all $(\te, x) \in \mathbb{K}_R$, then we say that $A \in \mathcal{M}(\mathbb{K}_R; d, m, \La)$.
    If $A^{(2)} = I_m$ as well, then $A \in \mathcal{M}_0(\mathbb{K}_R; d, m, \La)$.
\end{defn}

It is worth mentioning that although this structural assumption may seem somewhat arbitrary, the examples in \cite{KLP25} belong to this class of matrices. 
Moreover, given that we allow $d = 0$, our results also apply to Euclidean space, $\R^m$.
If $d \ne 0$, we may abuse notation and identify $x \in \R^m$ with $(0,x) \in \R^{d+m}$, only doing so when it is clear from context with the inner product notation or matrix multiplication.
For example, we notice that the structure of $A$ as in \eqref{eqn:A_struct} is such that 
\begin{align}
\label{eqn:struct_c1}
	A(\te, x)x = A^{(2)}(\te, x) x, \qquad 
    A(\te, x) x \cdot x = A^{(2)}(\te, x)x \cdot x.
\end{align}

\subsection{Estimates for the coefficient matrix}

In the forthcoming monotonicity results, we require certain estimates on the coefficient matrix.
For $A \in \mathcal{M}(d, m, \La)$, we introduce the following notation.
First, we have bounds on $\abs{\nabla A}$: 
for each $ r> 0$, define  
\begin{align}
\label{eqn:A_lip}
	\eta(r) \coloneqq \norm{\nabla  A(\te, x)}_{L^\infty(\mathbb{K}_r)} < \infty.
\end{align}
Then let $\phi$ denote the primitive of $\eta$, 
\begin{equation}
\label{eqn:phi_def}   
    \phi(r) \coloneqq \int_0^r \eta(s) \, ds = \int_0^r \norm{\nabla  A(\te, x)}_{L^\infty(\mathbb{K}_s)}\, ds .
\end{equation}
In the end, our estimates depend quantitatively only on the value $\eta(r)$ as defined in \eqref{eqn:A_lip}.
The conformal factor $\mu(\te, x) \simeq_{\Lambda } 1$ is defined as
\begin{align}
\label{eqn:conf}
	\mu(\te, x) \coloneqq A(\te, x) \frac{x}{\abs{x}} \cdot \frac{x}{\abs{x}} = A^{(2)}(\te, x) \frac{x}{\abs{x}} \cdot \frac{x}{\abs{x}}, 
\end{align}
where the second equality follows from the observation \eqref{eqn:struct_c1}.
Finally, we also use the following vector field
\begin{align}
\label{eqn:vec_field}
	Z(\te, x) \coloneqq \frac{ A(\te, x) x }{\mu(\te, x)}.
\end{align}
The next lemma provides estimates on $Z(\te, x)$ and $\mu(\te, x)$ that we use below to establish monotonicity. 
We provide a brief proof, since similar computations can be found in \cite{Dav26}, for example. 

\begin{lemma}[Matrix estimates]
\label{lem:conf_est}
    For $A \in \mathcal{M}(d, m, \La)$, let $\mu(\te, x)$ and $Z(\te, x)$ be defined by \eqref{eqn:conf} and \eqref{eqn:vec_field}, respectively. 
    If $A^{(2)}(\te_0, 0) = I$ for some $\te_0 \in \mathbb{T}^d$, then there exists $C(d, m, \La) > 0$ so that for almost all $(\te, x) \in \T^d \times (B_r \setminus \{0\}) \subset \T^d \times \R^m$, 
	\begin{align}
		\abs{ \divv(A(\te, x) x)- \mu(\te, x) m} & \le C (r+1)\eta(r) \label{eqn:conf1}\\ 
		\abs{ D Z(\te, x)  - \begin{pmatrix}
			0 & 0 \\
			0 & I 
		\end{pmatrix}} & \le  C (r+1)\eta(r) ,
        \label{eqn:conf2}
	\end{align}
    where $\eta(r)$ is defined in \eqref{eqn:A_lip}.
\end{lemma}

\begin{proof}
    We only provide the proof of \eqref{eqn:conf1}, since that of \eqref{eqn:conf2} is very similar.

    To this end, first we notice that since $A$ is uniformly elliptic, then for $\te \in \T^d$, $x \in B_r \setminus \{0\}$, 
    \begin{align*}
        \mu(\te, x)  = \frac{ \langle A(\te, x) x, x \rangle }{\abs{x}^2} = \frac{ \langle A^{(2)}(\te, x) x, x \rangle }{\abs{x}^2}  \le \Lambda, \, \, \mu(\te, x)^{-1} \le \Lambda.
    \end{align*}
    Since $A$ is Lipschitz, then
    \begin{align*}
        \abs{  \mu(\te, x) - 1   } & = \abs{  \frac{ \langle A^{(2)}(\te, x) x, x \rangle }{\abs{x}^2} - \frac{\langle I x, x \rangle }{\abs{x}^2} } \\
        & = \abs{  x }^{-2} \abs{ \innp{ \brac{A^{(2)}(\te, x) - A^{(2)}(\te_0,0)}x, x }}\\
        & \le \norm{\nabla A}_{L^\infty(\mathbb{K}_r)} \abs{ (\te, x) - (\te_0, 0)  } \le C (r+1)\eta(r).
    \end{align*}
    Next, we compute directly using the structural assumption \eqref{eqn:A_struct}, 
    \begin{align*}
        \divv(A(\te, x)x)  
        &= \divv_x(A^{(2)}(\te, x) x)  
        = \sum_{i,j=1}^{m} \partial_{x_j}  (  a^{(2)}_{ij} (\te, x) x_i ) \\
        &= \sum_{i,j=1}^{m} a^{(2)}_{ij} (\te, x) \delta_{ij} 
        + \sum_{i,j=1}^{m} \partial_{x_j} a^{(2)}_{ij}(\te, x) x_i 
        = \mathrm{Tr}(A^{(2)}(\te, x)) + O( r \eta(r)  ) 
    \end{align*}
    for almost every $(\te, x)$.
    Using again the Lipschitz nature of $A$ and the fact that $A^{(2)}(\te_0,0) = I$, we can estimate 
    \begin{align*}
        \abs{\mathrm{Tr}(A^{(2)}(\te, x)) - m} \le C (r+1)\eta(r),
    \end{align*}
    and use the estimate on $\abs{\mu(\te, x) - 1}$ above to conclude that for a.e. $(\te, x) \in \T^d \times B_r \setminus \set{0}$, it holds
    \begin{align*}
        \abs{ \divv(A(\te, x)x) - m\mu(\te, x)  } \le Cm \eta(r)(r+1),
    \end{align*}
    which is the estimate \eqref{eqn:conf1}.
    The proof of \eqref{eqn:conf2} follows a very similar computation, but one also needs to estimate $\nabla(\mu(\te, x)^{-1})$, which additionally uses the fact that $\mu(\te, x)\simeq_\Lambda 1$.
\end{proof}

\subsection{Frequency functions and almost-monotonicity}

We now introduce our frequency functions and the associated quantities that are used to prove almost-monotonicity results for solutions $u$ to \eqref{eqn:soln_cyl}. 
To begin, we define the weight function 
$$w_r(\te, x) = w_r(x) \coloneqq r^2 - \abs{x}^2$$
and notice that $w_r$ vanishes on $\mathbb{T}^d \times \del B_r(0)$.
This observation allows us to compute the following derivatives. 
This lemma will be used throughout this subsection.

\begin{lemma}[Derivative result]
\label{lem:deriv}
	Given $f \in W^{1,1}(\mathbb{K}_R)$, $\al \ge 1$, define
	\begin{align*}
		F(r) \coloneqq \int_{\mathbb{K}_r} f(\te, x) \, w_{r}(\te, x)^\alpha \, d\te dx.
	\end{align*}
	For any $r \in (0, R$), we have that 
	\begin{align*}
		F'(r) & = \dfrac{2 \alpha }{r} F(r) + \dfrac{2\alpha}{r} \int_{\mathbb{K}_r} \abs{x}^2 f w_r^{\alpha-1} \\
		& = \dfrac{2 \alpha + m}{r} F(r) + \dfrac{1}{r} \int_{\mathbb{K}_r} (\nabla f \cdot x) w_r^\alpha.
	\end{align*}
\end{lemma}

\begin{proof}
	Using that $w_r^\alpha$ vanishes on $\T^d \times \partial B_r$, we compute
	\begin{align*}
		F'(r) & = 2\alpha r \int_{\mathbb{K}_r} f w_r^{\alpha -1} \\
		& = \dfrac{2\alpha}{r} \int_{\mathbb{K}_r} f w_r^\alpha + \dfrac{2\alpha}{r} \int_{\mathbb{K}_r} f \abs{x}^2  w_r^{\alpha-1}   \\
		& = \dfrac{2\alpha}{r} \int_{\mathbb{K}_r} f w_r^\alpha + \dfrac{-1}{r} \int_{\mathbb{K}_r} f \, x \cdot \nabla(  w_r^{\alpha} )   \\
		& = \dfrac{2\alpha + m}{r} \int_{\mathbb{K}_r} f w_r^\alpha + \dfrac{1}{r} \int_{\mathbb{K}_r} (\nabla f \cdot x) w_r^\alpha.
	\end{align*}
	This proves both equalities, as required.
\end{proof}

Next, we introduce our frequency functions and establish almost-monotonicity results.
For $\alpha \ge 1$, a parameter to be determined, let
\begin{align}
	\label{eqn:H_def}
    H(r) & \coloneqq \int_{\mathbb{K}_r} \mu(\te, x) \abs{u(\te, x)}^2 w_r^{\alpha -1 }(x) \, d\te dx, \\
    \label{eqn:D_def}
    D(r) & \coloneqq \int_{\mathbb{K}_r}  A(\te, x) \nabla u(\te, x)   \cdot   \nabla u(\te, x) \,  w_r^\alpha(x) \, d\te dx, \\
	\label{eqn:N_def}
	N(r) & \coloneqq  \frac{D(r)}{H(r)}.
\end{align} 
It will also be useful to define the quantities
\begin{align}
	L(r) & \coloneqq  \int_{\mathbb{K}_r} u(\te, x) \divv( A(\te, x) \nabla u(\te, x)) w_r^\alpha(x) \, d\te dx, 
    \label{eqn:L_def} \\
	I(r) & \coloneqq 2 \alpha \int_{\mathbb{K}_r} u(\te, x) \brac{A(\te, x) \nabla u(\te, x) \cdot x} w_r^{\alpha - 1}(x) \, d\te dx
    \label{eqn:I_def}
\end{align}
since an integration by parts shows that
\begin{align}
\label{eqn:I_id}
    I(r)
    & = \int_{\mathbb{K}_r} \brac{\pr{A \nabla u \cdot \nabla u }  + u \divv(A \nabla u)} w_r^\alpha 
	= D(r) + L(r).
\end{align}
We now state the derivative bounds that lead to almost-monotonicity results. 

\begin{prop}[Derivative bounds]
\label{prop:monotone_cyl}
    For some $R > 0$, let $A \in \mathcal{M}(\mathbb{K}_R; d, m, \La)$ and assume that $A^{(2)}(\te_0,0) = I$ for some $\te_0 \in \T^d$.
    If $u \in H^2(\mathbb{K}_R)$ and $\al \ge 1$, then for every $r \in (0, R)$,
	\begin{align}
    \label{eqn:logHDeriv}
		\dfrac{H'(r)}{H(r)} 
        & = \dfrac{2  (\alpha - 1) + m }{r} 
        + \dfrac{N(r)}{\alpha r } 
        + \dfrac{L(r)}{\alpha r H(r)} 
        + e_1(r) \\
	\label{eqn:NDeriv}
		N'(r) 
        & \ge \dfrac{-1}{4 \alpha r H(r)} \int_{\mathbb{K}_r}  \abs{\divv(A \nabla u)}^2  \mu^{-1} w_r^{\alpha +1} 
        - c_1  \left( \frac{r+1}{r}  \right)\eta(r) N(r),
    \end{align}
	where $\abs{e_1(r)} \le c_1  \left( \tfrac{r+1}{r}  \right)  \eta(r)$ for some $c_1(d, m, \La) > 0$.
\end{prop}

\begin{rmk}
In general, one might worry about the error terms involving the factor $\disp \tfrac{r+1}{r}$ for $r$ small, since this factor is non-integrable at the origin. 
However, since we prove unique continuation results at infinity, we only ever consider $r \ge \Lambda^{-1/2}$, say. 
A close inspection of the proof (and the proof of Lemma \ref{lem:conf_est}) shows that one can replace the term $\disp \tfrac{r+1}{r}$ by $1$ if either $A^{(2)}(\te, x) = a^{(2)}(\te, x) I$ is a scalar multiple of the $m \times m$ identity matrix or $d=0$.
\end{rmk}

\begin{proof}
	First, we compute the derivative of $H(r)$, see \eqref{eqn:H_def}. 
    The first equality of Lemma \ref{lem:deriv}, the definition of $\mu$, and an integration by parts give
	\begin{equation*}
    	\begin{split}
			H'(r) - \dfrac{2(\alpha-1)}{r} H(r) 
            & = \dfrac{2(\alpha -1)}{r} \int_{\mathbb{K}_r} \mu \abs{u}^2 \abs{x}^2  w_r^{\alpha -2} \\
			& = \dfrac{1}{r} \int_{\mathbb{K}_r} \abs{u}^2 ( A x \cdot  2(\alpha -1)x ) w_r^{\alpha -2} \\
			& = \dfrac{-1}{r} \int_{\mathbb{K}_r} \abs{u}^2 (Ax \cdot \nabla w_r^{\alpha-1} ) \\
			& = \dfrac{1}{r} \int_{\mathbb{K}_r} \divv( \abs{u}^2 A x ) w_r^{\alpha -1 } \\
			&  = \dfrac{1}{r} \int_{\mathbb{K}_r} \abs{u}^2 \divv(A x) w_r^{\alpha -1}
			+ \dfrac{2}{r} \int_{\mathbb{K}_r} u A \nabla u \cdot x \,  w_r^{\alpha -1} .
		\end{split}
	\end{equation*}
    That is,
    \begin{equation}
        \label{eqn:H'}
	\begin{aligned}
	    H'(r)
        &= \dfrac{2(\alpha-1) + m}{r} H(r) 
        + \dfrac{1}{r} \int_{\mathbb{K}_r} \abs{u}^2 \brac{\divv(A x) - m \mu} w_r^{\alpha -1}
            \\ 
            &\quad + \dfrac{2}{r} \int_{\mathbb{K}_r}   u \nabla u \cdot A x \, w_r^{\alpha -1} 
            \\
			&  = H(r)  \brac{\dfrac{2(\alpha-1) + m}{r}  + O \left ( \left(  \frac{r+1}{r} \right) \eta(r)  \right ) } 
            + \dfrac{D(r) + L(r)}{\alpha r},
	\end{aligned}
    \end{equation}
    where we have used \eqref{eqn:conf1} in Lemma \ref{lem:conf_est}, \eqref{eqn:I_def}, and \eqref{eqn:I_id}.
    Referring to \eqref{eqn:N_def}, the expression in \eqref{eqn:logHDeriv} has been shown.
    
	We now consider the derivative of $D(r)$, see \eqref{eqn:D_def}.
    Appealing to Lemma \ref{lem:deriv} again, and using that $\abs{x}^2 = Z(\te, x) \cdot x$, we get
    \begin{equation}
    \label{eqn:D'comp}
        \begin{aligned}
        D'(r) - \frac{2\alpha }{r} D(r) 
        & = \frac{2 \alpha}{r} \int_{\mathbb{K}_r} (A \nabla u \cdot \nabla u) Z \cdot x  \,  w_r^{\alpha -1} \\
    	& = \frac{-1}{r} \int_{\mathbb{K}_r}  (A \nabla u \cdot \nabla u) Z \cdot \nabla(w_r^\alpha ) \\
    	& = \frac{1}{r} \int_{\mathbb{K}_r} \divv \left(  A \nabla u \cdot \nabla u \, Z   \right) w_r^\alpha \\
        &= \frac 1 r \int_{\mathbb{K}_r} \brac{ \sum_{k=1}^m  (A \nabla u \cdot \nabla u)_{x_k} Z_k + A \nabla u \cdot \nabla u \divv(Z) } w_r^\alpha.    
        \end{aligned}
    \end{equation}
    Since
    $$A \gr u \cdot \gr u = \sum_{i, j = 1}^d a^{(1)}_{ij} u_{\te_i} u_{\te_j} +  \sum_{i, j = 1}^{m}  a^{(2)}_{ij} u_{x_i} u_{x_j},$$
    then the symmetry of $A$ shows that in $\mathbb{K}_r$,
    $$\pr{A \gr u \cdot \gr u}_{x_k} = 
    2 \sum_{i, j = 1}^d a^{(1)}_{ij} u_{\te_i} u_{\te_j x_k} 
    + 2 \sum_{i,j=1}^m  a^{(2)}_{ij} u_{x_i} u_{x_j x_k} 
    + O(\eta(r)) \abs{\nabla u}^2.$$
    An integration by parts shows that
    \begin{align*}
        &\int_{\mathbb{K}_r} \sum_{k=1}^m  \brac{\sum_{i, j = 1}^d a^{(1)}_{ij} u_{\te_i} u_{\te_j x_k}} Z_k w_r^\alpha \\
        &\quad = - \int_{\mathbb{K}_r} \sum_{k=1}^m u_{x_k} Z_k \sum_{i, j = 1}^d \del_{\te_j} \pr{a^{(1)}_{ij} u_{\te_i}} w_r^\alpha 
        - \int_{\mathbb{K}_r} \sum_{k=1}^m u_{x_k} \sum_{i, j = 1}^d a^{(1)}_{ij} u_{\te_i} \del_{\te_j}{Z_k } w_r^\alpha
    \end{align*}
    and
    \begin{align*}
        &\int_{\mathbb{K}_r} \sum_{k=1}^m  \brac{\sum_{i,j=1}^m  a^{(2)}_{ij} u_{x_i} u_{x_j x_k} } Z_k w_r^\alpha \\
        &\quad = -\int_{\mathbb{K}_r} \sum_{k=1}^m u_{x_k} Z_k \sum_{i,j=1}^m  \del_{x_j} \pr{a^{(2)}_{ij} u_{x_i} }   w_r^\alpha
        -\int_{\mathbb{K}_r} \sum_{k=1}^m u_{x_k} \sum_{i,j=1}^m a^{(2)}_{ij} u_{x_i} \pr{\del_{x_j} Z_k - \de_{jk} } w_r^\alpha \\
        &\quad-\int_{\mathbb{K}_r} \sum_{i,j=1}^m  a^{(2)}_{ij} u_{x_i} u_{x_j} w_r^\alpha
        + 2 \al \int_{\mathbb{K}_r} \sum_{k=1}^m u_{x_k} Z_k \sum_{i,j=1}^m a^{(2)}_{ij} u_{x_i}  x_j w_r^{\alpha-1}.
    \end{align*}
    Therefore, with an application of \eqref{eqn:conf2} in Lemma \ref{lem:conf_est}, we see that
    \begin{equation*}
        \begin{aligned}
         \int_{\mathbb{K}_r} \brac{ \sum_{k=1}^m  (A \nabla u \cdot \nabla u)_{x_k} Z_k} w_r^\alpha
        &=  - 2\int_{\mathbb{K}_r} {\di\pr{A \gr u}} \pr{A \gr u \cdot x} \mu^{-1} w_r^\alpha \\
        &\quad-2\int_{\mathbb{K}_r} A^{(2)} \gr_x u \cdot \gr_x u \, w_r^\alpha
        + 4 \al \int_{\mathbb{K}_r} \pr{A \gr u \cdot x}^2 \mu^{-1} w_r^{\alpha-1} \\
        &\quad+ O((r+1) \eta(r)) \int_{\mathbb{K}_r} \abs{\gr u}^2 w_r^\alpha,   
        \end{aligned}
    \end{equation*}
    where we have used the definition of $Z$ from \eqref{eqn:vec_field}.
    Another application of \eqref{eqn:conf2} in Lemma \ref{lem:conf_est} shows that
    \begin{equation*}
        \begin{aligned}
        \int_{\mathbb{K}_r} \pr{A \nabla u \cdot \nabla u} \divv(Z) w_r^\alpha
        &= m D(r) + \int_{\mathbb{K}_r} \pr{A \nabla u \cdot \nabla u} \brac{\divv(Z) - m} w_r^\alpha \\
        &= m D(r) + O((r+1) \eta(r))  \int_{\mathbb{K}_r} \abs{\gr u}^2 w_r^\alpha.   
        \end{aligned}
    \end{equation*}
    Using the fact that $A^{(2)} \nabla_x u \cdot \nabla_x u \le A \nabla u \cdot \nabla u$, we substitute the last two equalities into \eqref{eqn:D'comp} and complete the square to get
    \begin{equation}\label{eqn:D'}
    	\begin{split}
    		D'(r)  
            &\ge \frac{2(\alpha-1) +m}{r} D(r)
            + O\pr{\pr{\frac{r+1}r} \eta(r)}  \int_{\mathbb{K}_r} \abs{\gr u}^2 w_r^\alpha \\
             &\quad+ \frac{4 \al}r \int_{\mathbb{K}_r} \pr{A \gr u \cdot x}^2 \mu^{-1} w_r^{\alpha-1}
             - \frac 2 r \int_{\mathbb{K}_r} {\di\pr{A \gr u}}\pr{A \gr u \cdot x} \mu^{-1} w_r^\alpha \\
        	& \ge \dfrac{2(\alpha - 1) + m}{r} D(r)
    		+ \dfrac{4\alpha}{r} \int_{\mathbb{K}_r} \left( A\nabla u \cdot x  - \dfrac{ \divv(A \nabla u) \omega_r}{ 4\alpha  } \right)^2  \mu^{-1} \omega_r^{\alpha -1} \\
    		& \quad - \dfrac{1}{4 \alpha r} \int_{\mathbb{K}_r}  \abs{\divv(A \nabla u)}^2  \mu^{-1} w_r^{\alpha +1} 
    		+ O \left(  \left( \frac{r+1}{r} \right)\eta(r) \right ) D(r).
    	\end{split}
    \end{equation}
   
	Next, we set
	\begin{align*}
		J(r) 
        \coloneqq \dfrac{I(r)+D(r)}{2} 
        = 2\alpha  \int_{\mathbb{K}_r} u \left( A\nabla u \cdot x - \dfrac{\divv(A \nabla u) w_r}{4\alpha} \right) w_r^{\alpha -1},
	\end{align*}
	and notice from the relations of $L$ and $I$ in \eqref{eqn:L_def} and \eqref{eqn:I_def}, respectively, that 
	\begin{align*}
		D(r) & = J(r) - \dfrac{1}{2} L(r), \\
		I(r) & = J(r) + \dfrac{1}{2}L(r).
	\end{align*}
	Putting together our estimates from \eqref{eqn:H'} and \eqref{eqn:D'}, we get
	\begin{align*}
		H(r)^2 N'(r) 
        & = D'(r) H(r) - D(r) H'(r) \\
		& \ge H(r) \brac{ \dfrac{4\alpha}{r} \int_{\mathbb{K}_r} \left( A \nabla u \cdot x  -\dfrac{ \divv(A \nabla u) \omega_r}{ 4\alpha  } \right)^2  \mu^{-1} \omega_r^{\alpha -1} } \\
		& \quad - \dfrac{1}{\alpha r} \left( J(r) - \dfrac{1}{2} L(r) \right) \left( J(r) + \dfrac{1}{2} L(r)  \right) \\
		& \quad - H(r)\dfrac{1}{4 \alpha r} \int_{\mathbb{K}_r}  \abs{\divv(A \nabla u)}^2  \mu^{-1} w_r^{\alpha +1}  + O\left ( \left(\frac{r+1}{r}  \right)   \eta(r) \right ) D(r) H(r) \\
		& = \dfrac{4\alpha}{r} \brac{ \int_{\mathbb{K}_r}  \mu \abs{u}^2 w_r^{\alpha -1} } \brac{ \int_{\mathbb{K}_r} \left( A \nabla u \cdot x  -\dfrac{ \divv(A \nabla u) \omega_r}{ 4\alpha  } \right)^2  \mu^{-1} \omega_r^{\alpha -1} } \\
		& \quad - \dfrac{4 \alpha }{ r} \brac{ \int_{\mathbb{K}_r} u \left( A \nabla u \cdot x - \dfrac{\divv(A \nabla u) w_r}{4\alpha} \right) w_r^{\alpha -1} }^2  
        +  \dfrac{L(r)^2 }{4 \alpha r} \\
		& \quad - H(r) \dfrac{1}{4 \alpha r} \int_{\mathbb{K}_r} \abs{\divv(A \nabla u)}^2  \mu^{-1} w_r^{\alpha +1}  
        + O\left ( \left(\frac{r+1}{r}  \right)   \eta(r) \right ) D(r) H(r)   \\
		& \ge  - H(r) \dfrac{1}{4 \alpha r} \int_{\mathbb{K}_r}  \abs{\divv(A \nabla u)}^2  \mu^{-1} w_r^{\alpha +1} 
        + O\left ( \left(\frac{r+1}{r}  \right)   \eta(r) \right ) D(r) H(r)
	\end{align*}
	by courtesy of Cauchy-Schwarz. 
    Using the expression for $N(r)$ in \eqref{eqn:N_def}, one infers from the inequality above that 
	\begin{align*}
		N'(r) \ge \dfrac{-1}{4 \alpha r H(r)} \int_{\mathbb{K}_r}  \abs{\divv(A \nabla u)}^2  \mu^{-1} w_r^{\alpha +1} - c_1  \left( \frac{r+1}{r} \right)\eta(r) N(r) 
	\end{align*}
	for some $c_1 > 0$, which completes the proof.
\end{proof}

For solutions to \eqref{eqn:soln_cyl}, we obtain the following corollary. 

\begin{cor}[Almost-monotonicity]
\label{cor:monotone_cyl}
    For some $R > \Lambda^{-1/2}$, let $A \in \mathcal{M}(\mathbb{K}_R; d, m, \La)$ and assume that $A^{(2)}(\te_0,0) = I$ for some $\te_0 \in \T^d$.
    Let $V$ and $W$ satisfy the bounds in \eqref{eqn:pot_bounds}. 
    If $u \in H^1(\mathbb{K}_R)$ is a solution of \eqref{eqn:soln_cyl} in $\mathbb{K}_R$, then for each $r \in \pr{\La^{-1/2}, R}$, we have that
    \begin{align*}
		\dfrac{H'(r)}{H(r)} & = \dfrac{2(\alpha -1) +m}{r} + \dfrac{N(r)}{\alpha r} + e_1(r) + e_2(r).
	\end{align*}
    There exists $c_1(d, m, \La) > 0$ so that $\abs{e_1(r)}  \le c_1 \eta(r)$ and for each $\epsilon > 0$, there exists $c_\epsilon(d, m, \La, \eps) > 0$ so that
	\begin{align*}
		\abs{e_2(r)}  \le c_{\epsilon} \dfrac{(K^2 + M)r}{\alpha}
        + \epsilon  \dfrac{ N(r)}{\alpha r} .
	\end{align*}
    Moreover, recalling \eqref{eqn:phi_def}, if we define 
    \begin{equation}
    \label{eqn:tilNDef}
    \tilde{N}(r) \coloneqq \brac{N(r) +\dfrac{c_1 M^2 r^4}{\alpha} }  \exp\brac{c_1 \pr{\phi(r) + \frac{K^2 r^2}{\alpha} }},
    \end{equation}
    then $\tilde{N}(r)$ is monotone non-decreasing in $r$ on $(\Lambda^{-1/2}, R)$.
\end{cor}

\begin{proof}
    The estimate on $e_1(r)$ follows from Proposition \ref{prop:monotone_cyl} and that $r \ge \La^{- \frac 1 2}$.
    With
    $$e_2(r) := \frac{L(r)}{ \al r H(r)},$$
    since
    \begin{align*}
        \abs{L(r)}
        &\le \abs{\int_{\mathbb{K}_r} u \pr{W \cdot \gr u + V u} w_r^\alpha}
        \le \La^{-1} \int_{\mathbb{K}_r} \brac{ \pr{c_\eps \abs{W}^2 + \abs{V}} \abs{u}^2 + \eps \abs{\gr u}^2 } w_r^\alpha \\
        &\le c_\eps \pr{K^2 + M} r^2 H(r) + \eps D(r),
    \end{align*}
    then the bound on $e_2$ follows. 
   
	Using the equation for $u$ in \eqref{eqn:soln_cyl} and the derivative bound described by \eqref{eqn:NDeriv} in Proposition \ref{prop:monotone_cyl}, we see that for $r \in \pr{\La^{-1/2}, R}$,
	\begin{align*}
		N'(r) \ge -c_1 \brac{ \pr{\eta(r) + \dfrac{K^2 r}{\alpha}} N(r) +  \dfrac{M^2r^{3}}{\alpha } }
	\end{align*}
	for some $c_1(d, m, \La) > 0$. 
    Differentiating the expression in \eqref{eqn:tilNDef}, using \eqref{eqn:phi_def}, and comparing with the above bound shows that $\tilde{N}(r)$ is monotone non-decreasing.
\end{proof}

\subsection{Three-cylinder inequalities}

By using the almost-monotonicity results above, and choosing $\al \gg 1$, we prove the following three-cylinder inequality. 

\begin{prop}[Three-cylinder inequality for solutions]
\label{prop:3cylinder}
    For some $R > \Lambda^{-1/2}$, let $A \in \mathcal{M}(\mathbb{K}_R; d, m, \La)$ and assume that $A^{(2)}(\te_0,0) = I$ for some $\te_0 \in \T^d$.
    Let $V$ and $W$ satisfy the bounds in \eqref{eqn:pot_bounds}. 
    If $u \in H^1(\mathbb{K}_R)$ is a solution of \eqref{eqn:soln_cyl} in $\mathbb{K}_R$, then for $\Lambda^{-1/2} \le  r_1 < r_2 < \sigma r_2 < r_3 < R$, it holds that
	\begin{equation}
    \label{eqn:3_rect}
	   \begin{aligned}
	   \norm{u}_{r_2}
        &\le \exp \set{ \brac{ c_2 + \log\pr{\frac{\sigma^2}{\sigma^2-1}} } F(R)
        + c_2 \phi(R) } 
        \norm{u}_{r_1}^\kappa 
        \norm{u}_{r_3}^{1-\kappa},  
	    \end{aligned}
	\end{equation}
    where $\norm{u}_r := \norm{u}_{L^2 (\mathbb{K}_{r})}$, $c_2(d, m, \La) > 0$, 
    \begin{equation}
    \label{eqn:bigF_defn}
        F(R) \coloneqq 1 + \pr{KR}^2 + \pr{MR^2}^{\frac 2 3},
    \end{equation}
    $\phi$ is defined in \eqref{eqn:phi_def}, and with $c_1(d, m, \La) > 0$, $\kappa \in (0,1)$ is the exponent defined by 
	\begin{align}
    \label{eqn:kappa_def}
		\kappa \coloneqq 
		\dfrac{\log\pr{\frac{r_3}{\sigma r_2}}}{ e^{c_1 \pr{\phi(R) + 2}}  \log\pr{\frac{\sigma r_2}{r_1}}  + \log\pr{\frac{r_3}{\sigma r_2}}}.
	\end{align}
\end{prop}

\begin{proof}
	We define $h(r) \coloneqq r^{-[2(\alpha -1) + m]} H(r)$, so that by Corollary \ref{cor:monotone_cyl}, for any $\epsilon \in (0,1)$, there is $c_\epsilon \ge c_1 >0$ so that 
	\begin{align*}
		\dfrac{d}{dr} \brac{\log(h(r))} 
        &  \le  \dfrac{ (1 + \epsilon) N(r) }{\alpha r} 
        + c_\epsilon \dfrac{(K^2 + M)r}{\alpha} 
        + c_1 \eta (r) \\
		&  \le  \dfrac{(1 + \epsilon)\tilde{N}(r)}{\alpha r} 
        + c_\eps \brac{\dfrac{(K^2 + M)r}{\alpha} + \eta (r)},
	\end{align*}
    where we recall \eqref{eqn:tilNDef}.
    We can also bound from below to get
	\begin{align*}
		\dfrac{d}{dr} \brac{\log(h(r))} 
        & \ge  \dfrac{(1-\epsilon)N(r)}{\alpha r} 
        - c_\epsilon \dfrac{(K^2 + M)r}{\alpha} 
        - c_1 \eta (r) \\ 
		& \ge   \dfrac{ (1- \epsilon) \tilde{N}(r)e^{-c_1 \pr{\phi(r) + \frac{K^2 r^2}{\alpha} }}}{\alpha r} 
        - c_\eps \brac{ \dfrac{M^2 r^3}{\alpha^2 } + \dfrac{(K^2 + M)r}{\alpha} + \eta (r) }.
	\end{align*}
	Using the monotonicity of $\tilde{N}$, we integrate the upper bound to get
    \begin{align*}
		\log \brac{\dfrac{ h(\sigma r_2) }{ h(r_1) }} 
        & \le \dfrac{1+\epsilon}{\alpha } \int_{r_1}^{\sigma r_2}  \frac{\tilde{N}(r)}{r} dr
        + c_\eps \int_{r_1}^{\sigma r_2} \brac{\dfrac{(K^2 + M)r}{\alpha} + \eta (r)} dr \\
        &\le \dfrac{(1+\epsilon) \tilde N(\si r_2)}{\alpha } I(r_1, \sigma r_2)
        + c_\eps \brac{\dfrac{(K^2 + M)R^2}{2\alpha} + \phi (R)},
	\end{align*}
    where we introduce $\disp I(r_1, r_2) := \int_{r_1}^{r_2}  \frac{dr}{r}$.
    Similarly, integrating the lower bound shows that
	\begin{align*}
		 \log &\brac{\dfrac{ h(r_3) }{ h(\sigma r_2) }} \\
        & \ge \frac{1-\eps}{\al} \int_{\sigma r_2}^{r_3}  \dfrac{ \tilde{N}(r) e^{-c_1 \pr{\phi(r) + \frac{K^2 r^2}{\alpha} }}}{r} dr 
        - c_\eps \int_{\sigma r_2}^{r_3} \brac{\dfrac{M^2 r^3}{\alpha^2 } + \dfrac{(K^2 + M)r}{\alpha} +  \eta (r) } dr \\
		& \ge \dfrac{(1-\epsilon) \tilde{N}(\sigma r_2)  e^{-c_1 \pr{\phi(R) + \frac{K^2 R^2}{\alpha} }} }{\alpha } I(\sigma r_2, r_3) 
        -  c_\eps \brac{\dfrac{M^2 R^4}{4\alpha^2 } + \dfrac{(K^2 + M)R^2}{2\alpha} + \phi (R)}.
	\end{align*}
    With $\al = 1 + \pr{KR}^2 + \pr{MR^2}^{\frac 2 3}$, we see that
    \begin{align*}
        \log \brac{\dfrac{ h(\sigma r_2) }{ h(r_1) }} 
        &\le \dfrac{(1+\epsilon) \tilde N(\si r_2)}{\alpha } I(r_1, \sigma r_2)  
        + c_\eps \pr{ \alpha + \phi (R)} \\
        \log \brac{\dfrac{ h(r_3) }{ h(\sigma r_2) }} 
        & \ge \dfrac{(1-\epsilon) \tilde{N}(\sigma r_2)  e^{-c_1 \pr{\phi(R) + 1}} }{\alpha } I(\sigma r_2, r_3) 
        -  c_\eps \pr{ \al + \phi (R)}.
    \end{align*}
	Combining these inequalities yields
	\begin{align*}
		\dfrac{ \log \brac{\dfrac{ h(\sigma r_2) }{ h(r_1) }}  - c_\eps \pr{ \al + \phi (R)} }{ I(r_1, \sigma r_2) } 
        & \le \frac{1+\epsilon}{1 - \eps} e^{c_1 \pr{\phi(R) + 1 }} \frac{\pr{1 - \eps} \tilde{N}(\sigma r_2)}{\al} e^{-c_1 \pr{\phi(R) + 1 }}  \\
		& \le \frac{1+\epsilon}{1 - \eps} e^{c_1 \pr{\phi(R) + 1 }} \dfrac{ \log \brac{\dfrac{ h(r_3) }{ h(\sigma r_2) }} + c_\eps \pr{ \al + \phi (R)} }{ I(\sigma r_2, r_3) }.
	\end{align*}
	
	At this stage, we fix $\epsilon(d, m, \La) > 0$ small enough so that $(1+\epsilon) \le e^{c_1} (1-\epsilon)$.     
    With $\beta \coloneqq  I(r_1, \sigma r_2)   e^{c_1\pr{ \phi(R) + 2 }}$ and $\gamma \coloneqq  I(\sigma r_2, r_3) $, we simplify this expression to get
	\begin{align*}
		\log(h(\sigma r_2)) 
        \le \dfrac{\gamma}{\gamma + \beta} \log(h(r_1)) 
        + \dfrac{\beta}{\gamma + \beta} \log(h(r_3)) 
        + c_2 \pr{ \al + \phi (R)} ,
	\end{align*}
    where $c_2(d, m, \La) > 0$.
	Taking exponentials in this expression shows that
	\begin{align*}
		h(\sigma r_2) \le \exp \set{ c_2 \brac{ \al + \phi (R)} } h(r_1)^\kappa h(r_3)^{1-\kappa},
	\end{align*}
	where $\disp \kappa = \frac{\gamma}{\gamma + \beta} \in (0,1)$. 
    Computing $I(r_1, \sigma r_2)$ and $I(\sigma r_2, r_3)$ gives the formula for $\kappa$ as in the statement of the proposition.
		
	As a last step, define $\disp \tilde{H}(r) \coloneqq \int_{\mathbb{K}_r} \abs{u}^2$ and notice that $h(r) \le \La r^{-m} \tilde{H}(r)$ and $\tilde{H}(r) \le \La  \rho^m \left( 1 - \frac{r^2}{\rho^2}  \right)^{-\alpha +1 } h(\rho)$ for $0 < r < \rho$. 
    Hence we conclude that   
    \begin{equation}
    \label{eqn:almost3cyl}
        \begin{aligned}
        \tilde{H}(r_2) 
        & \le \La  (\sigma r_2)^{m} \left( \frac{\sigma^{2}}{ \sigma^2-1}  \right)^{\alpha-1} h(\sigma r_2) \\
		& \le \La  (\sigma r_2)^{m} \left( \frac{\sigma^{2}}{ \sigma^2-1}  \right)^{\alpha-1}  \exp \set{ c_2 \brac{ \al + \phi (R)} } h(r_1)^\kappa h(r_3)^{1-\kappa} \\
		& \le \La^{2}  \left (\dfrac{\sigma r_2}{r_1^{\kappa} r_3^{1-\kappa}} \right )^{m} \left( \frac{\sigma^{2}}{ \sigma^2-1}  \right)^{\alpha-1}  \exp \set{ c_2 \brac{ \al + \phi (R)} } \tilde{H}(r_1)^\kappa \tilde{H}(r_3)^{1-\kappa}.
        \end{aligned}
    \end{equation}
	Since
    \begin{align*}
        \log \pr{\dfrac{\sigma r_2}{r_1^{\kappa} r_3^{1-\kappa}}}
        &= \kappa \log\pr{\frac{\sigma r_2}{r_1}} - \pr{1 - \kappa} \log\pr{\frac{r_3}{\si r_2}} \\
        &= \frac{I(\sigma r_2, r_3)I(r_1, \sigma r_2) \pr{1 - e^{c_1\pr{ \phi(R) + 2 }}}}{I(\sigma r_2, r_3) + I(r_1, \sigma r_2)   e^{c_1\pr{ \phi(R) + 2 }}}  
        \le 0,
    \end{align*}
    then
    \begin{align*}
        \tilde{H}(r_2) 
		& \le \exp \set{ c_2 \brac{ \al + \phi(R)} 
        + \al \log \pr{\frac{\si^2}{\si^2 - 1}} 
        + 2 \log \La} 
        \tilde{H}(r_1)^\kappa 
        \tilde{H}(r_3)^{1-\kappa}.
    \end{align*}
    Because we take $\al = F(R)$, see \eqref{eqn:bigF_defn}, then the conclusion described by \eqref{eqn:3_rect} follows after we adjust the definition of $c_2$.
\end{proof}

We next prove an eigenfunction version of the previous result.

\begin{cor}[Three-cylinder inequality for eigenfunctions]
\label{cor:3cylinder}
    For some $R > \Lambda^{-1/2}$, let $A \in \mathcal{M}(\mathbb{K}_R; d, m, \La)$ and assume that $A^{(2)}(\te_0,0) = I$ for some $\te_0 \in \T^d$.
    Let $V$ and $W$ satisfy the bounds in \eqref{eqn:pot_bounds}. 
    If $u \in H^1(\mathbb{K}_R)$ is a solution of 
    \begin{equation*}
    %\label{eqn:eig_eqn}
        -\di\pr{A \gr u} + W \cdot \gr u + V u = \la u
        \,\, \text{ in } \, \mathbb{K}_R,
    \end{equation*}
    then for $\Lambda^{-1/2} \le  r_1 < r_2 < \sigma r_2 < r_3 < R$, where $\si \le \si_\La$, it holds that
	\begin{align}
    \label{eqn:3_rect_eig}
   	\norm{u}_{r_2} 
        \le \exp \set{ \brac{c_2 + \log\pr{\frac \si {\si - 1}} }F(R)  + c_2 G(R)  } 
        \norm{u}_{r_1}^\kappa 
        \norm{u}_{r_3}^{1-\kappa},	
    \end{align}
	where $c_2(d, m, \La) > 0$, $F(R)$ is defined in \eqref{eqn:bigF_defn}, 
    \begin{equation*}
        G(R) \coloneqq \sqrt \la R + \phi(R),
    \end{equation*}
    where $\phi$ is defined in \eqref{eqn:phi_def}, and $\kappa \in (0,1)$ is the exponent defined by \eqref{eqn:kappa_def} with $c_1(d, m+1, \La) > 0$. 
\end{cor}

\begin{proof}
Given $\te \in \T^d$, $x \in \R^m$, $z \in \R$, we write $(\te, x, z) \in \T^d \times \R^{m+1}$ and $(x, z) \in \R^{m+1}$.
Let $\overline{\mathbb{K}}_R = \T^d \times B_{R}(0, 0) \su \T^d \times \R^{m+1}$, a cylinder in higher dimensions.
On $\overline{\mathbb{K}}_R$, define $u_0(\te, x, z) = e^{\sqrt \la z} u(\te, x)$, $V_0(\te, x, z) = V(\te, x)$, $W_0(\te, x, z) = (W(\te, x), 0)$, and
\begin{align*}
A_0(\te, x, z) &= \begin{pmatrix}
A^{(1)}(\te, x) & 0 & 0 \\
0 & A^{(2)}(\te, x) & 0 \\
0 & 0 & 1
	\end{pmatrix}.
\end{align*}
Since $A \in \mathcal{M}(\mathbb{K}_R; d, m, \La)$, then $A_0 \in \mathcal{M}(\overline{\mathbb{K}}_R; d, m+1, \La)$.
Moreover,
$$\phi_0(R)
    := \int_0^R \norm{\nabla  A_0(\te, x, z)}_{L^\infty(\overline{\mathbb{K}}_s)}\, ds
    = \int_0^R \norm{\nabla  A(\te, x)}_{L^\infty(\mathbb{\mathbb{K}}_s)}\, ds = \phi(R).$$
Observe also that
\begin{align*}
  -\di\pr{A_0 \gr u_0} + W_0 \cdot \gr u_0 + V_0 u_0 = 0
  \iff & -\di\pr{A \gr u} + W \cdot \gr u + V u = \la u.
\end{align*} 
As $\norm{W_0}_\iny \le K$ and $\norm{V_0}_\iny \le M$, then an application of \eqref{eqn:almost3cyl} from the proof of Proposition \ref{prop:3cylinder} (with $r_2$ on the left replaced by $\sqrt \si r_2$) shows that
    \begin{align*}
   	\norm{u_0}_{L^2 (\overline{\mathbb{K}}_{\sqrt \si r_2})}
        &\le \La^{2} \pr{\dfrac{\sigma r_2}{r_1^{\kappa} r_3^{1-\kappa}} }^{m+1} \pr{ \frac{\sigma}{ \sigma-1} }^{\alpha-1} e^{ c_2 \brac{ \al + \phi (R)} } \norm{u_0}_{L^2 (\overline{\mathbb{K}}_{r_1})}^\kappa 
        \norm{u_0}_{L^2 (\overline{\mathbb{K}}_{r_3})}^{1-\kappa},	
    \end{align*}
	where $\al = F(R)$, $c_2(d, m+1, \La) > 0$, and $\kappa$ is given by \eqref{eqn:kappa_def} with $c_1(d, m+1, \La) > 0$. 
	Since 
    \begin{align*}
        \norm{u_0}_{L^2(\ol{\mathbb{K}}_r)}^2 
        &= \iint_{\T^d \times B_r(0,0) \subset \T^d \times \R^{m+1}} \abs{u(\te, x)}^2 e^{2\sqrt{\lambda} z} \, d\te \, dx \,dz 
        \le 2 r e^{2 \sqrt \la r} \norm{u}_{L^2(\mathbb{K}_r)}^2,
    \end{align*}
    while
    \begin{align*}
        \norm{u}_{L^2(\mathbb{K}_r)}^2
        &= \int_{\T^d \times B_{ r }(0)} \abs{u(\te, x)}^2 \, d\te \, dx
        \le \frac{1}{2 r \sqrt{\si - 1} } \int_{\T^d \times B_{ r }(0)}  \abs{u(\te, x)}^2 \int_{ - \sqrt{\si r^2 - \abs{x}^2}  }^{\sqrt{\si r^2 - \abs{x}^2 }} \, dz \, d\te \, dx     \\
        &\le \frac{1}{2 r \sqrt{\si - 1} }  \iint_{\T^d \times B_{\sqrt \si r}(0,0)} \abs{u_0(x,t, z)}^2 e^{- 2 \sqrt \la z}  \, d\te \, dx \, dz \\
        &\le  \frac{e^{2 \sqrt{\si \la} r}}{2 r \sqrt{\si - 1} }  \norm{u_0}^2_{L^2 (\overline{\mathbb{K}}_{\sqrt \si r})}
        = \frac{e^{2 \sqrt{\si \la} r}}{2 r \si } \sqrt \si \sqrt{\frac{\si}{\si - 1}}  \norm{u_0}^2_{L^2 (\overline{\mathbb{K}}_{\sqrt \si r})},
    \end{align*}
    then we see that
    \begin{align*}
		\norm{u}_{L^2(\mathbb{K}_{r_2})}
        &\le \si_\La^{1/4} \La^{2} \pr{\dfrac{\sigma r_2}{r_1^{\kappa} r_3^{1-\kappa}} }^{m+ \frac 1 2} \pr{ \frac{\sigma}{ \sigma-1} }^{\alpha-\frac 3 4} e^{ c_2 \brac{ \al + \phi (R)} + 2 \sqrt \la R} \norm{u}_{L^2(\mathbb{K}_{r_1})}^\kappa 
        \norm{u}_{L^2(\mathbb{K}_{r_3})}^{1-\kappa}.
    \end{align*}
    After simplifying and adjusting the definition of $c_2$, we reach the conclusion in \eqref{eqn:3_rect_eig}.
\end{proof}

\section{Applications to unique continuation at infinity} \label{sec:applications}

In this section, we apply the three-cylinder inequality in \eqref{eqn:3_rect_eig} to solutions of the eigenvalue equation
\begin{equation}
    \label{eqn:eig_eqn0}
    - \di\pr{A \gr u} + W \cdot \gr u + V u = \la u,
\end{equation}
when $\abs{\nabla A}$ exhibits polynomial decay in the $x$-direction:
\begin{align}
\label{eqn:A_decay_expl}
	\abs{\nabla A(\te, x)} \le \dfrac{L}{(1 + \abs{x})^\tau} \,\,\,\, \text{ for a.e. } (\te, x)
\end{align}
and some $\tau \in [0, 1]$.
Depending on the value of $\tau$, we obtain a range of unique continuation-type results at infinity:
\begin{enumerate}[(I)]
	\item If $\tau \in [0,1)$, solutions of \eqref{eqn:eig_eqn0} cannot decay faster than  $\exp \set{ -C  \exp \pr{\abs{x}^{1 - \tau } }}$.
    \label{cond:1}
	\item If $\tau = 1$, solutions of \eqref{eqn:eig_eqn0} cannot decay faster than $\exp\pr{ -C \abs{x}^p}$ for some exponent $p  >  0$.
    \label{cond:thresh}
\end{enumerate}
We also consider the behavior of solutions of \eqref{eqn:eig_eqn0} when $V, W \equiv 0$ and $\abs{\nabla A}$ decays at a slightly faster rate, i.e.,
\begin{align}
\label{eqn:A_decay_expllog}
	\abs{\nabla A(\te, x)} \le \dfrac{L}{(1 + \abs{x}) \log (1 + \abs{x})}
    \,\,\,\, \text{ for a.e. } (\te, x).
\end{align}

\begin{enumerate}[(I)]
\setcounter{enumi}{2}
	\item If \eqref{eqn:A_decay_expllog} holds, solutions  of \eqref{eqn:eig_eqn0} cannot decay faster than $  \exp\set{-C \abs{x}  e^{\log(\abs{x})^{\beta}}}$ for any $\beta \in (1/2, 1)$.
    \label{cond:1+}
\end{enumerate}

One can view the case \ref{cond:thresh} above as the threshold for ``unique continuation at infinity'' for equations of the type \eqref{eqn:eig_eqn0} since $\abs{x}$ only appears polynomially inside the exponential. 
This observation could be compared with the fact that when $A$ is \textit{Lipschitz} continuous, non-trivial local solutions to equations of the type \eqref{eqn:eig_eqn0} have finite order of vanishing, \cite{GL86}.

We point out that the power $p$ of the polynomial decay in \ref{cond:thresh} depends on the constant $L$; we can show that if $L$ is sufficiently small, then $p$ can be taken equal to $2$ (or in the case $W \equiv 0$, $p$ can be taken equal to $4/3$); see Theorem \ref{thm:tau=1}. 
Moreover, the examples we construct in Section \ref{sec:examples} show that this power $p > 0 $ tends to infinity as $L$ tends to infinity (see Theorem \ref{thm:eg_1}).
The proofs of each of the results claimed in \ref{cond:1}--\ref{cond:1+} are consequences of a general iterative argument described by Proposition \ref{prop:decay_cyl_p_eig}.

\subsection{The iterative argument}

We turn to the iterative argument that we use to prove unique continuation estimates at infinity for solutions of \eqref{eqn:eig_eqn0} when $\abs{\gr A}$ satisfies decay like \eqref{eqn:A_decay_expl} or \eqref{eqn:A_decay_expllog}. 
In fact, in this section, we prove a more general result that allows us to deal with general behavior for the decay of $\abs{\nabla A}$. 
The set up for this result (Proposition \ref{prop:decay_cyl_p_eig}) is the following.

For some decreasing function $f:[1, \infty] \rightarrow [0, \infty)$ and some constant $L > 0$, assume the following two conditions on the gradient $\abs{\nabla A}$ :
\begin{equation}
\label{eqn:eta_cyl_p}
	\begin{split}
		\abs{\nabla A(\te, x)} & \le L \quad\quad\,\,\,\,\, \text{for a.e. } (\te, x) \in \T^d \times \R^m, \\
		\abs{ \nabla A(\te, x)  } & \le f(\abs{x}) \,\,\,\,  \text{ for a.e. } (\te, x) \in \T^d \times \R^m, \, \abs{x} > 1.
	\end{split}
\end{equation}
This first global bound on $\abs{\nabla A}$ in the full space $\T^d \times \R^m$ is in some sense necessary, as per the counterexamples to unique continuation of Plis \cite{Pli63} and Miller \cite{Mil74}. 
On the other hand, the function $f$ governs some sort of decay of $\abs{\nabla A(\te, x)}$ as $\abs{x} \ra \infty$.
It is precisely this decay that relates to the optimal rate of decay at infinity for solutions.

Let $\gamma :[1, \infty) \ra [1, \infty)$ be a non-decreasing continuous function.
We need some non-degeneracy of the map $\gamma$, so we assume (for simplicity) that there is some $R_0 \ge 1$ so that 
\begin{align}
\label{eqn:gamma_nondegen}
    \gamma(R) \ge 2  \, \text{ for all }  \,R \ge R_0.
\end{align} 
For $R \ge \gamma(1)$, we introduce the notation 
\begin{equation}
\label{eqn:R_cyl_p}
    \begin{aligned}
     \hat{R} & \coloneqq R \gamma(R), \\
	R &\eqqcolon \check{R} \gamma(\check{R}).   
    \end{aligned}
\end{equation}
Since $\gamma$ is non-decreasing, the existence of $\check{R}$ is guaranteed by the inverse function theorem for the strictly increasing map $R \mapsto R \gamma(R)$. 
The definition of $\check{R}$ is so that if we inductively define the sequence $R_1 \ge \gamma(1)$, $R_k = \hat{R}_{k-1}$, then $\check{R}_k = R_{k-1}$. 
In addition, we introduce the notation
\begin{align}
\label{eqn:phi_tilde_cyl_p}
	\tilde{\phi}(R) & \coloneqq \int_{\frac 1 2 \check{R}}^{R} f(s) \, ds. 
\end{align}
As we show in the following result, $\tilde{\phi}(R)$ captures the influence of $\abs{\gr A}$ .

In the proposition below, we provide a general (albeit cumbersome) framework to deal with a wide variety of the functions $f$ in condition \eqref{eqn:eta_cyl_p}.
We then use this framework to prove quantitative unique continuation results at infinity for solutions to generalized Schr\"{o}dinger equations. 
The point of this proposition is that for each decay function $f$, one must construct a collection of functions, $g$, $h$, and $\gamma$, that satisfy certain compatibility conditions.
Then, using an iterative argument, we can establish lower bounds on the rate of decay at infinity for solutions $u$ in terms of the functions $g$, $h$, and $\tilde \phi$. 
Later, in Theorems \ref{thm:tau_betw}, \ref{thm:tau=1}, and \ref{thm:tau=1+}, we explicitly construct such functions for particular examples of $f$, then appeal to Proposition \ref{prop:decay_cyl_p_eig} to estimate the optimal rates of decay.

\begin{prop}[Iteration scheme]
\label{prop:decay_cyl_p_eig}
    Assume that $A \in \mathcal{M}(d,m,\La)$ satisfies the gradient decay estimate in \eqref{eqn:eta_cyl_p}. 
    Let $\tilde{\phi}$ be as defined in \eqref{eqn:phi_tilde_cyl_p}, where we recall the notation from \eqref{eqn:R_cyl_p} for continuous, non-decreasing $\ga : [1, \iny) \to [1, \iny)$ that satisfies \eqref{eqn:gamma_nondegen}.
    Let $g, h: [0,\infty) \ra [0, \infty)$ be non-negative, increasing functions that satisfy the vanishing condition
	\begin{align}
    \label{eqn:vanishing_eig}
		\limsup_{R \ra \infty} \dfrac{ \gamma(R) \log(\hat R) g(R) h(R) e^{ \tilde c_1 \tilde{\phi}(R)}   }{g(\hat R) h(\hat R)  } & = 0,  
	\end{align}
	where $\tilde c_1 := \La^{3/2} c_1(d, m , \Lambda^2)$ and $c_1 > 0$ is the constant from Corollary \ref{cor:3cylinder}.
    Let $W$ and $V$ satisfy \eqref{eqn:pot_bounds} and let $\la \in \C$.
    For $R_0(d, m, \La, f, \ga, g, h) \ge \max\set{6 \La, \La^2, \ga^{-1}(2)}$ that satisfies \begin{align}
    \label{eqn:R0BigCond_eig}
        \sup_{R \ge R_0} \dfrac{ \gamma(R) \log(\hat R) g(R) h(R) e^{ \tilde c_1 \tilde{\phi}(R)}   }{g(\hat R) h(\hat R)  } 
        \le \frac 1 {64 \La e^{2 c_1}},
    \end{align}
    assume additionally that for all $R \ge R_0$, we have
	\begin{align}
		\log \abs{\mathbb{K}_R} & \le g(R)
        \label{eqn:g_1_crudeBd_eig} \\ 
          F(\hat R) \log(\gamma(R)) + \tilde G(\hat R) 
          & \le g(\hat{R}), 
        \label{eqn:g_1_cyl_p_eig} \\
		\gamma(R) \log (\hat{R}) g(2 \La \hat{R}) & \le g(\hat{R}) h(\hat{R}) ,
        \label{eqn:h_1_cyl_p_eig}
	\end{align}
    where $F(R)$ is defined in \eqref{eqn:bigF_defn} and
    \begin{equation}
    \label{eqn:tG_defn}
        \begin{aligned}
            \tilde G(R) &:= \sqrt \la R + \tilde \phi(R).
        \end{aligned}
    \end{equation}
    If $u \in H^1_{\loc}(\cyl)$ is a nontrivial solution of \eqref{eqn:eig_eqn0} in $\T^d \times \R^m$ that is normalized so that 
	\begin{align}
    \label{eqn:u_normalized_eig}
		\norm{u}_{L^2(\mathbb{K}_1)} = 1,
	\end{align}
	and bounded in the sense that
	\begin{align}
    \label{eqn:u_bdd_cyl_p_eig}
		\abs{u(\te, x)} \le \exp \set{ c_g \brac{g(\abs{x}) + 1} } \quad
        \text{for all } (\te, x) \in \T^d \times \R^m,
	\end{align}
    then there exists $c_3(d, m, \La, L, f, \ga, g, h, R_0, c_g) > 0$ so whenever $\abs{x_0} \ge R_0$, it holds that
	\begin{align}
    \label{eqn:indConclusion_eig}
		  \norm{u}_{L^2(\mathbb{K}_1(x_0))} 
          \ge  \exp \set{ -c_3 g(\abs{x_0}) h(\abs{x_0}) e^{\tilde c_1 \tilde{\phi}(\abs{x_0})}}.
	\end{align}
\end{prop}

\begin{proof}
	
    By \eqref{eqn:vanishing_eig}, we can choose $R_0 \gg 2 + \gamma(1)$ so that $R_0 \ge \max\set{6 \La, \La^2, \ga^{-1}(2)}$ and \eqref{eqn:R0BigCond_eig} holds.
    Then, inductively define $R_{k} \coloneqq \hat{R}_{k-1} = R_{k-1} \ga(R_{k-1})$ for $k \in \N$.
       
	We claim that there exists $c_3 > 1$ so that for each $k \in \N \cup \set{0}$ and each $x_k \in \R^m$ with $\abs{x_k} = R_k$, it holds that
	\begin{align}
    \label{eqn:ind_cyl_p_eig}
		\norm{u}_{L^2(\mathbb{K}_1(x_k))}  
        \ge \exp \set{ -c_3 g(R_k) h(R_k) e^{\tilde c_1 \tilde{\phi}(R_k)}}.
	\end{align}
	
    Let $x_0 \in \R^m$ with $\abs{x_0} = R_0$.
    With $S_0 \coloneqq A^{(2)}(0, x_0)^{1/2}$, we consider the shifted solution, 
    \begin{align*}
    	u_0(\te, x) \coloneqq u(\te, x_0 + S_0 x).
    \end{align*}
    Set $V_0(\te, x) := V(\te, x_0 + S_0 x)$, $W_0(\te, x) := \pr{W^{(1)}(\te, x_0 + S_0 x), S_0^{-1} W^{(2)}(\te, x_0 + S_0 x)}$, where $W^{(1)}$ is $\C^d$-valued and $W^{(2)}$ is $\C^m$-valued, and note that $\norm{V_0}_\iny \le M$ and $\norm{W_0}_\iny \le \La^{1/ 2} K$.
    With
    \begin{align*}
        A_0 (\te, x) \coloneqq \begin{pmatrix}
    		A^{(1)}(\te, x_0 + S_0 x) & 0 \\
    		0 & S_0^{-1} A^{(2)}(\te, x_0 + S_0 x) S_0^{-1}
    	\end{pmatrix},
    \end{align*}
    we see that 
    \begin{align*}
    	-\divv(A_0 \nabla u_0) + W_0 \cdot \nabla u_0 + V_ 0 u_0 = \la u_0 .
    \end{align*}
    Since $A \in \mathcal{M}(d,m,\La)$, then $A_0 \in \mathcal{M}(d,m,\La^2)$.
    Moreover, $A_0^{(2)}(0,0) = I$ and $\abs{\nabla A_0} \le \Lambda \abs{\nabla A} \le \Lambda L$ a.e.
    Thus, with $\eta(r) = \Lambda L$ and $\phi(r) = \Lambda L r$, we can apply Corollary \ref{cor:3cylinder} with the radii $r_1 = \Lambda^{-1/2}$, $r_2 = \Lambda^{1/2} (R_0 + 1)$, $\sigma = 2 \frac{R_0}{R_0 + 1} < 2$, and $r_3 = R = 3\Lambda^{1/2} R_0$, to obtain
    \begin{equation*}
  	   \begin{aligned}
	   \norm{u_0}_{L^2 (\mathbb{K}_{r_2})}
        &\le \exp \set{ \tilde c_2 \brac{F(R_0) + G_0(R_0)}  } 
        \norm{u_0}_{L^2 (\mathbb{K}_{r_1})}^{\kappa_0}
        \norm{u_0}_{L^2 (\mathbb{K}_{r_3})}^{1-\kappa_0},  
	    \end{aligned}
	\end{equation*}
    where $\tilde c_2 := 9 \La^2 \brac{c_2(d, m, \La^2) + \log 3} > 0$, $G_0(R) \coloneqq \pr{\sqrt \la + L} R$, and $\kappa_0 \in (0,1)$ is the exponent defined by 
	\begin{align*}
    	\kappa_0^{-1} 
        \coloneqq \dfrac{ e^{c_1\pr{3 \Lambda^{3/2} L R_0 + 2}}  \log\pr{2 \La R_0}  }{\log\pr{3/2}} + 1,
	\end{align*}
	with $c_1(d, m+1, \La^2) > 0$.
    Since $B_1 \su x_0 + S_0 B_{r_2}$, then
    \begin{align*}
        \int_{B_{r_2}} \abs{u_0(\cdot, x)}^2 dx
        &= \int_{B_{r_2}} \abs{u(\cdot, x_0 + S_0 x)}^2 dx
        = \abs{S_0^{-1}} \int_{x_0 + S_0 B_{r_2}} \abs{u(\cdot, y)}^2 dy \\
        &\ge \La^{-m/2}\int_{B_{1}(0)} \abs{u(\cdot, y)}^2 dy,
    \end{align*}
    and the normalization condition in \eqref{eqn:u_normalized_eig} shows that $\disp 1 \le \norm{u}_{L^2(\mathbb{K}_{1})} \le \La^{m/ 4} \norm{u_0}_{L^2 (\mathbb{K}_{r_2})}$.
    Similarly, as $x_0 + S_0 B_{r_1} \su B_1(x_0)$, then
    \begin{equation}
    \label{eqn:littleBallEst_eig}
        \begin{aligned}
        \int_{B_{r_1}} \abs{u_0(\cdot, x)}^2 dx
        &= \abs{S_0^{-1}}\int_{x_0 + S_0 B_{r_1}} \abs{u(\cdot, y)}^2 dy 
        \le \La^{m/ 2} \int_{B_1(x_0)}\abs{u(\cdot, y)}^2 dy,   
        \end{aligned}
    \end{equation}
     which implies that $\norm{u_0}_{L^2 (\mathbb{K}_{r_1})} \le \La^{m/ 4} \norm{u}_{L^2 (\mathbb{K}_{1}(x_0))}$.
    Finally, because $x_0 + S_0 B_{r_3} \su B_{4 \La R_0}$, then the bound in \eqref{eqn:u_bdd_cyl_p_eig} shows that
    \begin{align*}
        \norm{u_0}_{L^2(\mathbb{K}_{r_3})} 
        &\le \abs{\mathbb{K}_{4 \La R_0}}^{1/2} \norm{u}_{L^\iny(\mathbb{K}_{4\La R_0})}
        \le \exp \set{ c_g \brac{g(4 \Lambda R_0) + 1} + \tfrac 1 2 \log \abs{\mathbb{K}_{4 \La R_0}} } \\
        &\le \exp \set{ \pr{2 c_g + \tfrac 1 2} g(4 \Lambda R_0)},
    \end{align*}
    where we have used \eqref{eqn:g_1_crudeBd_eig}.
    Putting these observations together and simplifying yields
    \begin{align*}
        \norm{u}_{L^2 (\mathbb{K}_{1}(x_0))}
        &\ge \exp \set{- e^{c_4 \pr{L R_0 +1}} \brac{ F(R_0)  + G_0(R_0)  + g(4 \Lambda R_0)} },
    \end{align*}
    where $c_4(d, m, \Lambda, c_g) > 0$. 
    If we choose $c_3 \ge \be(R_0)$, where
    \begin{equation}
        \label{eqn:c3bcCond_eig}
        \be(R) := e^{c_4 \pr{L R + 1} - \tilde c_1 \tilde{\phi}(R)}\frac{F(R)  + G_0(R)  + g(4 \Lambda R)}{g(R) h(R)},
    \end{equation}
    then the base case $k = 0$ of \eqref{eqn:ind_cyl_p_eig} holds.

    \vspace{1em}
    
 	We now assume that \eqref{eqn:ind_cyl_p_eig} holds for $k-1 \in \N \cup \set{0}$ and prove the claim for $k \in \N$. 
    To this end, we fix any $x_k \in \partial B_{R_k}$, and define $S_k \coloneqq A^{(2)}(0, x_k)^{1/2}$. 
    As in the base case, we define $u_k(\te, x) = u(\te, x_k + S_k x)$, $V_k(\te, x) = V(\te, x_k + S_k x)$, $W_k(\te, x) = \pr{W^{(1)}(\te, x_k + S_k x), S_k^{-1} W^{(2)}(\te, x_k + S_k x)}$, and 
    $$A_k(\te, x) = \begin{bmatrix}
        A^{(1)}(\te, x_k + S_k x) & 0 \\ 0 & S_k^{-1} A^{(2)}(\te, x_k + S_k x) S_k^{-1}.
    \end{bmatrix}$$ 
    We have $\norm{V_k}_\iny \le M$, $\norm{W_k}_\iny \le \La^{1/2} K$, $A_k \in \mathcal{M}(d,m,\La^2)$, and $A_k^{(2)}(0,0) = I$.
    
    Next, we introduce $\rho : [0, R_k] \to [0, \La^{1/2} R_k]$ defined as 
    \begin{align*}
		\rho(r) & \coloneqq \sup \,  \{  s  > 0 \, : \, (x_k + S_k B_s) \cap B_{r}  = \emptyset  \}
	\end{align*}
    with $\rho(R_k) = 0$.
    The inverse function $\rho^{-1} : \text{Im}(\rho) \to [0, R_k]$ is given by 
	\begin{align*}
		\rho^{-1}(r) & \coloneqq \inf  \, \{  \abs{x}  \, : \, x \in x_k + S_k B_r \}.
	\end{align*}
	It is straight-forward to check that indeed $\rho$ and $\rho^{-1}$ are inverses, and moreover, that
	\begin{align}
    \label{eqn:rho_prime_eig}
	\Lambda^{-1/2}	\le - \rho'(r) \le \La^{1/2},  \, \, \, \La^{-1/2} \le -(\rho^{-1})'(r) \le \Lambda^{1/2}.
	\end{align}
	Define $R = \rho\pr{\tfrac 1 2 R_{k-1}}$, and note that
    \begin{align*}
    	-\divv(A_k \nabla u_k) + W_k \cdot \nabla u_k + V_k u_k = \la u_k \text{ in } \mathbb{K}_R.
    \end{align*}
    From the second part of \eqref{eqn:eta_cyl_p}, we see that 
	\begin{equation*}
    \begin{split}
        \phi_k(R) 
        & \coloneqq \int_0^R \norm{ \nabla A_k  }_{L^\infty(\mathbb{K}_r)} \, dr
        \le \Lambda \int_{\rho(R_k)}^{\rho\pr{\frac 1 2 R_{k-1}}} \norm{\nabla A}_{L^\infty \left (  \T^d \times  \left( \R^m \setminus B_{\rho^{-1}(r) }   \right)  \right )} \, dr \\
		& \le \Lambda  \int_{\rho(R_k)}^{\rho\pr{\frac 1 2 R_{k-1}}} f( \rho^{-1}(r)   ) \, dr 
        \le \Lambda^{3/2} \int_{\frac 1 2 R_{k-1}}^{R_k} f(s) \, ds 
        \le \Lambda^{3/2} \tilde{\phi}(R_k),
    \end{split}
	\end{equation*}
    where we recall \eqref{eqn:phi_tilde_cyl_p}.

    With $r_1 = \Lambda^{-1/2}$, $r_2 = \rho(R_{k-1}) + \Lambda^{1/2}$, $r_3 = R \le \La^{1/2} R_k$, and $\sigma \in (1, \La)$ chosen so that $\disp \sigma r_2 = \frac{r_2 + r_3} 2$, we may apply Corollary \ref{cor:3cylinder} to $u_k$ to obtain the following three-cylinder inequality:
    \begin{equation}
    \label{eqn:threCylk_eig}
       \begin{aligned}
	   \norm{u_k}_{L^2 (\mathbb{K}_{r_2})}
        &\le \exp \set{ \La^2 \brac{c_2 + \log\pr{\frac \si {\si - 1}} }F(R_k)  + \La^{3/2} c_2 \tilde G(R_k)  } \\
        &\times \norm{u_k}_{L^2 (\mathbb{K}_{r_1})}^{\kappa_k}
        \norm{u_k}_{L^2 (\mathbb{K}_{r_3})}^{1-\kappa_k},  
	    \end{aligned}
	\end{equation}
    where $\kappa_k \in (0,1)$ is given by
	\begin{align*}
		\kappa_k^{-1} \coloneqq 
		e^{c_1 \pr{\Lambda^{3/2} \tilde{\phi}(R_k) + 2}} \dfrac{   \log\pr{\La^{1/2} \sigma r_2}}{\log\pr{1 + \frac{r_3 - r_2}{r_3 + r_2 }}} + 1.
	\end{align*}
    By the estimates in \eqref{eqn:rho_prime_eig}, one readily checks that 
    \begin{align*}
        r_2 &= \Lambda^{1/2} + \int_{R_{k-1}}^{R_k} - \rho'(r) dr
        \ge \La^{-1/2} \pr{R_k - R_{k-1}}
        = \La^{-1/2} \pr{\ga(R_{k-1}) - 1} R_{k-1}
        \\
        r_3 &= \int_{\tfrac 1 2 R_{k-1}}^{R_k} - \rho'(r) dr
        \ge \La^{-1/2} \pr{R_k - \tfrac 1 2 R_{k-1}}
        \ge \tfrac 1 2 \La^{-1/2} R_{k}
        \\
    	r_3 - r_2 
        &= - \Lambda^{1/2} +  \int_{ R_{k-1} / 2}^{ R_{k-1}  }  - \rho'(r)\, dr 
        \ge \tfrac 1 2 \La^{-1/2} R_{k-1} - \La^{1/2} 
        \ge \tfrac 1 3 \La^{-1/2} R_{k-1} \\
        r_3 - r_2 
        &\le  \int_{ R_{k-1} / 2}^{ R_{k-1}  }  - \rho'(r)\, dr 
        \le \tfrac 1 2 \La^{1/2} R_{k-1},
    \end{align*}
   where the third bound uses that $R_0 \ge 6 \La$.
    Since $\log(1+ \epsilon) \ge \eps - \frac {\eps^2} 2 = \eps\pr{1 - \frac \eps 2} \ge \frac 1 2 \eps$ for $\epsilon \in (0, 1)$, then
    \begin{equation}
    \label{eqn:kappakBound_eig} 
        \begin{aligned}
         \kappa_k^{-1} 
        &\le 4 e^{c_1 \pr{\Lambda^{3/2} \tilde{\phi}(R_k) + 2}} \log\pr{\La R_k} \frac{r_3 + r_2 }{r_3 - r_2}
        \le 24 \La e^{\tilde c_1 \tilde{\phi}(R_k) + 2c_1 } \frac{R_k \log\pr{\La R_k} }{R_{k-1}} \\
        &\le 32 \La e^{2c_1 } \ga(R_{k-1})\log\pr{R_k}e^{\tilde c_1 \tilde{\phi}(R_k)},
        \end{aligned}
    \end{equation}
    where we used that $R_0 \ge \La^2$.
    Because $\disp \si = \frac{r_2 + r_3}{2 r_2} = 1 + \frac{r_3 - r_2}{2 r_2} \le \frac 1 2 + \frac \La 4 < \La$, then
    $$\frac{\si}{\si - 1} 
    = \frac{r_3 + r_2}{r_3 - r_2} 
    \le 6 \La \ga(R_{k-1}),$$ 
    and then with $c_5 := 1 + \log_2(6\La)$ and using that $\ga(R_{k-1}) \ge 2$, we get
    \begin{align*}
    	\abs{\log\pr{\frac{\sigma}{\sigma-1}} } 
    	\le c_5 \log(\gamma(R_{k-1})).
    \end{align*}
    For $y_{k-1} \in \R^m$ defined by $y_{k-1} = \frac{R_{k-1}}{R_k} x_k$, since $B_1(y_{k-1}) \su x_k + S_k B_{r_2}$, then 
    \begin{align*}
        \int_{B_{r_2}} \abs{u_k(\cdot, x)}^2 dx
        &= \abs{S_k^{-1}} \int_{x_k + S_k B_{r_2}} \abs{u(\cdot, y)}^2 dy
        \ge \La^{-m/2} \int_{B_{1}(y_{k-1})} \abs{u(\cdot, y)}^2 dy
    \end{align*}
    and the inductive hypothesis described by \eqref{eqn:ind_cyl_p_eig} shows that
    \begin{align*}
        \norm{u_k}_{L^2 (\mathbb{K}_{r_2})}
        &\ge \La^{-m/4} \exp \set{ -c_3 g(R_{k-1}) h(R_{k-1}) e^{\tilde c_1 \tilde{\phi}(R_{k-1})}}.
    \end{align*}
    As in \eqref{eqn:littleBallEst_eig}, $\norm{u_k}_{L^2 (\mathbb{K}_{r_1})} \le \La^{m/ 4} \norm{u}_{L^2 (\mathbb{K}_{1}(x_k))}$.
    On the other hand, using that $x_k + S_k B_R \su B_{2 \La R_k}$ and the boundedness assumption in \eqref{eqn:u_bdd_cyl_p_eig}, we get 
    \begin{align*}
    	\norm{u_k}_{L^2\pr{\mathbb{K}_{r_3}}}
    	\le \exp \set{ c_g \brac{g(2 \Lambda R_k) + 1} + \tfrac 1 2 \log \abs{\mathbb{K}_{2 \La R_k}} } 
    	\le \exp \left \{ \pr{2 c_g + \tfrac 1 2} g(2\Lambda R_k)  \right \},
    \end{align*}
    where \eqref{eqn:g_1_crudeBd_eig} was applied.
    Returning to the three-cylinder inequality in \eqref{eqn:threCylk_eig}, we take logarithms and use the bounds from above to get 
    \begin{equation*}
       \begin{aligned}
	   \kappa_k \log \pr{\norm{u}_{L^2 (\mathbb{K}_{1}(x_k))}^{-1}}
        &\le c_3 g(R_{k-1}) h(R_{k-1}) e^{\tilde c_1 \tilde{\phi}(R_{k-1})}
        +\La^2 \brac{c_2 + c_5 \log(\gamma(R_{k-1})) }F(R_k) \\
        &+ \La^{3/2} c_2 \tilde G(R_k)  
        + \pr{2 c_g + \tfrac 1 2} g(2\Lambda R_k) 
        + \frac m 4 \log \La \\
        &\le c_3 g(R_{k-1}) h(R_{k-1}) e^{\tilde c_1 \tilde{\phi}(R_{k-1})}
        + \tilde c_2 \brac{ F(R_k) \log(\gamma(R_{k-1})) + \tilde G(R_k)} \\
        &+ \pr{2 c_g + \tfrac 1 2} g(2\Lambda R_k) 
	    \end{aligned}
	\end{equation*}
    where $\tilde c_2 := \max\set{c_5 \La^2 + \frac{c_2 \La^2}{\log 2 } + \frac{ m \log \La}{4 \log 2 }, \Lambda^{3/2} c_2}$.
    With $\tilde c_g =\tilde c_2  + 2 c_g + \tfrac 1 2$, we use the bound in \eqref{eqn:g_1_cyl_p_eig} and the monotonicity of $g$ to get 
    \begin{equation*}
       \begin{aligned}
	   \kappa_k \log \pr{\norm{u}_{L^2 (\mathbb{K}_{1}(x_k))}^{-1}} 
        &\le c_3 g(R_{k-1}) h(R_{k-1}) e^{\tilde c_1 \tilde{\phi}(R_{k-1})}
        + \tilde c_g g(2\Lambda R_k).
	    \end{aligned}
	\end{equation*}
    Using the bound in \eqref{eqn:kappakBound_eig} shows that
    \begin{equation*}
       \begin{aligned}
	   \log &\pr{\norm{u}_{L^2 (\mathbb{K}_{1}(x_k))}^{-1}} 
        \le \kappa_k ^{-1}\brac{c_3 g(R_{k-1}) h(R_{k-1}) e^{\tilde c_1 \tilde{\phi}(R_{k-1})}
        + \tilde c_g g(2\Lambda R_k)} \\
        &\le 32 \La e^{2c_1} \frac{\ga(R_{k-1}) \log\pr{R_k}  g(R_{k-1}) h(R_{k-1}) e^{\tilde c_1 \tilde{\phi}(R_{k-1})} }{g(R_k) h(R_k)} c_3 g(R_k) h(R_k) e^{\tilde c_1 \tilde{\phi}(R_k) } \\
        &\quad + 32 \La \tilde c_g e^{2c_1 } \ga(R_{k-1}) \log\pr{R_k}  g(2\Lambda R_k) e^{\tilde c_1 \tilde{\phi}(R_k) } \\
        &\le \pr{\frac{c_3}{2} + 32 \La \tilde c_g e^{2c_1 } } g(R_k) h(R_k) e^{\tilde c_1 \tilde{\phi}(R_k) },
	    \end{aligned}
	\end{equation*}
    where we have used \eqref{eqn:R0BigCond_eig} and \eqref{eqn:h_1_cyl_p_eig}. 
    If $c_3 \ge 64 \La \tilde c_g e^{2c_1 }$, then \eqref{eqn:ind_cyl_p_eig} holds.
    In particular, if $c_3 \ge \max \set{\be(R_0), 64 \La \tilde c_g e^{2c_1 }}$ (see \eqref{eqn:c3bcCond_eig} for the base case), then \eqref{eqn:ind_cyl_p_eig} holds for all $k \in \N \cup \set{0}$, establishing the claim.
    
    For general $\abs{x} \ge R_0$, choose $k_0 \in \N$ so that $R_{k_0-1} \le \abs{x} < R_{k_0}$. 
    Take $\rho_0 \in [R_0, R_1]$ so that with $\set{\rho_k}_{k \in \N}$ defined recursively by $\rho_{k} = \hat{\rho}_{k-1}$, we get $\abs{x} = \rho_{k_0}$.
    With $\disp c_3 = \max\set{ \sup_{R \in [R_0, R_1]} \be(R), 64 \La \tilde c_g e^{2c_1 }}$, the conclusion described by \eqref{eqn:indConclusion_eig} holds. 
 \end{proof}

\subsection{Unique continuation for polynomially decaying gradients}

Here we show the first two applications of Proposition \ref{prop:decay_cyl_p_eig}.
These two results address the unique continuation at infinity for solutions $u$ of \eqref{eqn:eig_eqn0} when the matrix $A$ satisfies \eqref{eqn:A_decay_expl} for some $\tau \in [0,1]$. 
The reader should note that the behavior of solutions at infinity is much different for $\tau = 1$ than it is for $\tau \in [0,1)$. 
In both cases, given the decay function $f$, we need to appropriately choose the functions $\gamma$, $g$, and $h$ and check the compatibility conditions in Proposition \ref{prop:decay_cyl_p_eig}. 
We first deal with $\tau \in [0,1)$.
    
\begin{thm}[Rate of decay at infinity for $\tau \in [0,1)$]
\label{thm:tau_betw}
    For $A \in \mathcal{M}(d,m,\La)$ and $\tau \in [0, 1)$, assume that $A$ satisfies the following condition:
	\begin{align*}
    	&\abs{\gr A(\te, x)}\le \frac{L}{(1 + \abs{x})^\tau}  \,\,\,\, \text{ for a.e. } (\te, x) \in \T^d \times \R^m.
	\end{align*}
    Let $W$ and $V$ satisfy \eqref{eqn:pot_bounds} with $K, M \in \brac{0,1}$ and take $\la \in \C$.
    Let $u \in H^1_{\loc}(\T^d\times \R^m)$ be a nontrivial solution of \eqref{eqn:eig_eqn0} in $\T^d \times \R^m$ that is normalized so that \eqref{eqn:u_normalized_eig} holds and bounded in the sense that
	\begin{equation}
	\label{uBound2Decay}
		\abs{u(\te, x)} \le \exp\set{\exp\pr{c_0 \abs{x}^{1 - \tau}} + c_0 }\quad
        \text{for all } (\te, x) \in \T^d \times \R^m.
	\end{equation} 
    Then there exists $R_0(d, m, \La, \la, \tau, L, c_0) > 1$ and $c_3(d, m, \La, \la, \tau, L, c_0) > 0$ so that whenever $\abs{x_0} \ge R_0$, it holds that
    \begin{equation}
    \label{uEstSmallTau}
		\norm{u}_{L^2(\mathbb{K}_1(x_0))} \ge \exp\set{- \exp\pr{c_3 \abs{x_0}^{1-\tau}}}.
	\end{equation}	
\end{thm}  
	
\begin{proof}
	To apply Proposition \ref{prop:decay_cyl_p_eig}, we first need to choose $\gamma$, $g$, and $h$ and check that conditions \eqref{eqn:vanishing_eig} -- \eqref{eqn:h_1_cyl_p_eig} and \eqref{eqn:u_bdd_cyl_p_eig} hold. 
    If we let $g(R) = e^{c_6 R^{1-\tau}}$ for some $c_6 \ge c_0$ and set $c_g = c_0$, then \eqref{uBound2Decay} implies \eqref{eqn:u_bdd_cyl_p_eig}.
    The condition \eqref{eqn:g_1_crudeBd_eig} holds for all $R \ge 4$ if $c_6(d, m) \gg 1$.
    Let $\gamma(R) = R$ so that $\hat{R} = R^2$ and $\check{R} = \sqrt{R}$. 
    With the notation from \eqref{eqn:phi_tilde_cyl_p}, we have for all $R \ge 4$ that 
	\begin{align}
    \label{eqn:phi_tau1}
		\tilde{\phi}(R) 
        &= \int_{\sqrt{R}/2}^{R} L (1 + s)^{-\tau} \, ds
        \le c_\tau L R^{1-\tau}.
	\end{align}  
    With $c_6 \ge 16^{\tau - 1} \brac{4 \log 4 + \log \pr{\log 4} + \log\pr{5 + \frac{\sqrt \la + c_\tau L}{16 \log 4}}}$, we see that 
    \begin{align*}
        F(\hat R) \log(\ga(R)) + \tilde G(\hat R)
        &\le \pr{1 + R^{8/3} + R^4} \log R
        + \sqrt \la R^2 + c_\tau L R^{2-2\tau} \\
        &\le \exp\set{4 \log R + \log(\log R) + \log\pr{5 + \frac{\sqrt \la + c_\tau L}{16 \log 4}}} \\
        &\le \exp\set{c_6 R^{2(1-\tau)}}
        = g(\hat{R}),
    \end{align*}
    showing that \eqref{eqn:g_1_cyl_p_eig} holds whenever $R \ge 4$.
	Since
    \begin{align*}
        \gamma(R) \log (\hat{R}) g(2 \La \hat{R}) 
        &= 2R \log (R) e^{c_6(2\Lambda )^{1 - \tau} R^{2(1-\tau)}} \\
        &\le \exp\set{2 \Lambda c_6 R^{2(1-\tau)} + \log R + \log(\log R)  + \log 2} \\
        &\le \exp\set{2 \Lambda c_6 R^{2(1-\tau)}} g(\hat R)
    \end{align*}
    by the previous inequality, then with $h(R) = e^{2 \La c_6 R^{1-\tau}}$, \eqref{eqn:h_1_cyl_p_eig} also holds whenever $R \ge 4$.
    Since $\tau < 1$, then
    \begin{align*}
		\lim_{R \ra \infty} \exp\set{(1 + 2\Lambda) c_6 \pr{R^{1-\tau} - R^{2(1-\tau)}} + \tilde c_1 c_\tau L R^{1-\tau} + \log(2R) + \log( \log R)} =  0,
	\end{align*}
    and the asymptotic decay described by \eqref{eqn:vanishing_eig} holds.
    In fact, if $R_0 \ge \max\set{6 \La, \La^2, 4}$ satisfies
    \begin{align*}
		\exp\set{\brac{(1+2\Lambda) c_6 + \tilde c_1 c_\tau L }R_0^{1-\tau} + \log R_0 + \log\pr{\log R_0} }
        \le \frac {\exp\set{(1+2\Lambda) c_6 R_0^{2(1-\tau)}}}{128 \Lambda e^{2 c_1}},
	\end{align*}
    then \eqref{eqn:R0BigCond_eig} holds and we see that $R_0$ depends on $\La$, $c_1$, $\tilde c_1$, $c_6$, $\tau$, and $L$.

    Proposition \ref{prop:decay_cyl_p_eig} shows that there exists $c_3(d, m, \La, L, \tau, c_6, R_0, c_0) > 0$ so that whenever $\abs{x_0} \ge R_0$, it holds that
	\begin{align*}
    	\norm{u}_{L^2(\mathbb{K}_1(x_0))} 
        \ge  \exp \set{ -\exp\brac{\pr{(1 + 2\Lambda) c_6 + \tilde c_1 c_\tau L} \abs{x_0}^{1 - \tau} + \log c_3}},
	\end{align*}
    where we have used estimate \eqref{eqn:phi_tau1}.
    After simplifications, we reach the conclusion described by \eqref{uEstSmallTau}. 
\end{proof}
	
As a second application of Proposition \ref{prop:decay_cyl_p_eig}, we consider the case when $\tau =1$. 
This theorem has three conclusions, depending on the constant $L_2$ below. 
In Section \ref{sec:examples}, we show by example that such a dependence on $L_2$ is necessary (and in some sense sharp).	

\begin{thm}[Rate of decay at infinity for $\tau = 1$]
\label{thm:tau=1}
    For $A \in \mathcal{M}(d,m,\La)$ and $R_1 > 1$, assume that $A$ satisfies the following conditions:
	\begin{equation*}
		\begin{split}
			\abs{\nabla A(\te, x)} & \le L_1 \,\,\,\, \text{ for a.e. } (\te, x) \in \T^d \times \R^m, \\
			\abs{ \nabla A(\te, x)  } & \le \frac{L_2}{\abs{x}} \,\,\, \,  \text{ for a.e. } (\te, x) \in \T^d \times \R^m, \, \abs{x} > R_1.
		\end{split}
	\end{equation*}
    Let $W$ and $V$ satisfy \eqref{eqn:pot_bounds} with $K, M \in \brac{0, 1}$ and define $q$ as 
    \begin{align}
    \label{eqn:q-def}
    	q := \begin{cases}
    		2 & K \ne 0 \\
    		\frac 4 3 & K = 0, M \ne 0  \\
    		1 & K = 0, M = 0
    	\end{cases}.
    \end{align}
    With $\tilde c_1(d, m, \Lambda) > 0$ from Proposition \ref{prop:decay_cyl_p_eig}, define 
    \begin{equation}
        \label{eqn:p-def}
        p := 1 + \tilde c_1 L_2.
    \end{equation}
    For $\la \in \C$, let $u \in H^1_{\loc}(\T^d\times \R^m)$ be a nontrivial solution of \eqref{eqn:eig_eqn0} in $\T^d \times \R^m$ that is normalized so that \eqref{eqn:u_normalized_eig} holds and bounded in the sense that
    \begin{align}
    \label{eqn:u_bdd_2}
		\abs{u(\te, x)} \le  \exp\set{ c_0 \pr{\abs{x}^{q} + 1}} \quad
        \text{ for all } (\te, x) \in \T^d \times \R^m.
    \end{align}
    For any $\eps > 0$, there exists a distance $R_0(d, m, \Lambda, \la, L_2, R_1, p, q, \eps) > 1$ and a constant $c_3(d, m, \Lambda, \la, L_1, L_2, R_1, p, q, \eps, R_0, c_0) > 1$  so that 
    \begin{align}
    \label{eqn:L2Conc}
        \norm{u}_{L^2(\mathbb{K}_1(x_0))} 
        \ge \exp\set{-c_3 \abs{x_0}^{s} E(\abs{x_0})},
    \end{align}
    where $s := \max\set{p, q}$ and
    \begin{align}
    \label{eqn:bigH_defn}
        E(R) := \begin{cases}
            \exp\brac{\eps \pr{\log R}^{\frac {1+\eps} 2}} & p > q \\
            \exp\brac{\pr{p + \eps} \pr{\log R}^{\frac{1+\eps}{2}}} & p = q \\
            \pr{\log R}^{\frac{q}{q - p} + \eps} & p < q
        \end{cases}.
    \end{align}
\end{thm}

 \begin{rmk}
     When $p < q$, $R_0$ may be independent of $\la$.
 \end{rmk}

\begin{proof}
	In each case, we appeal to Proposition \ref{prop:decay_cyl_p_eig}.
    Therefore, we need to choose $\gamma$, $g$, and $h$ and check that conditions \eqref{eqn:vanishing_eig} -- \eqref{eqn:h_1_cyl_p_eig} and \eqref{eqn:u_bdd_cyl_p_eig} hold. 
    Since $f(s) = L_2 s^{-1}$ whenever $s > R_1$, recalling \eqref{eqn:phi_tilde_cyl_p}, we have that for all $R > 2 R_1 \ga(2R_1)$,  
	\begin{align}
    \label{eqn:phi_tau_eig}
		\tilde{\phi}(R) 
        = \int_{\frac 1 2 \check{R}}^{R} L_2 s^{-1} ds
        = L_2 \log\pr{2 \ga(\check{R})}.
	\end{align}
     
    First, we consider the case of relatively large $L_2$, i.e., $p := 1 + \tilde c_1 L_2 > q$. 
    Let $\gamma(R) = \exp\brac{\pr{\log R}^{\frac {1+\eps} 2}} = R^{\pr{\log R}^{-\frac {1-\eps} 2}}$ so that $\hat{R} = R \ga(R) = R^{1 + \pr{\log R}^{-\frac {1-\eps} 2}}$. 
    From \eqref{eqn:phi_tau_eig}, whenever $R \ge \max\set{2 R_1 \ga(2R_1), \exp\brac{\pr{\log 2}^{\frac{2}{1 + \eps}}}}$,
    \begin{equation}
    \label{eqn:phi_tau_eig+}
        \tilde{\phi}(R) 
        = L_2 \log\pr{2 \exp\brac{\pr{\log \check R}^{\frac {1+\eps} 2}}}
        \le 2 L_2 \pr{\log \check R}^{\frac {1+\eps} 2}.
    \end{equation}
    Choose $R_0(d, m, \Lambda, \la, R_1, L_2, p, q, \eps, c_1, \tilde c_1) > \max\set{6 \La, \La^2, \exp\brac{\pr{\log 2}^{\frac{2}{1 + \eps}}}, 2 R_1 \ga(2R_1)}$ so that 
    \begin{equation}
    \label{eqn:cond1_+}
        \log \abs{\mathbb{K}_{R_0}} \le \frac{R_0^p}{\ga(\check R_0)^{p-\eps} \log R_0 } ,
    \end{equation}
    \begin{equation}
    \label{eqn:cond2_+}
        R_0^{p-q} 
        \ge \pr{3 + \frac{\sqrt \la}{\hat R^{q-1}} + \frac{2 L_2}{\hat R^q}} \ga(\check R_0)^{p-\eps} \pr{\log R_0}^{\frac{3+\eps}{2}},
    \end{equation}
    and
    \begin{align}
    \label{eqn:van_cond_+}
       \frac{\eps \pr{1 + \eps}\pr{3 + \eps}}{8} \pr{\log R_0}^{\eps} - 4 \log \log R_0 
       \ge 2 c_1 + \log \La + \pr{6 + \tilde c_1 L_2} \log 2.
    \end{align}
    With $\disp g(R) = \frac{R^p}{\ga(\check R)^{p-\eps} \log R }$, the condition in \eqref{eqn:cond1_+} shows that \eqref{eqn:g_1_crudeBd_eig} holds for all $R \ge R_0$.
    From \eqref{eqn:bigF_defn}, \eqref{eqn:tG_defn}, \eqref{eqn:q-def}, and \eqref{eqn:phi_tau_eig+}, we have
    \begin{equation}
    \label{eqn:FGCond+}
        \begin{aligned}
        F(\hat R) \log(\gamma(R)) + \tilde G(\hat R)
        &\le 3 \hat R^q \pr{\log R}^{\frac{1 + \eps}{2}} 
        + \sqrt \la \hat R
        + L_2 \log\pr{2 \ga(R)} \\
        &\le \pr{3 + \frac{\sqrt \la}{\hat R^{q-1}} + \frac{2 L_2}{\hat R^q}} \hat R^q \pr{\log R}^{\frac{1 + \eps}{2}} 
        \le g(\hat R),
        \end{aligned}
    \end{equation}
    where we used the condition in \eqref{eqn:cond2_+}.
    We have shown that \eqref{eqn:g_1_cyl_p_eig} holds for all $R \ge R_0$.
    As
    \begin{align*}
        \frac{\gamma(R) \log (\hat{R}) g(2 \La \hat{R})}{g(\hat{R}) }
        &\le \pr{2 \La }^p \gamma(R) \log (\hat{R}),
    \end{align*}
    then \eqref{eqn:h_1_cyl_p_eig} holds for all $R \ge R_0$ if we define $h(R) = (2 \La)^p \gamma(\check R) \log R$.
    To achieve the vanishing condition, we observe that
    \begin{align*}
		\dfrac{ \gamma(R) \log(\hat R) g(R) h(R) e^{ \tilde c_1 \tilde{\phi}(R)}   }{g(\hat R) h(\hat R) }
        &= 2^{\tilde c_1 L_2 } \dfrac{ \log(\hat R) \gamma(\check R)^\eps }{  \ga(R)^\eps },
	\end{align*}
    where \eqref{eqn:p-def} was used to simplify the expression.
    A Taylor expansion shows that
    \begin{equation}
    \label{eqn:gammaComp}
        \begin{aligned}
        \frac{\ga(R)}{\ga(\check R)}
        &= \exp\set{\pr{\log \check R}^{\frac{1+\eps}{2}}  \brac{\pr{1  + \pr{\log \check R}^{-\frac{1-\eps}{2}}}^{\frac{1+\eps}{2}} - 1}} \\
        &\ge \exp\set{ \frac{\pr{1+\eps}\pr{3 + \eps}}{8} \pr{\log \check R}^{\eps}},    
        \end{aligned}
    \end{equation}
    where we have used that $R_0 \ge e$.
    Since $\log \log R \le 2\log \log \check R$, then
    \begin{align*}
		\dfrac{ \gamma(R) \log(\hat R) g(R) h(R) e^{ \tilde c_1 \tilde{\phi}(R)}   }{g(\hat R) h(\hat R) }
        &\le \dfrac{ 2^{\tilde c_1 L_2}}{\exp\set{ \frac{\eps\pr{1+\eps}\pr{3 + \eps}}{8} \pr{\log \check R}^{\eps} - 4\log \log \check R }}  
	\end{align*}
    and then \eqref{eqn:vanishing_eig} is satisfied.
    In fact, the condition in \eqref{eqn:van_cond_+} shows that \eqref{eqn:R0BigCond_eig} holds.
    Finally, with $c_g = c_0$, the bound in \eqref{eqn:u_bdd_2} implies \eqref{eqn:u_bdd_cyl_p_eig} and Proposition \ref{prop:decay_cyl_p_eig} is applicable.
    That is, there exists a constant $c_3(d, m, \Lambda, L_1, L_2, R_1, p, \eps, R_0, c_0) > 0$ so whenever $\abs{x_0} \ge R_0$, it holds that
	\begin{align*}
		  \norm{u}_{L^2(\mathbb{K}_1(x_0))} 
          \ge  \exp \set{ -c_3 2^{\tilde c_1 L_2 } (2 \La)^p \abs{x_0}^{p} e^{\eps \pr{\log \abs{x_0}}^{\frac {1+\eps} 2}}},
	\end{align*}
    which shows that \eqref{eqn:L2Conc} holds with \eqref{eqn:bigH_defn} for $p > q$.

\hspace{0.5em}

    Next we consider the borderline case where $p := 1 + \tilde c_1 L_2 = q$. 
    As in the previous case, we let $\ga(R) = \exp\brac{\pr{\log R}^{\frac{1+\eps}{2}}}$ so that \eqref{eqn:phi_tau_eig+} holds.
    This time, choose $R_0(d, m, \Lambda, \la, R_1, L_2, p, \eps, c_1, \tilde c_1) > \max\set{6 \La, \La^2, \exp\brac{\pr{\log 2}^{\frac{2}{1 + \eps}}}, 2 R_1 \ga(2R_1)}$ so that 
    \begin{align}
    \label{eqn:cond1_=}
        \log \abs{\mathbb{K}_{R_0}} \le R_0^{p} \log R_0, \qquad
        &\pr{\log R_0}^{\frac{1-\eps}2} \ge 3 + \frac{\sqrt \la}{R_0^{p-1}} + \frac{2 L_2}{R_0^p},
    \end{align}
    \begin{align}
    \label{eqn:van_cond2_=}
       \frac{p \pr{1 + \eps}\pr{3 + \eps}}{8} \pr{\log R_0}^{\eps} - 4 \log \log R_0 
       \ge 2 c_1 + \log \La + \pr{6 + \tilde c_1 L_2} \log 2,
    \end{align}
    and
    \begin{align}
    \label{eqn:cond3_=}
       2\log \pr{\log R_0} \le \eps \pr{\log R_0}^{\frac{1+\eps}2}.
    \end{align}
    With $g(R) = R^p \log R$, the first condition in \eqref{eqn:cond1_=} shows that \eqref{eqn:g_1_crudeBd_eig} holds for all $R \ge R_0$.
    As shown in \eqref{eqn:FGCond+}, since $p = q$,
    $$F(\hat R) \log(\gamma(R)) + \tilde G(\hat R) 
    \le \pr{3 + \frac{\sqrt \la}{\hat R^{p-1}} + \frac{2 L_2}{\hat R^p}} \hat R^p \pr{\log R}^{\frac{1 + \eps}{2}}
    \le g(\hat R),$$
    where we have used the second condition in \eqref{eqn:cond1_=}.
    In particular, \eqref{eqn:g_1_cyl_p_eig} holds for all $R \ge R_0$.
    As
    \begin{align*}
        \frac{\gamma(R) \log (\hat{R}) g(2 \La \hat{R})}{g(\hat{R}) }
        &= (2 \La)^p \gamma(R) \log\pr{2 \La \hat R},
    \end{align*}
    then \eqref{eqn:h_1_cyl_p_eig} holds for all $R \ge R_0$ if we define $h(R) = (2 \La)^p \gamma(\check R) \log\pr{2 \La R}$.
    To achieve the vanishing condition, we recall \eqref{eqn:p-def} and observe that
    \begin{align*}
		\dfrac{ \gamma(R) \log(\hat R) g(R) h(R) e^{ \tilde c_1 \tilde{\phi}(R)}   }{g(\hat R) h(\hat R) }
        &= 2^{\tilde c_1 L_2} \log R  \brac{\dfrac{  \gamma(\check R)}{\ga(R)}}^p \frac{\log\pr{2 \La R}}{\log\pr{2 \La \hat R}} \\
        &\le \dfrac{ 2^{\tilde c_1 L_2}}{\exp\set{ \frac{p \pr{1 + \eps}\pr{3+\eps}}{8} \pr{\log \check R}^{\eps} - 4\log \log \check R }} ,
	\end{align*}
    where we have repeated the computations from \eqref{eqn:gammaComp}.
    Then \eqref{eqn:vanishing_eig} is satisfied and the condition in \eqref{eqn:van_cond2_=} shows that \eqref{eqn:R0BigCond_eig} holds.
    Finally, with $c_g = c_0$, the bound in \eqref{eqn:u_bdd_2} implies that \eqref{eqn:u_bdd_cyl_p_eig} holds and then Proposition \ref{prop:decay_cyl_p_eig} is applicable.
    That is, there exists a constant $c_3(d, m, \Lambda, L_1, L_2, R_1, p, \eps, R_0, c_0) > 0$ so whenever $\abs{x_0} \ge R_0$, it holds that
	\begin{align*}
		  \norm{u}_{L^2(\mathbb{K}_1(x_0))} 
          &\ge  \exp \set{ -c_3 (4 \La)^p \abs{x_0}^{p} e^{p \pr{\log \abs{x_0}}^{\frac{1+\eps}{2}}} \pr{\log\abs{x_0}}^2 } \\
          &= \exp \set{ -c_3 (4 \La)^p \abs{x_0}^{p} e^{\pr{p + \eps} \pr{\log \abs{x_0}}^{\frac{1+\eps}{2}}}},
	\end{align*}
    where we have used \eqref{eqn:cond3_=}.
    This establishes \eqref{eqn:L2Conc} with \eqref{eqn:bigH_defn} for $p = q$.
    
    \hspace{0.5em}

    Finally, we consider relatively small $L_2$ for which $p := 1 + \tilde c_1 L_2 < q$.
    Note that in this case, it must hold that $q \ge \frac 4 3$.
    With $\mu := \frac{1}{q - p} + \frac{\eps}{2p}$, set $\gamma(R) = \pr{\log R}^{\mu}$ so that $\hat{R} = R \pr{\log R}^{\mu}$ and \eqref{eqn:phi_tau_eig} shows that whenever $R > \max\set{2 R_1 \brac{\log \pr{2 R_1}}^\mu, \exp\pr{2^{\frac 1 \mu}}}$,  
	\begin{align}
    \label{eqn:phi_tau_eig-}
		\tilde{\phi}(R) 
        = L_2 \log\brac{2\pr{\log \check{R}}^{\mu}}
        \le 2 \mu L_2 \log (\log \check R).
	\end{align}
    Choose $R_0(d, m, \La, R_1, L_2, p, q, \eps, \mu, c_1, \tilde c_1) > \max\set{6 \La, \La^2, \exp\pr{2^{\frac 1 \mu}}, 2 R_1 \brac{\log \pr{2 R_1}}^\mu}$ so that
    \begin{align}
        \label{eqn:cond1_-}
        \log\pr{\log(2\Lambda R_0)} \le 2 \log(\log R_0), \qquad
        \log\pr{\log R_0} \le \pr{\log R_0}^{\frac {\eps} {2}}
    \end{align}
    and
    \begin{align}
    \label{eqn:cond2_-}
        \pr{\log R_0}^{q - p - \frac{1}{\mu}} 
        \ge \dfrac{\pr{2^{6 + \tilde c_1 L_2} \La e^{2 c_1}}^{\frac 1 \mu} }{1 + \mu \frac{\log(\log R_0)}{\log R_0}  },
	\end{align}
    where the definition of $\mu$ and that $q > p$ ensures that $q - p - \frac{1}{\mu} > 0$.
    With some $c_8 > c_0$ to be determined, define $g(R) \ge c_0 \brac{R^q \log\pr{\log\pr{R + e}} + 1}$ for all $R \ge 0$ so that $g(R) = c_8 R^q \log(\log R)$ when $R \ge R_0$. 
    The bound in \eqref{eqn:u_bdd_2} implies \eqref{eqn:u_bdd_cyl_p_eig} with $c_g = c_0$.
    If $\disp c_8 \ge \frac{\log \abs{\mathbb{K}_{R_0}}}{R_0^q \log(\log R_0)}$, then \eqref{eqn:g_1_crudeBd_eig} holds for all $R \ge R_0$.
    Recall \eqref{eqn:bigF_defn}, \eqref{eqn:tG_defn}, and \eqref{eqn:q-def}, then use \eqref{eqn:phi_tau_eig-} to get
    \begin{align*}
        F(\hat R) \log(\ga(R)) + \tilde G(\hat R)
        &\le 3 \mu \hat R^q \log\pr{\log R} + \sqrt \la \hat R +  2 \mu L_2 \log (\log R) \\
        &\le \pr{3 \mu + \frac{\sqrt \la}{\hat R^{q-1}} + \frac{2 \mu L_2}{\hat R^q} }\hat R^q \log\pr{\log \hat R}.
    \end{align*}
    On the other hand,
    \begin{align}
    \label{eqn:ghatrLower}
        g(\hat{R})
        &= c_8 \brac{ R \pr{\log R}^{\mu}}^q \log(\log \brac{R \pr{\log R}^{\mu}})
        \ge c_8 R^q \pr{\log R}^{q\mu} \log(\log R).
    \end{align}
    Therefore, with $c_8 \ge 3 \mu + \sqrt \la R_0^{1-q} + 2 \mu L_2R_0^{-q}$, \eqref{eqn:g_1_cyl_p_eig} holds for all $R \ge R_0$.
    Using the first condition in \eqref{eqn:cond1_-}, for $R \ge R_0$,
    \begin{align*}
        \frac{\gamma(R) \log (\hat{R}) g(2 \La \hat{R})}{g(\hat{R})} 
        &= \pr{2\Lambda}^q \frac{ \pr{\log R}^{\mu} \log (\hat R)   \log(\log(2\Lambda \hat{R}))}{\log(\log(\hat{R}))} 
        \le 2 (2\Lambda)^{q} \brac{\log(\hat{R})}^{\mu + 1}.
    \end{align*}
    Thus, if we choose $h(R) = 2(2\Lambda)^{q} \pr{\log R}^{\mu + 1}$, \eqref{eqn:h_1_cyl_p_eig} holds for all $R \ge R_0$.
    Next we check the vanishing condition.
    Using \eqref{eqn:p-def}, \eqref{eqn:phi_tau_eig-}, and \eqref{eqn:ghatrLower}, for $R \ge R_0$, we get
    \begin{align*}
        \dfrac{ \gamma(R) \log(\hat R) g(R) h(R) e^{ \tilde c_1 \tilde{\phi}(R)}   }{g(\hat R) h(\hat R)  }
        &\le 2^{\tilde c_1 L_2} \brac{\dfrac{ \pr{\log R}^{p + 1 - q + \frac{1}{\mu}}}{\log R + \mu \log(\log R)}}^{\mu}.
    \end{align*}
    Since the definition of $\mu$ ensures that $p + 1 - q + \frac{1}{\mu} < 1$, we see that
    \begin{align*}
		\lim_{R \ra \infty} \dfrac{ \gamma(R) \log(\hat R) g(R) h(R) e^{ \tilde c_1 \tilde{\phi}(R)}   }{g(\hat R) h(\hat R)  }
        &= 0,
	\end{align*}
    which establishes \eqref{eqn:vanishing_eig}.
    In fact, the bound in \eqref{eqn:cond2_-} implies \eqref{eqn:R0BigCond_eig}. 
    Proposition \ref{prop:decay_cyl_p_eig} shows that there exists a constant $c_3(d, m, \La, L_1, L_2, R_1, \mu, c_8, q, R_0, c_0) > 0$ so whenever $\abs{x_0} \ge R_0$, it holds that
	\begin{align*}
    	  \norm{u}_{L^2(\mathbb{K}_1(x_0))} 
          &\ge  \exp \set{ -c_3 c_8 2^p (2\Lambda)^{q} \abs{x_0}^q \pr{\log \abs{x_0}}^{\mu p + 1} \log(\log \abs{x_0}) } \\
          &\ge  \exp \set{ -c_3 c_8 2^p (2\Lambda)^{q}  \abs{x_0}^q \pr{\log \abs{x_0}}^{\mu p + 1  + \frac{\eps}{2}}},
	\end{align*}
    where we have used the second bound in \eqref{eqn:cond1_-} to reach the second line.
    Since $\mu p + 1 + \frac {\eps} {2} = \frac{q}{q-p} + \eps$, then the conclusion in \eqref{eqn:L2Conc} with \eqref{eqn:bigH_defn} for $p < q$ follows.

\end{proof}

\subsection{Unique continuation for additional decay}

We finish this section with a final application of Proposition \ref{prop:decay_cyl_p_eig} to eigenfunctions of operators with no lower order terms.
But first, we explain why we consider this setting separately.

If $u$ is an eigenfunction, 
\begin{align*}
    -\divv(A \nabla u)  + W \cdot \nabla u + V u = \lambda u \,\, \text{ in } \, \T^d \times \R^m,
\end{align*}
and either $V \not\equiv 0$ or $W \not\equiv 0$, then depending on how $\abs{\nabla A}$ decays at infinity, we can either apply Theorem \ref{thm:tau_betw} or Theorem \ref{thm:tau=1}. 
If $\abs{\nabla A}$ has a relatively slow rate of decay that satisfies the hypotheses of Theorem \ref{thm:tau_betw} or Theorem \ref{thm:tau=1} for $L_2$ large, then in view of the sharp examples in Theorem \ref{thm:eg_1}\footnote{Based on the constructions in \cite{KLP25}, we expect that slight modifications to the arguments in Section \ref{sec:examples} would yield eigenfunctions instead of solutions with compactly supported potentials. }, $u$ could decay as fast as $\exp\{-C \abs{x}^{1 + \tilde{c}_1 L_2 }\} \ll \exp\{-C \abs{x} \}$.
On the other hand, if $\abs{\nabla A} \le L_2(1 + \abs{x})^{-1}$ for $L_2$ sufficiently small, then in view of the examples of Meshkov \cite{M92} and their generalizations in \cite{Dav14}, Theorem \ref{thm:tau=1} shows that the eigenfunction $u$ might decay at the critical rate of $\exp\{ - C \abs{x}^q\}$, where $q > 1$ is defined in \eqref{eqn:q-def}. 

If $V, W \equiv 0$ and $L_2 \ll 1$, Theorem \ref{thm:tau=1} implies that it is possible for $u$ to decay like $\exp\{ - C \abs{x}^p\}$, where $p := 1 + \tilde c_1 L_2$.
When $L_2 = 0$, $p= 1$ and this decay rate is known to be sharp.
However, for $L_2 \ll 1$, one might expect $u$ to decay like $\exp\{ - C \abs{x}\}$, but the results of Theorem \ref{thm:tau=1} do not give such a conclusion.
In the following theorem, under a slightly stronger assumption on the rate of decay of $\abs{\gr A}$, we show that eigenfunctions can decay like $\exp\{ - C \abs{x}\}$.
See Remark \ref{rmk:log_prob} that follows the proof for a discussion of this faster decay rate.

\begin{thm}[Rate of decay at infinity when coefficients have extra decay]
\label{thm:tau=1+}
    For $A \in \mathcal{M}(d,m,\La)$ and $R_1 > 1$, assume that $A$ satisfies the following conditions:
	\begin{equation*}
		\begin{split}
			\abs{\nabla A(\te, x)} & \le L_1 \,\,\,\,\quad\quad\quad \text{ for a.e. } (\te, x) \in \T^d \times \R^m, \\
			\abs{ \nabla A(\te, x)  } & \le \frac{L_2}{\abs{x} \log \abs{x}} \,\,\,\,  \text{ for a.e. } (\te, x) \in \T^d \times \R^m, \, \abs{x} > R_1.
		\end{split}
	\end{equation*}
    For $\la \in \C$, let $u \in H^1_{\loc}(\T^d\times \R^m)$ be a nontrivial solution of 
    $$-\di\pr{A \gr u} = \la u \,\, \text{ in } \, \T^d \times \R^m$$
    that is normalized so that \eqref{eqn:u_normalized_eig} holds and bounded in the sense that 
    \begin{align}
    \label{eqn:eig_bdd}
        \abs{u(\te, x)} \le \exp\{ c_0 (\sqrt{\lambda} + 1) (\abs{x} + 1)  \} \; \;  \text{ for all } (\te, x) \in \T^d \times \R^m.
    \end{align}
    For any $\beta \in (1/2, 1)$, there exists a distance $R_0(d,m,\Lambda, L_2, R_1, \beta ) > 1$ and a constant $c_3(d,m,\Lambda, L_1, L_2, R_1, \beta, R_0, c_0) > 1$ so that whenever $\abs{x_0} \ge R_0$, it holds that
    \begin{align*}
        \norm{u}_{L^2(K_1(x_0))} \ge \exp \{ - c_3 ( \sqrt{\la} + 1)\abs{x_0} \exp \pr{  \log(\abs{x_0})^\beta  }  \log(\abs{x_0})  \}.
    \end{align*}
\end{thm}

\begin{proof}

 Once again, we want to apply Proposition \ref{prop:decay_cyl_p_eig}, so we need to define $\gamma$, $g$, and $h$ and check the conditions \eqref{eqn:vanishing_eig} -- \eqref{eqn:h_1_cyl_p_eig} and \eqref{eqn:u_bdd_cyl_p_eig}. 
 For $\beta \in (1/2, 1)$ and $R > 1$, define $\gamma(R) = \exp\brac{\pr{\log R}^\beta}$.
 Choose $R_0(d, m, \Lambda , L_2,  R_1, \beta ) > \max \{6 \Lambda, \Lambda^2, 2 R_1 \gamma(2R_1), 8 \gamma(8)     \}$ so that the following conditions hold for all $R \ge R_0$
 \begin{align}
     & \log( \abs{\mathbb{K}_R}  ) \le R \label{eqn:qq1} \\
     & (\log R)^\beta + 2L_2 \le R \label{eqn:qq2} \\
     & 1 + \frac{\beta}{2} (\log R)^{\beta-1} \le \brac{1 + (\log R)^{\beta-1}}^\beta  \label{eqn:qq3} \\
     & 2 \tilde{c}_1 L_2 + 2c_1 \log(\Lambda) + \log(64) + \log(\log R)  \le \frac{\beta}{2} (\log R)^{2\beta -1} \label{eqn:qq5}
 \end{align}

First of all, notice that for $R \ge 8$,
\begin{align*}
    \log(\hat{R}) 
    = \log R + (\log R)^\beta 
    \le 2 \log R 
    \le 3 \log(R/2),
\end{align*}
so that $\log(R)/\log(\check{R}/2) \le 3$ for $R \ge R_0$. 
Since $f(s) = L_2/(s\log s)$ for $s > R_1$, then recalling \eqref{eqn:phi_tilde_cyl_p}, we see that 
\begin{align}
\label{eqn:qq4}
    \tilde{\phi}(R)  
    \le \int_{\frac{1}{2} \check{R}}^R \frac{L_2}{s\log s} \; ds 
    = L_2 \log\brac{\log R/\log(\check{R}/2  )}  
    \le \log 3 L_2 
    < 2L_2 .
\end{align}
In particular, $\tilde \phi$ is bounded for $R \ge R_0$. 

Define $g(R) = (\sqrt{\la} + 1) R$.
With the choice of $g(R)$, we have that \eqref{eqn:g_1_crudeBd_eig} is a consequence of \eqref{eqn:qq1}.
Since $F(\hat{R}) = 1$, \eqref{eqn:qq4} and \eqref{eqn:qq2} show that
 \begin{align*}
     F(\hat{R}) \log(\gamma(R)) + \sqrt{\lambda}\hat{R} + \tilde{\phi}(\hat{R}) 
     \le  (\log\hat{R})^\beta
 + \sqrt{\lambda }\hat{R} + 2L_2 \le (\sqrt{\la} + 1) \hat{R} = g(\hat{R}),
\end{align*}
which gives \eqref{eqn:g_1_cyl_p_eig}. 
As for the estimate \eqref{eqn:h_1_cyl_p_eig}, we verify that whenever $R \ge R_0$,
\begin{align*}
    \gamma(R) \log(\hat{R}) g(2\Lambda \hat{R}) \le 4 \Lambda \gamma(R) \log(R) g(\hat{R}),
\end{align*}
so taking $h(R) \coloneqq 4 \Lambda \gamma(R) \log(R)$ for $R \ge R_0$ and using that $h$ is increasing, we see \eqref{eqn:h_1_cyl_p_eig} holds with this choice of $h$.

Finally, we address the vanishing condition \eqref{eqn:vanishing_eig}. 
Using the definitions of $g$, $h$, and \eqref{eqn:qq4}, for $R \ge R_0$, 
\begin{align*}
    \frac{\gamma(R) \log(\hat{R}) g(R) h(R) e^{\tilde{c}_1 \tilde{\phi}(R)}}{g(\hat{R}) h(\hat{R})} 
    & \le  e^{2 \tilde{c}_1 L_2} \frac{ \gamma(R) \log R }{ \gamma(\hat{R})  } . 
\end{align*}
From \eqref{eqn:qq3}, we obtain the lower bound
\begin{align*}
    \gamma(\hat{R}) 
    = \exp \{  \brac{ \log R + (\log R)^\beta }^\beta   \}  
    \ge \exp \left \{   (\log R)^\beta + \frac{\beta}{2} (\log R)^{2\beta -1} \right \},
\end{align*}
so that \eqref{eqn:qq5} then gives
\begin{align*}
    e^{2 \tilde{c}_1 L_2}\frac{\gamma(R)\log R}{\gamma(\hat{R})} 
    \le \exp \left \{ 2 \tilde{c}_1 L_2 + \log(\log R) - \frac{\beta}{2} (\log R)^{2\beta -1}   \right \} 
    \le \frac{1}{64 \Lambda e^{2c_1}}.
\end{align*}
In particular, we see that \eqref{eqn:R0BigCond_eig} holds, and in addition, since $2 \beta  > 1$, that \eqref{eqn:vanishing_eig} holds as well. 
Thus with the bound \eqref{eqn:eig_bdd} and the estimate \eqref{eqn:qq4}, Proposition \ref{prop:decay_cyl_p_eig} then gives that as long as $\abs{x_0} \ge R_0$, then 
\begin{align*}
    \norm{u}_{L^2(K_1(x_0))} 
    \ge \exp \left \{  -4 \La c_3  e^{2 \tilde{c}_1 L_2} (\sqrt{\lambda} + 1)\abs{x_0} e^{(\log\abs{x_0})^\beta } \log 
   \abs{x_0} \right\},
\end{align*}
with $c_3(d,m,\Lambda, L_1, L_2, R_1, \beta, R_0, c_0)  >0$.
(An inspection of the proof of Proposition \ref{prop:decay_cyl_p_eig} shows that we can take $c_3$ independent of $\la$ even though $g$ depends on $\la$.)
This is our desired lower bound.
\end{proof}

\begin{rmk}
\label{rmk:log_prob}
    Given that the decay of $\abs{\nabla A}$ assumed in Theorem \ref{thm:tau=1+} is faster than that assumed in Theorem \ref{thm:tau=1}, the curious reader may wonder if the conclusion of Theorem \ref{thm:tau=1+} holds, i.e., that eigenfunctions decay at worst like  $\exp\{  -C\abs{x}f(\abs{x})\}$ for some subpolynomial function $f$, if we instead assume that for $L_2$ sufficiently small,
    \begin{align}\label{eqn:log_prob_alt}
        \abs{\nabla A(\te, x)} \le \frac{L_2}{1 + \abs{x}}
        \,\,\,\, \text{ for a.e. } (\te, x) \in \T^d \times \R^m.
    \end{align}
    It is unclear whether such a result is true, and at least with the techniques of this paper (namely, the technology of Proposition \ref{prop:decay_cyl_p_eig}), it does not seem possible to prove. 
    The main obstruction is that the critical exponent $q$ from \eqref{eqn:q-def} is not \textit{strictly} larger than $1$ when $W \equiv V \equiv 0$. 
    We elaborate on this below. 

    Assume the hypotheses of Theorem \ref{thm:tau=1+} with the decay condition \eqref{eqn:log_prob_alt}.
    Suppose that $\gamma$, $g$, and $h$ are functions that we can use in Proposition \ref{prop:decay_cyl_p_eig} to conclude that the eigenfunction $u$ decays at most like $\exp\{- C \abs{x}^{1 + \delta}\}$ for some $\delta >0$. 
    Then we must have $g(R)h(R) e^{\tilde c_1 \tilde \phi(R)} \approx R^{1+\de}$.
    Since $\tilde{\phi}(R) = L_2 \log(2\gamma(\check{R}))$ in this setting, then $e^{\tilde{c}_1 \tilde{\phi}(R)} = (2\gamma(\check{R}))^{\tilde{c_1} L_2}$ and we conclude that 
    $$g(R)h(R) \gamma(\check{R})^{\tilde{c_1} L_2} \approx R^{1+\de}.$$
    For Proposition \ref{prop:decay_cyl_p_eig} to be applicable, we need $\ga$, $g$, and $h$ to satisfy
    \begin{align*}
        \limsup_{R \ra \infty } \frac{\gamma(R) \log(\hat{R}) g(R) h(R) \gamma(\check{R})^{\tilde{c}_1 L_2}}{g(\hat{R}) h(\hat{R})} = 0.
    \end{align*}
    However,
    \begin{align*}
        \frac{\gamma(R) \log(\hat{R}) g(R) h(R) \gamma(\check{R})^{\tilde{c}_1 L_2}}{g(\hat{R}) h(\hat{R})} 
        &= \gamma(R)^p \log(\hat{R})  \frac{g(R) h(R) \gamma(\check{R})^{\tilde{c}_1 L_2}}{g(\hat{R}) h(\hat{R}) \gamma(R)^{\tilde{c}_1 L_2}} \\
        &\approx \gamma(R)^p \log(\hat{R})  \frac{R^{1+\de}}{\hat R^{1+\de}}
        = \frac{\log (\hat R)}{\ga(R)^{1 + \de - p}}.
    \end{align*}
    Since $\ga$ must be non-decreasing, we have to choose $\de > p - 1$.
    In particular, we cannot make $\de$ arbitrarily close to $0$.
    
\end{rmk}

\section{Examples of solutions with rapid decay} 
\label{sec:examples} 

Here we construct examples of operators and global solutions $u$ to \eqref{eqn:soln_cyl} that demonstrate the sharpness of the lower bounds described by Theorems \ref{thm:tau_betw} and \ref{thm:tau=1}. 

Throughout this section, we focus on the $3$-dimensional domain $\mathbb{T}^2 \times \R$, i.e., we fix $d = 2$, $m = 1$.
We use $(\te, \vp)$ to denote the coordinates in $\mathbb{T}^2$ and $x$ to denote the points in $\R$.
In all of our examples, $A_2 = 1$, so our coefficient matrices belong to $\mathcal{M}_0(2, 1, \La)$, see Definition \ref{defn:matClass}.

The following theorem describes these examples.

\begin{thm}[Sharp examples]
\label{thm:eg_1}
	For every $\epsilon \in (0,1]$, there exists a coefficient matrix $A \in \mathcal{M}_0(2, 1, 75)$, a bounded, compactly supported potential function $V$, and a real-valued, nontrivial solution $u$ to the equation
	\begin{align}
    \label{eqn:cyl_soln}
		\divv(A \nabla u ) = Vu \,\, \text{ in } \, \T^2 \times \R,
	\end{align}
	with the property that for every $(\te, \vp, x) \in \mathbb{T}^2 \times \R$,
	\begin{align}
    \label{eqn:egest1}
		 \abs{\nabla A(\te, \vp, x)} \le \dfrac{C_\epsilon}{(1+\abs{x})^{ 1 -  \epsilon}} , \qquad  
         \abs{u(\te, \vp, x)} \le c \exp\set{- \exp\pr{  C  \abs{x}^{\epsilon } }}. 
	\end{align}
    Similarly, for every $\epsilon \in (0, \tfrac 1 2]$, there exists $A \in \mathcal{M}_0(2, 1, 75)$, a bounded, compactly supported $V$, and a real-valued, nontrivial solution $u$ to \eqref{eqn:cyl_soln} with the property that for every $(\te, \vp, x) \in \mathbb{T}^2 \times \R$,
	\begin{align}
    \label{eqn:egest2}
		\abs{\nabla A (\te, \vp, x)} \le \dfrac{C}{\epsilon (1+\abs{x})}, \qquad 
        \abs{u(\te, \vp, x)} \le c \exp\pr{ -C_\epsilon  \abs{x}^{\frac 1 \epsilon} } .
	\end{align}
\end{thm}

To prove this theorem, we rely on the following proposition.

\begin{prop}[Sharp examples on half-space]
\label{prop:eg_half}
	For every $\epsilon \in (0,1]$, there exists a coefficient matrix $A \in \mathcal{M}_0(2, 1, 75)$ and a real-valued, nontrivial solution $u$ to 
	\begin{align}
    \label{eqn:halfcyl_soln}
		\divv(A \nabla u ) = 0 \,\, \text{ in } \, \T^2 \times (1, \iny),
	\end{align}
	with the property that for every $(\te, \vp, x) \in \mathbb{T}^2 \times [1, \iny)$,
	\begin{align}
    \label{eqn:egest3}
		 \abs{\nabla A(\te, \vp, x)} \le \dfrac{C_\epsilon}{x^{ 1 -  \epsilon}} , \qquad  
         \abs{u(\te, \vp, x)} \le c \exp\set{- \exp\pr{ C x^{\epsilon } }}. 
	\end{align}
    Similarly, for every $\eps \in (0, \tfrac 1 2]$, there exists $A \in \mathcal{M}_0(2, 1, 75)$ and a real-valued, nontrivial solution $u$ to \eqref{eqn:halfcyl_soln} with the property that for every $(\te, \vp, x) \in \mathbb{T}^2 \times [1, \iny)$,
	\begin{align}
    \label{eqn:egest4}
		\abs{\nabla A (\te, \vp, x)} \le \dfrac{C}{\eps x}, \qquad 
        \abs{u(\te, \vp, x)} \le c \exp\pr{ - C_\eps x^{\frac 1 \eps} } . 
	\end{align}
    In both cases, there exist constants $k_0 \in \N$ and $\de \in (0, 1)$ so that for every $(\te, \vp, x) \in \mathbb{T}^2 \times \brac{1, 1+\de}$,
    \begin{align}
    \label{eqn:localBehav}
        A(\te, \vp, x) \equiv I, \qquad 
        u(\te, \vp, x) \equiv \cos(k_0 \te) e^{-k_0 x} .
    \end{align}    
\end{prop}

The two results described by Theorem \ref{thm:eg_1} and Proposition \ref{prop:eg_half} are very similar, so we point out the main differences.
First, Theorem \ref{thm:eg_1} holds on the full space $\mathbb{T}^2 \times \R$ while Proposition \ref{prop:eg_half} is for the half-space $\mathbb{T}^2 \times (1, \iny)$.
Second, Theorem \ref{thm:eg_1} holds with a compactly supported potential function $V$ while we take $V \equiv 0$ in Proposition \ref{prop:eg_half}. 
By Harnack's inequality, we know that we cannot construct a bounded, non-constant solution in the full space $\T^2 \times \R$ without a nonzero lower order term $V$.

The main tool that we use to prove Proposition \ref{prop:eg_half} is the following construction block (which is largely inspired by the work of Miller, and the recent work of \cite{KLP25}).
We leave the proof of this lemma until the end of the section.

\begin{lemma}[Building blocks]
\label{lemma:transformation_block}
	For any $T > 0$, there exists $k_0(T) \in \N$ so that for any $k \in \N_{\ge k_0}$, there exists $A = A(\te, \vp, x; k) \in \mathcal{M}_0(2, 1, 75)$ that transforms the function $\cos(k\te) e^{-kx}$ into the function $\cos(2k\vp) e^{-2kx}$ on the interval $[0, T]$. 
    That is, there exists a real-valued solution $u = u(\te, \vp, x; k)$, defined on $\T^2 \times [0, T]$, that satisfies 
	\begin{align*}
	    \divv(A \nabla u) = 0 \text{ in } \T^2 \times (0, T),
	\end{align*}
    where
	\begin{align}
    \label{eqn:uPiecewise}
		u(\te, \vp, x; k) = \begin{cases}
			\cos(k\te) e^{-kx} & x \in [0, \frac 1 4 T] \\
			\cos(2k\vp) e^{-2kx} & x \in [\frac{3}4 T,  T]
		\end{cases}
	\end{align}
	and there exists universal $C_1 > 0$ so that for each $(\te, \vp, x) \in \T^2 \times [0,  T]$, 
    \begin{equation}
    \label{eqn:AuBlockBounds}
        \begin{aligned}
        &\abs{\nabla A (\te, \vp, x)} \le C_1 T^{-1} \\
        &\abs{u(\te, \vp, x)} \le 2 e^{- kx}.
        \end{aligned}
    \end{equation}
	Moreover, $A \equiv I$ in $\disp \mathbb{T}^2 \times \pr{\brac{0, \tfrac T {20}} \cup  \brac{\tfrac{3 T}{4}, T}}$.
\end{lemma}

Assuming Proposition \ref{prop:eg_half}, we now prove Theorem \ref{thm:eg_1}.

\begin{proof}[Proof of Theorem \ref{thm:eg_1}]
    First, we define a bounded potential $V_0 = V_0(x)$ on $[-1, 1]$ for which the associated Schr\"odinger operator transforms the function $\cos(k_0 \te) e^{k_0 x}$ near $x = -1$ into the function $\cos(k_0 \te) e^{-k_0 x}$ near $x = 1$ over the interval $(-1, 1)$. 
    Let $f:[-1,1] \rightarrow [e^{-k_0}, 1]$ be a smooth, even function with
    \begin{align*}
        \abs{f''(x)/f(x)} \le C_{k_0},  \quad
        f(x) = e^{k_0 x} \, \text{ for } x \in [-1, -\tfrac 12], \quad 
        f(x) = e^{-k_0 x} \, \text{ for } x \in [\tfrac 1 2, 1].
    \end{align*}
    A computation shows that $u_0(\te, \vp, x) := \cos(k_0 \te) f(x)$ satisfies
    \begin{align}
    \label{eqn:localEg}
        \Delta u_0 
        = - k_0^2 \cos(k_0 \te) f(x)  + \cos(k_0 \te) f''(x) 
        = V_0 u_0,
    \end{align}
    where $\disp V_0(x) \coloneqq \frac{f''(x)}{f(x)} - k_0^2$ is bounded by some constant that depends only on $k_0$. 
    
    Next we extend the solution $u(\te, \vp, x)$ in \eqref{eqn:halfcyl_soln} to all of $\mathbb{T}^2 \times \R$ by setting 
    \begin{align*}
        u(\te, \vp, x) & \coloneqq \begin{cases}
            u(\te, \vp,- x) & x < -1 \\
            u_0(\te, \vp, x) & |x| \le 1 \\
            u(\te, \vp, x) & x > 1
        \end{cases}.
    \end{align*}
    By the condition on $u$ in \eqref{eqn:localBehav} and the definition of $u_0$, we see that $u$ is smooth.
    We similarly extend the coefficient matrix from $\mathbb{T}^2 \times (1, \iny)$ to $\mathbb{T}^2 \times \R$ by setting
    \begin{align*}
         A(\te, \vp, x) &\coloneqq \begin{cases}
            A(\te, \vp,- x) & x < -1 \\
            I & |x| \le 1 \\
            A(\te, \vp, x) & x > 1
        \end{cases}.
    \end{align*}
    By the condition on $A$ in \eqref{eqn:localBehav}, $A$ is also smooth.
    Let $V : \mathbb{T}^2 \times \R \to \R$ be defined by
    $$V(\te, \vp, x) := V_0(x) \chi_{[-1, 1]}(x).$$
    Since $u$ and $A$ are smooth, \eqref{eqn:halfcyl_soln} and \eqref{eqn:localEg} show that $u$ satisfies \eqref{eqn:cyl_soln}. 
    Moreover, we readily see that the estimates in \eqref{eqn:egest3} or \eqref{eqn:egest4} for $x \ge 1$ imply those in \eqref{eqn:egest1} or \eqref{eqn:egest2}, respectively, for all $x \in \R$. 
\end{proof}

Our next step is the proof of Proposition \ref{prop:eg_half}.
In this construction, we heavily rely on the ``building blocks" in Lemma \ref{lemma:transformation_block}.

\begin{proof}[Proof of Proposition \ref{prop:eg_half}]
    For our construction, we use the following sequences. 
    First, we have a non-decreasing sequence corresponding to the lengths of our intervals in the $x$-variable, $\set{T_n}_{n=1}^\iny$; these will be specified below.
    Then we have the sequence of endpoints $\set{x_n}_{n = 0}^\iny$, defined recursively as $x_0 = 1$,
	\begin{align*}
		x_{n} \coloneqq x_{n-1} + T_n = 1 + \sum_{j = 1}^n T_j.
	\end{align*}
	For each such $n \in \N$, define the intervals $I_n \coloneqq [x_{n-1}, x_{n}] = [x_{n-1}, x_{n-1} + T_n]$ and note that $\disp \bigcup_{n = 1}^\iny I_n = [1, \iny)$.

    Choose $k_0(T_1) \in \N$ so that Lemma \ref{lemma:transformation_block} is applicable with $k \in \N_{\ge k_0}$ whenever $T \ge T_1$.
    Set $k_n = 2^{n} k_0$.

    For each $n \in \N$, we define $u_{n}$ and $A_{n} \in \mathcal{M}_0(2, 1, 75)$ on $\T^2 \times I_{n}$ as follows: 
    If $n$ is odd, we let 
    $$u_{n}(\te, \vp, x) = \exp\pr{\sum_{j = 0}^{n-2} k_j x_j - k_{n-1} x_{n-1}} u(\te, \vp, x - x_{n-1}; k_{n-1})$$ 
    and $A_{n}(\te, \vp, x) = A(\te, \vp, x - x_{n-1}; k_{n-1})$, where $u(\cdot; k)$ and $A(\cdot; k)$ from Lemma \ref{lemma:transformation_block}. 
    In this case, we have that $\di \pr{A_n \gr u_n} = 0$ in $\T^2 \times (x_{n-1}, x_n)$ and
    \begin{align*}
		u_{n}(\te, \vp, x) 
   %      &= \begin{cases}
			% e^{\sum_{j = 0}^{n-2} k_j x_j - k_{n-1} x_{n-1}} \cos\pr{k_{n-1} \te} e^{-k_{n-1}(x - x_{n-1})} 
   %          & x \in [x_{n-1}, x_{n-1} + \frac 1 4 T_{n}] \\
			% e^{\sum_{j = 0}^{n-2} k_j x_j - k_{n-1} x_{n-1}} \cos\pr{2 k_{n-1} \vp} e^{-2 k_{n-1}(x - x_{n-1})} 
   %          & x \in [x_{n-1} + \frac{3}4 T_{n}, x_{n-1} + T_n]
%		\end{cases} \\
        &= \begin{cases}
			\cos(k_{n-1} \te) \exp\pr{\sum_{j = 0}^{n-2} k_j x_j} e^{-k_{n-1} x} & x \in [x_{n-1}, x_{n-1} + \frac 1 4 T_{n}] \\
			\cos(k_n\vp)\exp\pr{\sum_{j = 0}^{n-1} k_j x_j}  e^{-k_n x} & x \in [x_{n} - \frac{1}4 T_{n}, x_{n}]
		\end{cases}.
	\end{align*}
    If $n = 1$, we interpret $\disp \sum_{j = 0}^{n-2} k_j x_j$ as an empty sum and take it equal to $0$.
    If $n$ is even, we switch the roles of $\te$ and $\vp$ and set 
    $$u_{n}(\te, \vp, x) = \exp\pr{\sum_{j = 0}^{n-2} k_j x_j - k_{n-1} x_{n-1}} u(\vp, \te, x - x_{n-1}; k_{n-1})$$ 
    and $A_{n}(\te, \vp, x) = A(\vp, \te, x - x_{n-1}; k_{n-1})$ so that $\di \pr{A_n \gr u_n} = 0$ in $\T^2 \times (x_{n-1}, x_n)$ and
    \begin{align*}
		u_{n}(\te, \vp, x) 
        = \begin{cases}
			\cos(k_{n-1} \vp) \exp\pr{\sum_{j = 0}^{n-2} k_j x_j} e^{-k_{n-1} x} & x \in [x_{n-1}, x_{n-1} + \frac 1 4 T_{n}] \\
			\cos(k_n\te)\exp\pr{\sum_{j = 0}^{n-1} k_j x_j}  e^{-k_n x} & x \in [x_{n} - \frac{1}4 T_{n}, x_{n}]
		\end{cases}.
	\end{align*}
    Notice that for each $n \in \N$, the definition of $u_n$ on $[x_n - \frac 1 4 T_n, x_n]$ agrees with the definition of $u_{n+1}$ on $[x_n, x_n +\frac 1 4 T_{n+1}]$.
	
    For each $n \in \N$, \eqref{eqn:AuBlockBounds} in Lemma \ref{lemma:transformation_block} shows that there exists universal $C_1 > 0$ so that for each $(\te, \vp, x) \in \T^2 \times I_{n}$, 
    \begin{equation}
    \label{eqn:AnunBlockBounds}
        \begin{aligned}
        &\abs{\nabla A_n(\te, \vp, x)} \le C_1 T_{n}^{-1} \\
        &\abs{u_n(\te, \vp, x)} \le 2 \exp\pr{\sum_{j = 0}^{n-2} k_j x_j} e^{- k_{n-1} x} = 2 \exp\pr{k_0 \sum_{j = 0}^{n-2} 2^j x_j} e^{- k_{n-1} x}.
        \end{aligned}
    \end{equation}

    We define the coefficient matrix $A(\te, \vp, x)$ and the solution $u(\te, \vp, x)$ on $\mathbb{T}^2 \times [1, \iny)$ by setting  
	 \begin{equation}
    \label{eqn:AuHalfDefn}
    \begin{aligned}
		A(\te, \vp, x) & \coloneqq A_{n}(\te, \vp, x), \, \text{ if } \, x \in I_n, \\
        u(\te, \vp, x) & \coloneqq u_{n}(\te, \vp, x), \, \text{ if } \, x \in I_n.
	\end{aligned}    
    \end{equation}
    Since $u_1(\te, \vp, x) = \cos(k_0 \te)e^{- k_0 x}$ when $x \in \brac{1, 1 + \frac{T_1}{4}}$, while $A_1(\te, \vp, x) = I$ when $x \in \brac{1, 1 + \frac{T_1}{20}}$, then \eqref{eqn:localBehav} holds with $\de = \frac{T_1}{20} > 0$.
    By construction, $u$ is well-defined and smooth.
    Because $A \equiv I$ near each $x_n$, then $A$ is also well-defined and we see that \eqref{eqn:halfcyl_soln} holds.
	
	It remains to choose the interval lengths $\set{T_n}_{n =1}^\iny$ appropriately to obtain our desired bounds described by \eqref{eqn:egest3} and \eqref{eqn:egest4}. 
	
	For the first set of examples, we define $\set{x_n}_{n=1}^\iny$ using powers. 
    Given $\eps \in (0, 1]$, set $p(\eps) = \frac 1 \eps \ge 1$ and define the endpoints as $x_n \coloneqq \pr{n+1}^p$. 
    If $x \in I_n$ for some $n \ge n_0 := \lceil p - 2 \rceil$, then we see that
    \begin{align*}
      T_n 
      &= x_n - x_{n-1} 
      = \pr{n+1}^p - n^p = \pr{n + 1}^p\brac{1 - \pr{1 - \frac {1} {n+1}}^p} \\
      &= \pr{n + 1}^{p-1} \brac{ \sum_{k=1}^\iny (-1)^{k-1} \frac{p(p-1) \ldots (p - k +1)}{k!}\pr{\frac 1 {n+1}}^{k-1}} \\
      &\ge \frac p 2 \pr{n + 1}^{p-1}
      = \frac p 2 x_n^{1 - \eps}
      \ge \frac 1 {2\eps} x^{1 - \eps}.
    \end{align*}
    % \begin{align*}
    %     &\sum_{k=1}^\iny (-1)^{k-1} \frac{p(p-1) \ldots (p - k +1)}{k!}\pr{\frac 1 {n+1}}^{k-1} \\
    %     &= p\brac{1 - \frac{p-1}{2}\pr{\frac 1 {n+1}}
    %     + \sum_{k=3}^\iny (-1)^{k-1} \frac{(p-1) \ldots (p - k +1)}{k!}\pr{\frac 1 {n+1}}^{k-1}} \\
    %     &\approx p\brac{1 - \frac{p-1}{2}\pr{\frac 1 {n+1}}
    %     + \sum_{k=3}^\iny (-1)^{k-1} \frac{1}{k!}\pr{\frac {p-1} {n+1}}^{k-1}}
    % \end{align*}
    Combining \eqref{eqn:AuHalfDefn} with the first part of \eqref{eqn:AnunBlockBounds} shows that for all $(\te, \vp, x) \in \mathbb{T}^2 \times [x_{n_0 - 1}, \iny)$,
    \begin{align*}
        \abs{\gr A(\te, \vp, x)} 
        &\le \frac{C_1}{T_n}
        \le \frac{2 \eps C_1}{x^{1 - \eps}}.
    \end{align*}
    By adjusting the constants, it can be shown that the above bound also holds for all $(\te, \vp, x) \in \mathbb{T}^2 \times [1, x_{n_0 - 1})$.
    In particular, the first part of \eqref{eqn:egest3} holds for any choice of $\eps \in (0, 1]$.
    
    To address the second part of \eqref{eqn:egest3}, first observe that
    \begin{align*}
        \sum_{j=0}^{n-2} 2^j x_j
        &= \sum_{j=0}^{n-2} 2^j \pr{j+1}^p
        = n^p \sum_{j=0}^{n-2} 2^j 
        - \frac 1 2 \sum_{j=1}^{n-1} 2^j \pr{n^p - j^p} \\
        &\le 2^{n-1} x_{n-1} - n^p
        - n^{p-1}\frac 1 2 \sum_{j=1}^{n-1} 2^j \pr{n - j} \\
        &= 2^{n-1} x_{n-1} - n^p - n^{p-1} \pr{2^n - n - 1}
        \le 2^{n-1} x_{n-1} - 2^{n-1},
    \end{align*}
    where we have used that $n^p - j^p \ge n^{p-1} \pr{n - j}$ in the second line.
    It then follows from the second part of \eqref{eqn:AnunBlockBounds} that if $x \in I_n$,
    \begin{align*}
        \abs{u_n(\te, \vp, x)} 
        &\le 2 \exp\pr{k_0 \sum_{j = 0}^{n-2} 2^j x_j} e^{- k_{n-1} x} \\
        &\le 2 \exp\set{k_0 \pr{2^{n-1} x - 2^{n-1}} - k_0 2^{n-1} x} \\
        &= 2 \exp\set{- \exp\brac{ \log 2 \, x_n^{\eps} + \log \pr{\frac{k_0}{4}}}}
        \le 2 \exp\set{- \exp\pr{\log 2 \, x^{\eps}  }},
    \end{align*}
    where we have assumed that $k_0 \ge 4$.
    In particular, the conditions described by the second part of \eqref{eqn:egest3} hold for any $\eps \in (0, 1]$.
    
    \vspace{0.5em}
    
	For the second set of examples, we define $\set{x_n}_{n=1}^\iny$ geometrically. 
    Given $ \eps \in (0, \tfrac 1 2]$, choose $\ga(\eps) \in (0, 1]$ so that $\frac{\log \pr{1 + \ga}}{\log 2} = \frac {\eps}{1 - \eps}$.
    Define the endpoints as $x_n := \pr{1 + \ga}^n$ so that $T_n = x_n - x_{n-1} = \ga \pr{1 + \ga}^{n-1}$ for each $n \in \N$.
    In particular, if $x \in I_n$, we have $\disp T_n = \frac{\ga}{1 + \ga} x_n \ge \frac{\ga}{1 + \ga} x$ and $\disp n \ge \frac{\log x}{\log\pr{1 + \ga}}$.
    Combining \eqref{eqn:AuHalfDefn} with the first part of \eqref{eqn:AnunBlockBounds} shows that for all $(\te, \vp, x) \in \mathbb{T}^2 \times [1, \iny)$,
    \begin{align*}
        \abs{\gr A(\te, \vp, x)} 
        &\le \frac{C_1\pr{1 + \ga}}{\ga x}
        \le \frac{2C_1}{\ga x}.
    \end{align*}
    Since $\ga \ge \log\pr{1 + \ga} \ge \eps \log 2$, then we see that for all $x \ge 1$,
    \begin{align*}
		\abs{\nabla A(\te, \vp, x)} \le \frac{C}{\eps x}, 
	\end{align*}
    showing that the first part of \eqref{eqn:egest4} holds for any $\eps \in (0, \tfrac 1 2]$.
    
    Since
    \begin{align*}
        \sum_{j=0}^{n-2} 2^j x_j
        &= \sum_{j=0}^{n-2} 2^j \pr{1 + \ga}^j
        = \frac{\brac{2\pr{1 + \ga}}^{n-1} - 1}{2\pr{1 + \ga} -1}
        < \frac{\brac{2\pr{1 + \ga}}^{n-1}}{1 + 2 \ga}  
        = \frac{2^{n-1} x_{n-1}}{1 + 2 \ga} 
    \end{align*}
    then for $x \in I_n$, the second part of \eqref{eqn:AnunBlockBounds} shows that
    \begin{align*}
        \abs{u(\te, \vp, x)} 
        &\le 2 \exp\pr{k_0 \sum_{j = 0}^{n-2} 2^j x_j} e^{- k_{n-1} x} 
        \le 2 \exp\pr{\frac{2^{n-1} k_0 x}{1 + 2 \ga} - 2^{n-1} k_0 x} \\
        &= 2 \exp\pr{-\frac{\ga k_0 }{1 + 2 \ga} x 2^{n}}
        \le 2 \exp\pr{-\frac{\ga k_0 }{1 + 2 \ga} x^{1 + \frac{\log 2}{\log\pr{1 + \ga}} }} \\
        &= 2 \exp\pr{-\frac{\ga k_0 }{1 + 2 \ga} x^{\frac 1 \eps}}.
    \end{align*}
    It follows that for any $x \ge 1$,
    \begin{align*}
        \abs{u(\te, \vp, x)} 
        \le 2 \exp\pr{-C_\eps x^{\frac 1 \eps}},
    \end{align*}
    showing that the second part of \eqref{eqn:egest4} also holds for any $\eps \in (0, \tfrac 1 2]$.
\end{proof}

It is worth emphasizing again that the second set of examples provided by Theorem \ref{thm:eg_1} (i.e., those with the bounds described by \eqref{eqn:egest2}) show that in Theorem \ref{thm:tau=1}, the relationship between the constant $L_2$ and the power $p = 1 + \tilde{c}_1 L_2$ is in some sense optimal. 
Indeed, these examples show that there must be a linear dependence like $p(L_2) \ge c L_2$ as $L_2 \ra \infty$ in the conclusion of the theorem. 
On the other hand, for the small $L_2$ regime, it is unlikely that any more can be said. 
In particular, even for $A \equiv I$ in $\T^2 \times \R$, the solution $v(\te, \vp, x) = \cos(\te) e^{-x}$ is an entire solution with exponential decay (growth) in the $x$($-x$)-direction, so we cannot expect solutions to be better behaved than those for constant-coefficient operators.

\vspace{1em}

We complete this section by filling in the details of the ``building block'' transformation. 
The proof we provide below expedites and simplifies some of the arguments of \cite{KLP25}.

\begin{proof}[Proof of Lemma \ref{lemma:transformation_block}]
    Fix $T > 0$ and set $J = [0, T]$.
	For $i = 1, \ldots, 5$, define the subintervals $J_i = \brac{ \tfrac{(i-1)}5 T, \tfrac{i}5 T}$ so that $\disp J = \bigcup_{i = 1}^5 J_i$.

	Define $\eta : \R \to [0, 1]$ to be a smooth cutoff function with the property that $\eta(x) = 0$ for $x \le \tfrac 1 4$, $\eta(x) = 1$ for $x \ge \tfrac 3 4$, and $\eta^{\prime}$, $\eta^{\prime\prime}$, and $\eta^{\prime\prime\prime}$ are supported in $[\tfrac 1 4, \tfrac 3 4]$ with $\abs{\eta^{\prime}} \le 4$, $\abs{\eta^{\prime\prime}} \lesssim 1$, and $\abs{\eta^{\prime\prime\prime}} \lesssim 1$ there.
    Let 
    \begin{align*}
        &\eta_1(x) := \eta\pr{\tfrac {5} T x}, \qquad\qquad\quad
        \eta_2(x) := \eta\pr{\tfrac {5x} T -1}, \\
        &\eta_3(x) := 1 - \eta\pr{\tfrac {5x} T -2}, \quad
        \eta_4(x) := \eta\pr{\tfrac {5x} T - 3}.
    \end{align*}
    Note that for each $i$, $\eta^{\prime}_i$, $\eta^{\prime\prime}_i$, and $\eta^{\prime\prime\prime}_i$ are supported on $J_i$ with $\abs{\eta^{\prime}_i} \lesssim T^{-1}$, $\abs{\eta^{\prime\prime}_i} \lesssim T^{-2}$, and $\abs{\eta^{\prime\prime\prime}_i} \lesssim T^{-3}$.

	Fix $k \in \N$ with $k T \gg 1$, then let  
    \begin{equation}
    \label{eqn:uiDefns}
        \begin{aligned}
            & u_1(\te, \vp, x) = \cos(k \te) e^{- k x} \\
		& u_2(\te, \vp, x) = \cos(2k \vp) e^{- k x} \\
		& u_3(\te, \vp, x) = \cos(2k \vp) e^{- 2k x}.
        \end{aligned}
    \end{equation}
	Observe that $u_1$ and $u_3$ satisfy $\LP u = 0$ while $u_1$ and $u_2$ both satisfy $\di \pr{A_0 \gr u} = 0$, where we introduce $A_0 = \begin{bmatrix} 1 & 0 & 0 \\ 0 & \frac 1 4 & 0 \\ 0 & 0 & 1 \end{bmatrix}$.
	
	Let $u : \T^2 \times J \to \R$ be given by
	\begin{equation}
		\label{uDefn}
		u(\te, \vp, x) = \eta_3(x) \cos(k \te) e^{- k x} + \eta_2(x) \cos(2k \vp) e^{- \brac{1 + \eta_4(x)}k x}    
	\end{equation}
	then note that 
	\begin{align*}
		u(\te, \vp, x)
		& = \begin{cases}
			\cos(k \te) e^{- k x} &\quad \text{for } x \in J_1 \\    
			\cos(k \te) e^{- k x} + \eta_2(x) \cos(2k \vp) e^{- k x} &\quad \text{for } x \in J_2 \\
			\eta_3(x) \cos(k \te) e^{- k x} + \cos(2k \vp) e^{- k x}  &\quad \text{for } x \in J_3 \\
			\cos(2k \vp) e^{- \brac{1 + \eta_4(x)}k x} &\quad \text{for } x \in J_4 \\
            \cos(2k \vp) e^{- 2k x} &\quad \text{for } x \in J_5
		\end{cases}.
	\end{align*}
	In fact, with the notation from \eqref{eqn:uiDefns}, we have
	\begin{align*}
		u(\te, \vp, x)
		& = \begin{cases}
			u_1(\te, \vp, x) &\quad \text{for } x \in \brac{0,  \tfrac{1}{4}T} \\    
			u_1(\te, \vp, x) + u_2(\te, \vp, x) &\quad \text{for } x \in \brac{\tfrac{7}{20}T, \tfrac{9}{20}T} \\
			u_2(\te, \vp, x) &\quad \text{for } x \in \brac{ \tfrac{11}{20}T, \tfrac{13}{20}T} \\
			u_3(\te, \vp, x) &\quad \text{for } x \in \brac{\tfrac{3}{4}T, T}
		\end{cases},
	\end{align*}
    which establishes \eqref{eqn:uPiecewise}.
	Since $\abs{u(\te, \vp, x)} \lesssim e^{-kx}$ for each $(\te, \vp, x) \in \T^2 \times J$, then the second part of \eqref{eqn:AuBlockBounds} holds.
	
    It remains to show that there exists a matrix-valued function $A \in \mathcal{M}_0(2, 1, 75)$ with $\abs{\gr A} \lesssim T^{-1}$ and $A \equiv I$ in $\disp \mathbb{T}^2 \times \pr{\brac{0, \tfrac T {20}} \cup  \brac{\tfrac{3 T}{4}, T}}$ that satisfies $\di\pr{A \gr u} = 0$ in $\T^2 \times J$.
	We construct $A$ in five steps corresponding to each subinterval $J_i$. \\
	
	\nid \textbf{Step 1:}
	Define $d_1 : J_1 \to \R$ by 
	$$d_1(x) = 1 + \eta_1(x) \pr{\frac 1 4 - 1}$$
	then set
	$$A_1(x) = \begin{bmatrix} 1 & 0 & 0 \\ 0 & d_1(x) & 0 \\ 0 & 0 & 1 \end{bmatrix}.$$
	Observe that $A_1 = I$ on on $\T^2 \times \brac{0, \tfrac{1}{20} T}$, $A_1 =  A_0$ on $\T^2 \times \brac{\frac{3}{20}T, \frac 1 5 T}$, $A_1 \in \mathcal{M}_0(2, 1, 4)$, and $\abs{\gr {A}_1}\lesssim T^{-1}$.
	Moreover, since $u(\te, \vp, x) = \cos(k \te) e^{- k x}$ on $\T^2 \times J_1$, then $u$ is a solution of $\di \pr{A_1 \gr u} = 0$ in $\T \times J_1$. \\
	
	\nid \textbf{Step 2:}
	Define $a_2, b_2 : \T^2 \times J_2 \to \R$ by
	\begin{align*}
		a_2(\te, \vp, x)
		&= 1 + \frac {2k \eta^{\prime}_2(x) - \eta^{\prime\prime}_2(x)} {k^2} \brac{2 \eta_2(x) \sin^2(2k \vp)  - \cos(k \te) \cos(2k\vp)} \\
		b_2(\te, \vp, x)
		&= - \frac {2k \eta^{\prime}_2(x) - \eta^{\prime\prime}_2(x)}{ k^2} \sin(k\te) \sin (2k\vp)
	\end{align*}
	and let
	$$A_2(\te, \vp, x) = \begin{bmatrix} a_2(\te, \vp, x) & b_2(\te, \vp, x) & 0 \\ b_2(\te, \vp, x) & \frac 1 4 & 0 \\ 0 & 0 & 1\end{bmatrix}.$$
	Notice that $A_2 = A_0$ on $\T^2 \times \pr{\brac{\frac 1 5 T, \frac{1}{4}T} \cup \brac{\frac{7}{20}T, \frac 2 5 T}}$.
    Since $\disp \abs{a_2 - 1} \lesssim \frac 1 {k T}$ and $\disp \abs{b_2} \lesssim \frac 1 {k T}$, then $A_2$ is bounded and elliptic with $\La = 5$ if we choose $k T \gg 1$.
    That is, $A_2 \in \mathcal{M}_0(2, 1, 5)$.
    Because
	\begin{align*}
		\del_\te a_2
		&= \pr{2 \eta^{\prime}_2(x)- \frac {\eta^{\prime\prime}_2(x)} {k}} \sin(k \te) \cos(2k\vp) \\
        \del_\vp a_2 
        &= 2\pr{2 \eta^{\prime}_2(x)- \frac {\eta^{\prime\prime}_2(x)} {k}} \brac{4 \eta_2(x) \cos(2k \vp)  + \cos(k \te)} \sin(2k\vp) \\
        \del_x a_2 
		&= \pr{\frac {2\eta^{\prime\prime}_2(x)}{k} - \frac{\eta^{\prime\prime\prime}_2(x)} {k^2}} \brac{2 \eta_2(x) \sin^2(2k \vp)  - \cos(k \te) \cos(2k\vp)} \\
        &\quad + 2\pr{\frac {2 \brac{\eta^{\prime}_2(x)}^2}{k} - \frac{\eta^{\prime}_2(x)\eta^{\prime\prime}_2(x)} {k^2}} \sin^2(2k \vp),
	\end{align*}
	then $\abs{\gr a_2} \lesssim T^{-1}\brac{1 + \pr{k T}^{-1} + \pr{k T}^{-2}} \lesssim T^{-1}$ in $\T^2 \times J_2$.
	As
	\begin{align*}		
        \del_\te b_2 
		&= \pr{-2 \eta^{\prime}_2(x) + \frac {\eta^{\prime\prime}_2(x)} {k}} \cos(k\te) \sin (2k\vp) \\
        \del_\vp b_2 
		&= 2\pr{-2 \eta^{\prime}_2(x) + \frac {\eta^{\prime\prime}_2(x)} {k}} \sin(k\te) \cos (2k\vp) \\
		\del_x b_2
		&= \pr{- \frac {2\eta^{\prime\prime}_2(x)}{ k} + \frac{\eta^{\prime\prime\prime}_2(x)}{ k^2}} \sin(k\te) \sin (2k\vp),
	\end{align*}
	then $\abs{\gr b_2} \lesssim T^{-1}$ in $\T^2 \times J_2$ as well, and we may conclude that $\abs{\gr A_2} \lesssim T^{-1}$ in $\T^2 \times J_2$. \\
	
	Since
	$$u(\te, \vp, x) = u_1(\te, \vp, x) + \eta_2(x) u_2(\te, \vp, x) = e^{- k x} \brac{\cos(k\te) + \eta_2(x) \cos(2k\vp)}$$
	in $\T^2 \times J_2$, then
	\begin{align*}
		\gr u &= - k e^{- k x} \pr{\sin(k\te), 2 \eta_2(x) \sin(2k \vp), \cos(k\te) + \frac{k\eta_2(x) - \eta^{\prime}_2(x)}{k}\cos(2k\vp)  }
	\end{align*}
	so that
	\begin{align*}
		e^{k x}  A_2 \gr u
		&= - k  \begin{bmatrix} a_2 & b_2 & 0 \\ b_2 & \frac 1 4 & 0 \\ 0 & 0 & 1 \end{bmatrix}
		\begin{bmatrix} \sin(k\te) \\ 2 \eta_2(x) \sin(2k \vp) \\ \cos(k\te) + \frac{k\eta_2(x) - \eta^{\prime}_2(x)}{k}\cos(2k\vp)  \end{bmatrix} \\
		&= - k \begin{bmatrix}
			a_2 \sin(k\te) + 2 b_2 \eta_2(x) \sin(2k \vp) \\  
			b_2 \sin(k\te) + \frac 1 2 \eta_2(x) \sin(2k \vp) \\
            \cos(k\te) + \frac{k\eta_2(x) - \eta^{\prime}_2(x)}{k}\cos(2k\vp) 
		\end{bmatrix} \\
		&= - k \begin{bmatrix}
			\brac{1 - \frac {2k \eta^{\prime}_2(x) - \eta^{\prime\prime}_2(x)} {k^2} \cos(k \te) \cos(2k\vp)}\sin(k\te) \\  
			\brac{\frac 1 2 \eta_2(x) - \frac {2k \eta^{\prime}_2(x) - \eta^{\prime\prime}_2(x)}{ k^2} \sin^2(k\te)} \sin (2k\vp) \\
            \cos(k\te) + \frac{k\eta_2(x) - \eta^{\prime}_2(x)}{k}\cos(2k\vp) 
		\end{bmatrix} .
	\end{align*}
	It follows that
	\begin{align*}
		e^{k x}  \di \pr{A_2 \gr u}
		&= - k 
		\del_\te \set{\brac{1 - \frac {2k \eta^{\prime}_2(x) - \eta^{\prime\prime}_2(x)} {k^2} \cos(k \te) \cos(2k\vp)}\sin(k\te)} \\
		&\quad - k \del_\vp \set{\brac{\frac 1 2 \eta_2(x) - \frac {2k \eta^{\prime}_2(x) - \eta^{\prime\prime}_2(x)}{ k^2} \sin^2(k\te)} \sin (2k\vp)} \\
        &\quad - k e^{k x} \del_x \set{e^{- k x} \brac{\cos(k\te) + \frac{k\eta_2(x) - \eta^{\prime}_2(x)}{k}\cos(2k\vp) }} \\
		&= - k^2 
		\brac{\cos(k\te) - \frac {2k \eta^{\prime}_2(x) - \eta^{\prime\prime}_2(x)} {k^2} \cos^2(k \te) \cos(2k\vp)} \\
		&\quad- k^2
		\frac {2k \eta^{\prime}_2(x) - \eta^{\prime\prime}_2(x)} {k^2} \sin^2(k \te) \cos(2k\vp) \\
		&\quad- 2 k^2 \brac{\frac 1 2 \eta_2(x) - \frac {2k \eta^{\prime}_2(x) - \eta^{\prime\prime}_2(x)}{ k^2} \sin^2(k\te)} \cos (2k\vp) \\
        &\quad+ k^2 \brac{\cos(k\te) + \frac{k\eta_2(x) - \eta^{\prime}_2(x)}{k}\cos(2k\vp) } \\
        &\quad - k \brac{ \frac{k\eta^{\prime}_2(x) - \eta^{\prime\prime}_2(x)}{k}}\cos(2k\vp) .
	\end{align*}
    After simplifying, we see that $\di \pr{A_2 \gr u} = 0$ in $\T^2 \times J_2$. \\

	\nid \textbf{Step 3:}
	Define $b_3, d_3 : \T^2 \times J_3 \to \R$ by
	\begin{align*}
		b_3(\te, \vp, x)
		&= - \frac {2k \eta^{\prime}_3(x) - \eta^{\prime\prime}_3(x)}{ k^2} \sin(k\te) \sin (2k\vp) \\
		d_3(\te, \vp, x)
		&= \frac 1 4
		- \frac {2k \eta^{\prime}_3(x) - \eta^{\prime\prime}_3(x)} {4k^2} \brac{\cos(k\te) \cos(2k\vp)
			- 2 \eta_3(x) \sin^2(k \te)}
	\end{align*}
	then set
	$$A_3(\te, \vp, x) = \begin{bmatrix} 1 & b_3(\te, \vp, x) & 0 \\ b_3(\te, \vp, x) & d_3(\te, \vp, x) & 0 \\ 0 & 0 & 1 \end{bmatrix}.$$
	We see that $A_3 = A_0$ on $\T^2 \times \pr{\brac{\frac 2 5 T, \frac{9}{20}T} \cup \brac{\frac{11}{20}T, \frac{3}{5}T}}$.
    Since $\disp \abs{b_3} \lesssim \frac 1 {k T}$ and $\disp \abs{d_3 - \tfrac 1 4} \lesssim \frac 1 {k T}$, then $A_3$ is bounded and elliptic with $\La = 5$ if we choose $k T \gg 1$.
    As
	\begin{align*}
		\del_\te b_3 
        &= \pr{- 2\eta^{\prime}_3(x) + \frac{\eta^{\prime\prime}_3(x)}{ k}} \cos(k\te) \sin (2k\vp) \\
        \del_\vp b_3 
        &= 2 \pr{- 2\eta^{\prime}_3(x) + \frac{\eta^{\prime\prime}_3(x)}{ k}} \sin(k\te) \cos (2k\vp)\\
		\del_x b_3
		&= \pr{- \frac {2 \eta^{\prime\prime}_3(x) }{ k} + \frac {\eta^{\prime\prime\prime}_3(x)}{ k^2}} \sin(k\te) \sin (2k\vp),
	\end{align*}
	then $\abs{\gr b_3} \lesssim T^{-1}$ in $\T^2 \times J_3$.
	Because
	\begin{align*}
        \del_\te d_3 
		&= \frac 1 4 \pr{2 \eta^{\prime}_3(x) - \frac {\eta^{\prime\prime}_3(x)} {k}} \brac{\sin(k\te) \cos(2k\vp) + 4\eta_3(x) \sin(k \te) \cos(k\te)} \\
        \del_\vp d_3 
		&= \frac 1 2 \pr{2 \eta^{\prime}_3(x) - \frac {\eta^{\prime\prime}_3(x)} {k}} \cos(k\te) \sin(2k\vp) \\
		\del_x d_3
		&= \pr{- \frac {\eta^{\prime\prime}_3(x) } {2k} + \frac { \eta^{\prime\prime\prime}_3(x)} {4k^2}}\brac{\cos(k\te) \cos(2k\vp)
			- 2 \eta_3(x) \sin^2(k \te)} \\
		&\quad+ \pr{\frac {\brac{\eta^{\prime}_3(x)}^2} {k} - \frac {\eta^{\prime}_3(x)\eta^{\prime\prime}_3(x)} {2k^2}}
		\sin^2(k \te)
	\end{align*}
	then $\abs{\gr d_3} \lesssim T^{-1}$ in $\T^2 \times J_3$ as well.
    Therefore, $\abs{\gr A_3} \lesssim T^{-1}$  in $\T^2 \times J_3$. \\
	
	Since
	$$u(\te, \vp, x) = \eta_3(x) u_1(\te, \vp, x) + u_2(\te, \vp, x) = e^{- k x} \brac{\eta_3(x) \cos(k\te) + \cos(2k\vp)}$$
    in $\T^2 \times J_3$, then
	\begin{align*}
		\gr u &= - k e^{- k x} \pr{\eta_3(x) \sin(k\te), 2 \sin(2k \vp), \frac{k\eta_3(x) - \eta^{\prime}_3(x)}{k} \cos(k\te) + \cos(2k\vp)  }
	\end{align*}
	so that
	\begin{align*}
		e^{k x} A_3 \gr u
		&= - k \begin{bmatrix} 1 & b_3 & 0 \\ b_3 & d_3 & 0 \\ 0 & 0 & 1 \end{bmatrix}
		\begin{bmatrix} \eta_3(x) \sin(k\te) \\ 2 \sin(2k \vp) \\ \frac{k\eta_3(x) - \eta^{\prime}_3(x)}{k} \cos(k\te) + \cos(2k\vp) \end{bmatrix} \\
		&= - k \begin{bmatrix}
			\eta_3(x) \sin(k\te) + 2 b_3 \sin(2k \vp) \\  
			b_3 \eta_3(x) \sin(k\te) + 2 d_3 \sin(2k \vp) \\
            \frac{k\eta_3(x) - \eta^{\prime}_3(x)}{k} \cos(k\te) + \cos(2k\vp)
		\end{bmatrix} \\
		&= - k \begin{bmatrix}
			\brac{\eta_3(x) - 2\frac {2k \eta^{\prime}_3(x) - \eta^{\prime\prime}_3(x)}{ k^2} \sin^2(2k\vp)} \sin(k\te) \\  
			\frac 1 2 \brac{1
			- \frac {2k \eta^{\prime}_3(x) - \eta^{\prime\prime}_3(x)} {k^2} \cos(k\te) \cos(2k\vp) } \sin(2k \vp) \\
            \frac{k\eta_3(x) - \eta^{\prime}_3(x)}{k} \cos(k\te) + \cos(2k\vp)
		\end{bmatrix}. 
	\end{align*}
	In $\T^2 \times J_3$, we get that
	\begin{align*}
		e^{k x} \di\pr{A_3 \gr u}
		&= - k \del_\te \set{\brac{\eta_3(x) - 2\frac {2k \eta^{\prime}_3(x) - \eta^{\prime\prime}_3(x)}{ k^2} \sin^2(2k\vp)} \sin(k\te)} \\
		&\quad- \frac 1 2 k \del_\vp \set{\sin(2k \vp)
			- \frac {2k \eta^{\prime}_3(x) - \eta^{\prime\prime}_3(x)} {k^2} \cos(k\te) \cos(2k\vp) \sin(2k \vp)} \\
        &\quad- k e^{k x} \del_x \set{e^{- k x}\brac{\frac{k\eta_3(x) - \eta^{\prime}_3(x)}{k} \cos(k\te) + \cos(2k\vp)}} \\
		&= - k^2 \brac{\eta_3(x) \cos(k\te)  + \cos(2k \vp)}
		+ \brac{2k \eta^{\prime}_3(x) - \eta^{\prime\prime}_3(x)} \cos(k\te) \\
        &\quad+ k^2\brac{\frac{k\eta_3(x) - \eta^{\prime}_3(x)}{k} \cos(k\te) + \cos(2k\vp)} 
        - k \brac{\frac{k\eta^{\prime}_3(x) - \eta^{\prime\prime}_3(x)}{k}  }\cos(k\te) \\
		&= 0.
	\end{align*}
  	
 	\nid \textbf{Step 4:}
    Let $\psi_4(x) = \brac{1 + \eta_4(x)} x$ so that 
    \begin{align*}
        \psi_4^{\prime}(x) &= 1 +  \eta_4(x) + x \eta^{\prime}_4(x) \\
        \psi_4^{\prime\prime}(x) &= 2\eta^{\prime}_4(x) + x \eta^{\prime\prime}_4(x) \\
        \psi_4^{\prime\prime\prime}(x) &= 3\eta^{\prime\prime}_4(x) + x \eta^{\prime\prime\prime}_4(x).
    \end{align*}
	Define $d_4 : J_4 \to \R$ by
	$$d_4(x) = \frac 1 4 \set{\brac{1 +  \eta_4(x)
			+ x \eta^{\prime}_4(x)}^2- \frac{2 \eta^{\prime}_4(x) + x \eta^{\prime\prime}_4(x)} {k}}
            = \frac 1 4 \set{\brac{\psi_4^{\prime}(x)}^2 - \frac{\psi_4^{\prime\prime}(x)}{k}}$$
	then set
	$$A_4(x) = \begin{bmatrix} 1 & 0 & 0 \\ 0 & d_4(x) & 0 \\ 0 & 0 & 1 \end{bmatrix}.$$ 
	We see that $A_4 = A_0$ on $\T^2 \times \brac{\frac 3 5 T, \frac{13}{20}T}$ and $A_4 = I$ on $\T^2 \times \brac{\frac{3}{4}T, \frac 4 5 T}$.
    For $x \in \brac{\frac{13}{20}T, \frac{15}{20}T}$, we have
    \begin{align*}
        1 \le 1 + \eta_4(x) + x \eta^{\prime}_4(x)
        &\le 1 + 1 + \abs{x} \frac 5 T \abs{\eta'}
        \le 17.
    \end{align*}
    Then, for $kT \gg 1$, we get $\frac 1 {5} \le \abs{d_4(x)} \le 75$ so that $A_4$ is bounded and elliptic with $\La = 75$.
	As
	\begin{align*}
		d_4^{\prime}(x)
		&= \frac 1 4 \set{2\psi_4^{\prime}(x)\psi_4^{\prime\prime}(x) - \frac{\psi_4^{\prime\prime\prime}(x)}{k}},
	\end{align*}
	then $\abs{\gr A_4} \lesssim T^{-1}$ in $\T^2 \times J_4$.
	
	Since $u(\te, \vp, x) = \cos(2k \vp) e^{- \brac{1 + \eta_4(x)}k x} = \cos(2k \vp) e^{- k \psi_4(x)}$ on $\disp \T^2 \times J_4$, then
    \begin{align*}
        \gr u &= - k e^{- k \psi_4(x)} \pr{0, 2\sin(2k \vp), \psi_4^{\prime}(x) \cos(2k \vp) }
    \end{align*}
    so that in $\T^2 \times J_4$,
	\begin{align*}
		\di\pr{A_4 \gr u}
		&= \del_\vp\set{ - 2k d_4(x) \sin(2k \vp) e^{-k \psi_4(x)}} 
        + \del_x \set{- k e^{- k \psi_4(x)}\psi_4^{\prime}(x) \cos(2k \vp)} \\
		&= - k^2 \set{\brac{\psi_4^{\prime}(x)}^2 - \frac{\psi_4^{\prime\prime}(x)}{k}} u
        + k^2 \brac{\psi_4^{\prime}(x)}^2 u
		- k\psi_4^{\prime\prime}(x) u
        = 0.
	\end{align*}
	
	\nid \textbf{Step 5:} On $\T^2 \times J_5$, set $A_5 = I$ so that $\LP u = 0$ there. \\
	
	In conclusion, if we define $A : \T^2 \times J \to \mathbb{M}_3$ by $A(\te, \vp, x) = A_i(\te, \vp, x)$ whenever $x \in J_i$, then we see that $A \in \mathcal{M}(2, 1, 75)$ with $\abs{\gr A} \lesssim T^{-1}$ in $\T^2 \times J$.
	Moreover, with $u$ as defined in \eqref{uDefn}, we have that $\di \pr{A \gr u} = 0$ in $\T^2 \times J$.	
\end{proof}

\bibliographystyle{alpha}
\bibliography{refs}

\end{document}